\documentclass[english]{article}
\usepackage[T1]{fontenc}
\usepackage[latin9]{inputenc}
\usepackage{geometry}
\usepackage{color}
\usepackage{amsmath}
\usepackage{amsthm}
\usepackage{amssymb}
\usepackage{setspace}

\makeatletter
\numberwithin{equation}{section}
\numberwithin{figure}{section}
\theoremstyle{plain}
\newtheorem{thm}{\protect\theoremname}
  \theoremstyle{remark}
  \newtheorem{rem}[thm]{\protect\remarkname}
  \theoremstyle{plain}
  \newtheorem{cor}[thm]{\protect\corollaryname}
  \theoremstyle{remark}
  \newtheorem*{acknowledgement*}{\protect\acknowledgementname}
  \theoremstyle{plain}
  \newtheorem{prop}[thm]{\protect\propositionname}
  \theoremstyle{definition}
  \newtheorem{defn}[thm]{\protect\definitionname}
  \theoremstyle{plain}
  \newtheorem{lem}[thm]{\protect\lemmaname}

\usepackage{import}  
\usepackage{graphicx}
\usepackage{xcolor}
\usepackage{calc}
\usepackage{amssymb}

\usepackage{ccfonts,eulervm}
\usepackage[T1]{fontenc}

\usepackage{dsfont}

\usepackage{hyperref}

\newcommand{\xdownarrow}[1]{%
  {\left\downarrow\vbox to #1{}\right.\kern-\nulldelimiterspace}
}

\usepackage{tocbibind}
\usepackage{ dsfont }

\usepackage{tikz}

\usepackage{babel}

\providecommand{\acknowledgementname}{Acknowledgement}
  \providecommand{\corollaryname}{Corollary}
  \providecommand{\definitionname}{Definition}
  \providecommand{\lemmaname}{Lemma}
  \providecommand{\propositionname}{Proposition}
  \providecommand{\remarkname}{Remark}
\providecommand{\theoremname}{Theorem}

\makeatother

\usepackage{babel}
  \providecommand{\acknowledgementname}{Acknowledgement}
  \providecommand{\corollaryname}{Corollary}
  \providecommand{\definitionname}{Definition}
  \providecommand{\lemmaname}{Lemma}
  \providecommand{\propositionname}{Proposition}
  \providecommand{\remarkname}{Remark}
\providecommand{\theoremname}{Theorem}

\begin{document}

\title{Asymptotic Behavior for $\mathcal{H}-$holomorphic Cylinders of Small
Area}

\author{Alexandru Doicu, Urs Fuchs}
\maketitle
\begin{abstract}
\begin{singlespace}
\noindent $\mathcal{H}-$holomorphic curves are solutions of a modified
pseudoholomorphic curve equation involving a harmonic $1-$form as
perturbation term. Following \cite{key-11} we establish an asymptotic
behavior of a sequence of finite energy $\mathcal{H}-$holomorphic
cylinders with small $d\alpha-$energies. Our results can be seen
as a first step toward establishing the compactness of the moduli
space of $\mathcal{H}-$holomorphic curves, which in turns, due to
the program initiated in \cite{key-4}, can be used for proving the
generalized Weinstein conjecture. 
\end{singlespace}
\end{abstract}
\begin{singlespace}
\noindent \tableofcontents{}
\end{singlespace}

\section{\label{sec:Introduction}Introduction}

\begin{singlespace}
\noindent Let $(M,\alpha)$ be a closed, contact $3-$dimensional
manifold, $\xi=\ker(\alpha)$ the contact structure, and $X_{\alpha}$
the Reeb vector field, defined by $\iota_{X_{\alpha}}\alpha\equiv1$
and $\iota_{X_{\alpha}}d\alpha\equiv0$. Denote by $\pi_{\alpha}:TM\rightarrow\xi$
the projection along the Reeb vector field. By definition, it is obvious
that the Reeb vector field is transverse to the contact structure
so that the tangent bundle splits as $TM=\mathbb{R}X_{\alpha}\oplus\xi$.
Note that $(\xi,d\alpha)$ is a symplectic vector bundle. We choose
a complex structure $J_{\xi}$ on $\xi$ compatible to $d\alpha$, i.e. such that $g_{J_{\xi}}(\cdot,\cdot)=d\alpha(\cdot,J_{\xi}\cdot)$
defines a smooth fiberwise metric on the vector bundle $\xi\rightarrow M$.
The complex structure $J_{\xi}$ on $\xi$ extends to a complex structure
$J$ on $T(\mathbb{R}\times M)$, defined by 
\begin{equation}
J(v,w)=(-\alpha(v),J_{\xi}\circ\pi_{\alpha}w+vX_{\alpha})\label{eq:JJJ}
\end{equation}
for all $(v,w)\in T(\mathbb{R}\times M)\cong T\mathbb{R}\oplus TM$.
This $J$ defines thus an  $\mathbb{R}-$invariant almost complex structure on $\mathbb{R}\times M$. Using the splitting of $TM$, the fibrewise metric $g_{J_\xi}$ on $\xi$
extends to a $\mathbb{R}$-invariant Riemannian metric $g$ on $\mathbb{R}\times M$ given
by $g=dr\otimes dr+\alpha\otimes \alpha+g_{J_\xi}$; here $r$ is the $\mathbb{R}-$coordinate
on $\mathbb{R}\times M$. As in \cite{key-11}, let $\mathcal{P}\subset\mathbb{R}$
be the set 
\[
\mathcal{P}=\{0\}\cup\{T>0\mid\text{there exists a }T-\text{periodic orbit of }X_{\alpha}\}.
\]
In the following we assume that all periodic orbits of the Reeb vector
field are non-degenerate, i.e. the linearization of the flow of $X_{\alpha}$,
restricted to the contact structure along the periodic orbit, does
not contain $1$ in its spectrum. A non-degenerate $T-$periodic orbit
is isolated among the periodic orbits with period close to $T$ \cite{key-6}.
Due to the non-degeneracy condition, we define for each $E_{0}>0$,
\begin{equation}
\hbar_{E_{0}}=\min\{|T_{1}-T_{2}|\mid T_{1},T_{2}\in\mathcal{P},T_{1},T_{2}\leq E_{0},T_{1}\not=T_{2}\}.\label{eq:Def_hbar}
\end{equation}
Note that $\hbar_{E_{0}}$ depends on $E_{0}$ and is positive. If
$E_{1}\geq E_{2}$, then $\hbar_{E_{1}}\leq\hbar_{E_{2}}$.

\noindent A pseudoholomorphic curve (or $J-$holomorphic curve) is
a smooth map $u=(a,f):S\rightarrow\mathbb{R}\times M$, where $(S,j)$
is a Riemann surface such that 
\begin{equation}
J(u)\circ du=du\circ j.\label{eq:pseudohol_eq_compact}
\end{equation}
Projecting the above equation to the contact structure and to the
Reeb direction equation (\ref{eq:pseudohol_eq_compact}) yields 
\begin{equation}
\begin{array}{rl}
\pi_{\alpha}df\circ j & =J_{\xi}(u)\circ\pi_{\alpha}df,\\
f^{*}\alpha\circ j & =da.
\end{array}\label{eq:split_of_pseudoholomorphic_curve}
\end{equation}
Further on, with 
\[
E_{\text{H}}(u;S)=\sup_{\varphi\in\mathcal{A}}\int_{S}u^{*}d(\varphi\alpha)
\]
being the Hofer energy, we assume the finite Hofer energy condition
$E_{\text{H}}(u;S)<+\infty$. Here $\mathcal{A}$ is the set of smooth
functions $\varphi:\mathbb{R}\rightarrow[0,1]$ with $\varphi'(r)\geq0$.
By a local computation it can be easily checked that the integrand
of the Hofer energy is non-negative.

\noindent In \cite{key-11}, Hofer, Wysocki and Zehnder studied pseudoholomorphic
cylinders $u=(a,f):[-R,R]\times S^{1}\rightarrow\mathbb{R}\times M$
with a finite Hofer energy and a small $d\alpha-$energy, defined
by 
\[
E_{d\alpha}(u;[-R,R]\times S^{1})=\int_{[-R,R]\times S^{1}}f^{*}d\alpha.
\]
It has been shown that for sufficiently large $R>0$ and a sufficiently
small $d\alpha-$energy, the curve $u$ is either close to a point
in $\mathbb{R}\times M$, in the case of vanishing center action,
or close to a cylinder over a periodic orbit, in the case of non-zero
center action. In the later case, precise estimates by determining
the shape of the cylinder have been derived. It should be pointed
out that these results have been used in the proof of the SFT compactness
theorem \cite{key-1}, \cite{key-10}.

\noindent In this paper we will derive similar results for $\mathcal{H}-$holomorphic
cylinders and establish a notion of convergence under certain conditions.
A smooth map $u=(a,f):S\rightarrow\mathbb{R}\times M$, defined on
a Riemann surface $(S,j)$, together with a harmonic $1-$form $\gamma$
on $S$ (i.e., $\gamma$ satisfies $d\gamma=d(\gamma\circ j)=0$)
is called a $\mathcal{H}-$holomorphic curve if 
\begin{equation}
\begin{array}{cl}
\pi_{\alpha}df\circ j & =J_{\xi}(u)\circ\pi_{\alpha}df,\\
f^{*}\alpha\circ j & =da+\gamma.
\end{array}\label{eq:H_hol}
\end{equation}
The $\mathcal{H}-$holomorphic curve equation differs from the pseudoholomorphic
curve equation due to the presence of the harmonic $1-$form $\gamma$
in the second equation of (\ref{eq:H_hol}).
Because of this perturbation, the Hofer energy of such a map may have
a somewhere negative integrand. To overcome this drawback, we define the \emph{$\alpha-$energy} resp. the \emph{$d\alpha-$energy}
of a $\mathcal{H}$-holomorphic curve $u$ on $S$ as
\[
E_\alpha(u;S)=\sup_{\varphi\in\mathcal{A}}\int_{S}\varphi'(a)da\circ j\wedge da\quad\text{ and } \quad E_{d\alpha}(u;S)=\int_{S}f^{*}d\alpha,
\]
and finally the \emph{energy of $u$} as $E(u;S)=E_\alpha(u;S)+E_{d\alpha}(u;S)$. If the perturbation $1-$form $\gamma$
vanishes, then the energy and the Hofer energy of $u$ mutually bound each other.
As an additional condition we require $u$ to have finite energy,
i.e. $E(u;S)<+\infty$. This modification of the pseudoholomorphic
curve equation was first suggested by Hofer in \cite{key-3} and used
extensively in the program initiated by Abbas et al. \cite{key-4}
to prove the generalized Weinstein conjecture in dimension three.
However, due to lack of a compactness result of the moduli space
of $\mathcal{H}-$holomorphic curves, the generalized Weinstein conjecture
has been proved only in the planar case, i.e. when the leaves of the
holomorphic open book decomposition \cite{key-18} have genus zero.
In this regard, our results can be seen as a first step towards establishing
compactness of the moduli space of $\mathcal{H}-$holomorphic
curves.

\noindent The $L^{2}-$norm of the harmonic perturbation $1-$form
$\gamma$ is defined as 
\[
\left\Vert \gamma\right\Vert _{L^{2}(S)}^{2}=\int_{S}\gamma\circ j\wedge\gamma.
\]

\noindent In the following, we consider finite energy $\mathcal{H}-$holomorphic
cylinders $u=(a,f):[-R,R]\times S^{1}\rightarrow\mathbb{R}\times M$
with harmonic perturbation $\gamma$. We define the period of $\gamma$
over the cylinder as 
\begin{equation}
P(\gamma)=\int_{\{0\}\times S^{1}}\gamma\label{eq:period_over_the_cylinder}
\end{equation}
and the co-period by 
\[
S(\gamma)=\int_{\{0\}\times S^{1}}\gamma\circ i.
\]
Furthermore, the conformal period is defined as $\tau=P(\gamma)R$
while the conformal co-period is defined by $\sigma=S(\gamma)R$.

\noindent The goal of our analysis is to establish the asymptotic
behaviour of finite energy $\mathcal{H}-$holomorphic cylinders with
a uniformly small $d\alpha-$energy and harmonic perturbation $1-$forms
having uniformly bounded $L^{2}-$norms and uniformly bounded conformal
periods and co-periods. \\
More specifically, for constants $E_{0},C_{0},C_{1},\delta_{1}$
and $C>0$ we consider for a $\mathcal{H}$-holomorphic cylinder
$u=(a,f):[-R,R]\times S^{1}\rightarrow\mathbb{R}\times M$ with harmonic
perturbation $\gamma$ the following conditions:
\end{singlespace}
\begin{description}
\item [{A0}] $E(u,[-R,R]\times S^{1})\leq E_{0}$ and $\left\Vert \gamma\right\Vert _{L^{2}([-R,R]\times S^{1})}^{2}\leq C_{0}$.
\item [{A1}] $\left\Vert df(z)\right\Vert :=\sup_{\left\Vert v\right\Vert _{\text{eucl.}}=1}\left\Vert df(z)v\right\Vert _{g}\leq C_{1}$
for all $z\in([-R,-R+\delta_{1}]\amalg[R-\delta_{1},R])\times S^{1}$.
\item [{A2}] $E_{d\alpha}(u,[-R,R]\times S^{1})\leq\hbar/2$, where $\hbar:=\hbar_{\tilde{E}_{0}}$
in the sense of (\ref{eq:Def_hbar}) for $\tilde{E}_{0}:=2C_{1}+E_{0}$.
\item [{A3}] $|\tau|,|\sigma|\leq C$, where $\tau$ (resp. $\sigma$)
is the conformal (co-)period of $\gamma$ on $[-R,R]\times S^{1}$.
\end{description}
For the rest of paper, we fix $E_{0},C_{0},C_{1},\delta_{1}$
and $C>0$. When we say that \emph{a $\mathcal{H}$-holomorphic cylinder $u$ satisfies $A0$-$A3$}, 
it always means that \emph{the $\mathcal{H}$-holomorphic cylinder $u$ with harmonic perturbation $\gamma$ satisfies $A0$-$A3$ 
for the fixed constants $E_{0},C_{0},C_{1},\delta_{1}$ and $C>0$}.

\begin{singlespace}
\noindent The task is to describe the asymptotic behavior of $\mathcal{H}$-holomorphic cylinders
satisfying $A0$-$A3$, whose domain cylinders $[-R,R]\times S^1$ have arbitrary large conformal modulus (that is, we consider families
where $R>0$ is arbitrary large). 
More precisely, we derive the following results. For a $\mathcal{H}-$holomorphic
cylinder $u=(a,f):[-R,R]\times S^{1}\rightarrow\mathbb{R}\times M$
satisfying $A0$-$A3$ with $R$ sufficiently large, we consider its center
action similarly as in \cite{key-11}. The center action of $u$ is defined as the unique element
$A(u)\in\mathcal{P}$ which is sufficiently close to 
\[
\left|\int_{S^{1}}u(0)^{*}\alpha\right|.
\]
For more details the reader might consult Section \ref{subsec:Center-action}.
It follows from Theorem \ref{thm:For-all-numbers} that either $A(u)=
0$ or $A(u)\geq \hbar$. For a sequence $u_n$ of $\mathcal{H}$-holomorphic cylinders $u_n$ satisfying $A0$-$A3$ we distinguish between two
cases: the first case is when there exists a subsequence of $u_{n}$
with vanishing center action and the second case is when there is
no subsequence of $u_{n}$ with this property. In this regard, Theorem
\ref{thm:With-the-same} deals with the asymptotic behavior in the
case of vanishing center action, while Theorem \ref{thm:With-the-same-1}
deals with the asymptotic behavior in the case of positive center
action.

\noindent Before stating the main result we construct a sequence of
diffeomorphisms $\theta_{n}:[-R_{n},R_{n}]\times S^1\rightarrow[-1,1]\times S^1$ with
certain properties. We first construct diffeomorphisms $\check{\theta}_n:[-R_n,R_n]\rightarrow [-1,1]$, the $\theta_n$ are then
obtained as $\check{\theta}_n\times id_{S^1}$.
The construction is similar to that given in \cite{key-10}
and will enable us to describe the $C^{0}-$convergence. 
\end{singlespace}
\begin{rem}
\begin{singlespace}
\noindent \label{rem:For-every-sequence}For all sequences $R_n, h_{n}\in\mathbb{R}_{>0}$
with $h_{n}<R_{n}$ and $h_{n},R_{n}/h_{n}\rightarrow\infty$
as $n\rightarrow\infty$, consider a sequence of diffeomorphisms $\check{\theta}_{n}:[-R_{n},R_{n}]\rightarrow[-1,1]$
with the following properties: 
\end{singlespace}
\begin{enumerate}
\begin{singlespace}
\item The left and right shifts $\check{\theta}_{n}^{\pm}(s):=\check{\theta}_{n}(s\pm R_{n})$ restrict to maps 
$\check{\theta}_{n}^+:[0,h_{n}]\rightarrow[-1,-1/2]$ resp. $\check{\theta}_{n}^-:[-h_{n},0]\rightarrow[1/2,1]$,
which converge in $C_{\text{loc}}^{\infty}$ to diffeomorphisms
$\check{\theta}^{-}:[0,\infty)\rightarrow[-1,-1/2)$ and $\check{\theta}^{+}:(-\infty,0]\rightarrow(1/2,1]$,
respectively. 
\item $\check{\theta}_{n}$ is a linear
diffeomorphism on $[-R_{n}+h_{n},R_{n}-h_{n}]$. More precisely, we require
\begin{align*}
\check{\theta}_{n}:\text{Op}([-R_{n}+h_{n},R_{n}-h_{n}]) & \rightarrow\text{Op}\left(\left[-\frac{1}{2},\frac{1}{2}\right]\right)\\
s & \mapsto\frac{s}{2(R_{n}-h_{n})},
\end{align*}
where $\text{Op}([-R_{n}+h_{n},R_{n}-h_{n}])$ and $\text{Op}([-1/2,1/2])$
are sufficiently small neighborhoods of the intervals $[-R_{n}+h_{n},R_{n}-h_{n}]$
and $[-1/2,1/2]$, respectively. 
\item These maps $\check{\theta}_n$ give rise to the desired diffeomorphisms $\theta_n:=\check{\theta}_n\times id_{S^1}:[-R_{n},R_{n}]\times S^1\rightarrow[-1,1]\times S^1$.
Similarly, we define $\theta_n^\pm:=\check{\theta}^\pm_n\times id_{S^1}$ and $\theta^\pm:=\check{\theta}^\pm\times id_{S^1}$.
\end{singlespace}
\end{enumerate}
\end{rem}

\begin{thm}
\begin{singlespace}
\noindent \label{thm:With-the-same}Let $u_{n}:[-R_{n},R_{n}]\times S^{1}\rightarrow\mathbb{R}\times M$
be a sequence of $\mathcal{H}-$holomorphic cylinders with harmonic
perturbations $\gamma_{n}$ satisfying $A0$-$A3$. Assume that $R_{n}\rightarrow\infty$
and that each $u_{n}$ has vanishing center action.\\
Then there exists a subsequence of $u_{n}$ (for each
$n$ suitably shifted in the $\mathbb{R}-$coordinate), still denoted
by $u_{n}$, $\mathcal{H}-$holomorphic cylinders $u^{\pm}$ with exact 
harmonic perturbations $d\Gamma^\pm$ defined
on $(-\infty,0]\times S^{1}$ and $[0,\infty)\times S^{1}$ respectively,
$\sigma,\tau\in\mathbb{R}$ and a point $w_{f}\in M$
such that for every sequence $h_{n}\in\mathbb{R}_{>0}$ with $h_n,R_n/h_n\rightarrow \infty$ the following $C_{\text{loc}}^{\infty}-$
and $C^{0}-$convergence results hold: $\sigma_{n}:=S_{n}R_{n}\rightarrow\sigma,\tau_{n}:=P_{n}R_{n}\rightarrow\tau$ and 

\noindent $C_{\text{loc}}^{\infty}-$convergence: 
\end{singlespace}
\begin{enumerate}
\begin{singlespace}
\item 
For any sequence $s_{n}\in[-R_{n}+h_{n},R_{n}-h_{n}]$ for which $s_{n}/R_{n}$
converges to $\kappa\in[-1,1]$, the shifted maps $u_{n}(s+s_{n},t)$,
defined on $[-R_{n}+h_{n}-s_{n},R_{n}-h_{n}-s_{n}]\times S^{1}$,
converge in $C_{\text{loc}}^{\infty}$ on $\mathbb{R}\times S^1$ to the constant map $(-\sigma \kappa,\phi_{-\tau\kappa}^{\alpha}(w_{f}))$.
The shifted harmonic $1-$forms $\gamma_{n}(s+s_{n},t)$
converge in $C_{\text{loc}}^{\infty}$ to $0$. 
\item The left shifts $u_{n}^{-}(s,t):=u_{n}(s-R_{n},t)$,
defined on $[0,h_{n})\times S^{1}$, converge in $C_{\text{loc}}^{\infty}$
to $u^{-}=(a^{-},f^{-})$ defined
on $[0,+\infty)\times S^{1}$ and $\lim_{s\rightarrow \infty}u^{-}(s,\cdot)=(\sigma,\phi_{\tau}^{\alpha}(w_{f}))$ in $C^\infty(S^1,\mathbb{R}\times M)$. 
The left shifted harmonic
$1-$forms $\gamma_{n}^{-}$ converge in $C_{\text{loc}}^{\infty}$
to $d\Gamma^{-}$ defined on $[0,+\infty)\times S^{1}$ with $\|d\Gamma^-\|^2_{L^2([0,\infty)\times S^1)}\leq C_0$. 
\item The right shifts $u_{n}^{+}(s,t):=u_{n}(s+R_{n},t)$
defined on $(-h_{n},0]\times S^{1}$, converge in $C_{\text{loc}}^{\infty}$
to $u^{+}=(a^{+},f^{+})$, defined
on $(-\infty,0]\times S^{1}$ and $\lim_{s\rightarrow -\infty} u^{+}(s,\cdot)=(-\sigma,\phi_{-\tau}^{\alpha}(w_{f}))$ in $C^\infty(S^1,\mathbb{R}\times S^1)$. The right shifted harmonic
$1-$forms $\gamma_{n}^{+}$ converge in $C_{\text{loc}}^{\infty}$
to $d\Gamma^{+}$ defined on $(-\infty,0]\times S^{1}$ with $\|d\Gamma^+\|^2_{L^2((-\infty,0]\times S^1)}\leq C_0$.
\end{singlespace}
\end{enumerate}
\begin{singlespace}
\noindent $C^{0}-$convergence:  For every sequence
of diffeomorphisms $\theta_{n}:[-R_{n},R_{n}]\times S^1\rightarrow[-1,1] \times S^1$
as in Remark \ref{rem:For-every-sequence}
\end{singlespace}
\begin{enumerate}
\begin{singlespace}
\item The maps $v_{n}:[-1/2,1/2]\times S^{1}\rightarrow\mathbb{R}\times M$
defined by $v_{n}=u_{n}\circ \theta_{n}^{-1}$, converge in
$C^{0}$ to the map \\$(s,t)\mapsto (-2\sigma s,\phi_{-2\tau s}^{\alpha}(w_{f}))$. 
\item The maps $v_{n}^{-}:[-1,-1/2]\times S^{1}\rightarrow\mathbb{R}\times M$
defined by $v_{n}^{-}=u_{n}\circ (\theta_{n}^{-})^{-1}$, converge
in $C^{0}$ to a map\\ $v^{-}:[-1,-1/2]\times S^{1}\rightarrow\mathbb{R}\times M$
such that $v^{-}=u^{-}\circ(\theta^{-})^{-1}$ and $v^{-}(-1/2,t)=(\sigma,\phi_{\tau}^{\alpha}(w_{f}))$. 
\item The maps $v_{n}^{+}:[1/2,1]\times S^{1}\rightarrow\mathbb{R}\times M$
defined by $v_{n}^{+}=u_{n}\circ (\theta_{n}^{+})^{-1}$, converge
in $C^{0}$ to a map\\ $v^{+}:[1/2,1]\times S^{1}\rightarrow\mathbb{R}\times M$
such that $v^{+}(s,t)=u^{+}\circ (\theta^{+})^{-1}$ and $v^{+}(1/2,t)=(-\sigma,\phi_{-\tau}^{\alpha}(w_{f}))$. 
\end{singlespace}
\end{enumerate}
\end{thm}
\begin{singlespace}
\noindent An immediate Corollary is 
\end{singlespace}
\begin{cor}
\begin{singlespace}
\noindent \label{cor:Under-the-same}Under the same hypothesis of
Theorem \ref{thm:With-the-same} the following $C_{\text{loc}}^{\infty}-$convergence
results hold. 
\end{singlespace}
\begin{enumerate}
\begin{singlespace}
\item The maps $v_{n}^{-}$ converge in $C_{\text{loc}}^{\infty}$
to $v^{-}$, where $\lim_{s\rightarrow -1/2}v^{-}(s,\cdot)=(\sigma,\phi_{\tau}^{\alpha}(w_{f}))$ in $C^\infty(S^1,\mathbb{R}\times M)$. The harmonic $1-$forms $[(\theta_{n}^{-})^{-1}]^{*}\gamma_{n}^{-}$
with respect to the complex structure $[(\theta_{n}^{-})^{-1}]^{*}i$
converge in $C_{\text{loc}}^{\infty}$ to a harmonic $1-$form $[(\theta^{-})^{-1}]^{*}d\Gamma^{-}$
with respect to the complex structure $[(\theta^{-})^{-1}]^{*}i$.
\item The maps $v_{n}^{+}$ converge in $C_{\text{loc}}^{\infty}$
to $v^{+}$, where $\lim_{s\rightarrow 1/2}v^{+}(s,\cdot)=(-\sigma,\phi_{-\tau}^{\alpha}(w_{f}))$ in $C^\infty(S^1,\mathbb{R}\times M)$. The harmonic $1-$forms $[(\theta_{n}^{+})^{-1}]^{*}\gamma_{n}^{-}$
with respect to the complex structure $[(\theta_{n}^{+})^{-1}]^{*}i$
converge in $C_{\text{loc}}^{\infty}$ to a harmonic $1-$form $[(\theta^{+})^{-1}]^{*}d\Gamma^{+}$
with respect to the complex structure $[(\theta^{+})^{-1}]^{*}i$.
\end{singlespace}
\end{enumerate}
\end{cor}
\begin{singlespace}

\begin{defn}
Let $(r_-,r_+)\subset \mathbb{R}$ be an interval and $u,v:(r_-,r_+)\times S^1\rightarrow \mathbb{R}\times M$ be smooth maps. We say $u$ and $v$ have \emph{the same asymptotics as $s\rightarrow r_\pm$},
if for some (or equivalently, any) metric $d_{C^\infty} $ metrizing the Whitney $C^\infty$-topology on $C^\infty(S^1,\mathbb{R}\times M)$, 
which is invariant under the $\mathbb{R}$-action on $\mathbb{R}\times M$, we have
$$\lim_{s\rightarrow r_\pm}
d_{C^\infty} (u(s,\cdot),v(s,\cdot))=0.$$
\end{defn}

\noindent Theorem \ref{thm:With-the-same} concerns $\mathcal{H}$-holomorphic curves with vanishing center action. If instead the maps $u_n$ (after passing to a subsequence) have all positive center action, we have the following. 
\end{singlespace}
\begin{thm}
\begin{singlespace}
\noindent \label{thm:With-the-same-1}Let $u_{n}:[-R_{n},R_{n}]\times S^{1}\rightarrow\mathbb{R}\times M$
be a sequence of $\mathcal{H}-$holomorphic cylinders with harmonic
perturbations $\gamma_{n}$ satisfying $A0$-$A3$. Assume that $R_{n}\rightarrow\infty$
and that all the $u_{n}$ have positive center action.\\
Then there exist a subsequence of $u_{n}$ (for each
$n$ suitably shifted in the $\mathbb{R}-$coordinate), still denoted
by $u_{n}$, $\mathcal{H}-$holomorphic half cylinders $u^{\pm}$ with exact harmonic perturbations $d\Gamma^\pm$
defined on $(-\infty,0]\times S^{1}$ and $[0,\infty)\times S^{1}$
respectively, $T\in \mathbb{R}\setminus\{0\}$ and a $|T|$-periodic Reeb orbit $x$ and
$\sigma,\tau\in\mathbb{R}$
such that for every sequence $h_{n}\in\mathbb{R}_{>0}$ with $h_n,R_n/h_n\rightarrow \infty$, the following convergence
results hold: $\sigma_{n}:=S_{n}R_{n}\rightarrow\sigma,\tau_{n}:=P_{n}R_{n}\rightarrow\tau$ and

\noindent $C_{\text{loc}}^{\infty}-$convergence: 
\end{singlespace}
\begin{enumerate}
\begin{singlespace}
\item 
For any sequence $s_{n}\in[-R_{n}+h_{n},R_{n}-h_{n}]$ for which $s_{n}/R_{n}$
converges to $\kappa\in[-1,1]$, the shifted maps $u_{n}(s+s_{n},t)-Ts_n$,
defined on $[-R_{n}+h_{n}-s_{n},R_{n}-h_{n}-s_{n}]\times S^{1}$,
converge in $C_{\text{loc}}^{\infty}$ on $\mathbb{R}\times S^1$ to the map $(s,t)\rightarrow (Ts-\sigma\kappa,\phi_{-\tau\kappa}^{\alpha}(x(Tt))=x(Tt-\tau\kappa))$ on $\mathbb{R}\times S^1$.
The shifted harmonic $1-$forms $\gamma_{n}(s+s_{n},t)$
converge in $C_{\text{loc}}^{\infty}$ to $0$. 
\item The left shifts $u_{n}^{-}(s,t):=u_n(s-R_n,t)+TR_n$, defined on $[0,h_{n})\times S^{1}$, 
converge in $C_{\text{loc}}^{\infty}$ to $u^{-}=(a^{-},f^{-})$ defined on $[0,+\infty)\times S^{1}$.
The curve $u^{-}$ has the same asymptotics as $(Ts+\sigma,\phi_{\tau}^{\alpha}(x(Tt))=x(Tt+\tau))$ as $s\rightarrow \infty$.
The left shifted harmonic $1-$forms $\gamma_{n}^{-}$
converge in $C_{\text{loc}}^{\infty}$ to
$d\Gamma^{-}$ defined on $[0,+\infty)\times S^{1}$ with $\|d\Gamma^-\|^2_{L^2([0,\infty)\times S^1)}\leq C_0$.
\item The right shifts $u_{n}^{+}(s,t):=u(s+R_n,t)-TR_n$, defined on $(-h_{n},0]\times S^{1}$ converge in $C_{\text{loc}}^{\infty}$ to $u^{+}=(a^{+},f^{+})$, defined on $(-\infty,0]\times S^{1}$.
The curve $u^{+}$ has the same asymptotics as $(Ts-\sigma,\phi_{-\tau}^{\alpha}(x(Tt))=x(Tt-\tau))$ as $s\rightarrow -\infty$.
The right shifted harmonic $1-$forms $\gamma_{n}^{+}$
converge in $C_{\text{loc}}^{\infty}$ to
$d\Gamma^{+}$ defined on $(-\infty,0]\times S^{1}$ with $\|d\Gamma^+\|^2_{L^2((-\infty,0]\times S^1)}\leq C_0$.
\end{singlespace}
\end{enumerate}
\begin{singlespace}
\noindent $C^{0}-$convergence: For every sequence
of diffeomorphisms $\theta_{n}:[-R_{n},R_{n}]\times S^1\rightarrow[-1,1]\times S^1$ as
in Remark \ref{rem:For-every-sequence}
\end{singlespace}
\begin{enumerate}
\begin{singlespace}
\item The maps $f_{n}\circ\theta_{n}^{-1}:[-1/2,1/2]\times S^{1}\rightarrow M$
converge in $C^{0}$ to $\phi_{-2\tau s}^{\alpha}(x(Tt))=x(Tt-2\tau s)$. 
\item The maps $f_{n}^{-}\circ(\theta_{n}^{-})^{-1}:[-1,-1/2]\times S^{1}\rightarrow M$
converge in $C^{0}$ to a map $f^{-}\circ(\theta^{-})^{-1}:[-1,-1/2]\times S^{1}\rightarrow M$
such that $f^{-}((\theta^{-})^{-1}(-1/2),t)=\phi_{\tau}^{\alpha}(x(Tt))=x(Tt+\tau)$. 
\item The maps $f_{n}^{+}\circ(\theta_{n}^{+})^{-1}:[1/2,1]\times S^{1}\rightarrow M$
converge in $C^{0}$ to a map $f^{+}\circ(\theta^{+})^{-1}:[1/2,1]\times S^{1}\rightarrow M$
such that $f^{+}((\theta^{+})^{-1}(1/2),t)=\phi_{-\tau}^{\alpha}(x(Tt))=x(Tt-\tau)$. 
\end{singlespace}
\item Set $r_n^-:=\inf_{t\in S^1} a_n(-\text{sgn}(T)R_n,t)$ and $r_n^+:=\sup_{t\in S^1}a_n(\text{sgn}(T)R_n,t)$ where $\text{sgn}(T):=T/|T|\in \{\pm 1\}$. \\
Then $r_n^+-r_n^-\rightarrow \infty$ and for every $R>0$ there exists $\rho>0$ and $N\in\mathbb{N}$
such that $a_{n}\circ\theta_{n}^{-1}(s,t)\in[r_{n}^{-}+R,r_{n}^{+}-R]$
for all $n\geq N$ and all $(s,t)\in[-\rho,\rho]\times S^{1}$.
\end{enumerate}
\end{thm}
\begin{singlespace}
\noindent An immediate corollary is 
\end{singlespace}
\begin{cor}\label{cor:underthesame_1}
\begin{singlespace}
\noindent Under the same hypothesis of Theorem \ref{thm:With-the-same-1}
we have the
following $C_{\text{loc}}^{\infty}-$convergence results. 
\end{singlespace}
\begin{enumerate}
\begin{singlespace}
\item The maps $v_{n}^{-}+TR_{n}$ converge in $C_{\text{loc}}^{\infty}$
to $v^{-}$ where $f^{-}((\theta^{-})^{-1}(-1/2,t))=x(Tt+\tau)$.
The harmonic $1-$forms $[(\theta_{n}^{-})^{-1}]^{*}\gamma_{n}^{-}$
with respect to the complex structure $[(\theta_{n}^{-})^{-1}]^{*}i$
converge in $C_{\text{loc}}^{\infty}$ to a harmonic $1-$form $[(\theta^{-})^{-1}]^{*}d\Gamma^{-}$
with respect to the complex structure $[(\theta^{-})^{-1}]^{*}i$.
\item The maps $v_{n}^{+}-TR_{n}$ converge in $C_{\text{loc}}^{\infty}$
to $v^{+}$ where $f^{+}((\theta^{+})^{-1}(1/2,t))=x(Tt-\tau)$. The
harmonic $1-$forms $[(\theta_{n}^{+})^{-1}]^{*}\gamma_{n}^{-}$ with
respect to the complex structure $[(\theta_{n}^{+})^{-1}]^{*}i$ converge
in $C_{\text{loc}}^{\infty}$ to a harmonic $1-$form $[(\theta^{+})^{-1}]^{*}d\Gamma^{+}$
with respect to the complex structure $[(\theta^{+})^{-1}]^{*}i$. 
\end{singlespace}
\end{enumerate}
\end{cor}
\begin{singlespace}
\noindent Establishing this, we need to make use of a modified version
of the results from \cite{key-11}. 
\end{singlespace}
\begin{rem}
\begin{singlespace}
\noindent For a sequence of $\mathcal{H}-$holomorphic
curves $u_{n}$ together with the harmonic perturbations $\gamma_{n}$
satisfying $A0$-$A3$ and $R_{n}\rightarrow\infty$ we
can conclude that the left and right shifts $u_{n}^{\pm}$ together
with the harmonic perturbations $\gamma_{n}^{\pm}$ defined on $[0,h_{n}]\times S^{1}$
and $[-h_{n},0]\times S^{1}$, respectively, converge after a suitable
shift in the $\mathbb{R}-$coordinate in $C_{\text{loc}}^{\infty}$,
to the $\mathcal{H}-$holomorphic half cylinders $u^{\pm}$ with harmonic
perturbations $d\Gamma^{\pm}$ defined on $[0,\infty)\times S^{1}$
and $(-\infty,0]\times S^{1}$, respectively. The $\mathcal{H}-$holomorphic
curves $u^{\pm}$ are asymptotic to the points $w^{\pm}=(w_{a}^{\pm},w_{f}^{\pm})\in\mathbb{R}\times M$
or trivial cylinders over Reeb orbits $(x^{\pm},T)$. Without the
assumption A3, the asymptotic data of $u^{-}$ and
$u^{+}$ cannot be described as in Theorems \ref{thm:With-the-same}
and \ref{thm:With-the-same-1}. In fact, dropping assumption A3
it is not possible to connect the asymptotic data $w^{-}$ or $x^{-}(T\cdot)$
of the left shifted $\mathcal{H}-$holomorphic curve $u^{-}$ to the
asymptotic data $w^{+}$ of $x^{+}(T\cdot)$ of the right shifted
$\mathcal{H}-$holomorphic curve $u^{+}$ by a compact cylinder as
in Theorems \ref{thm:With-the-same} and \ref{thm:With-the-same-1}.
Examples provided in Appendix \ref{subsec:A-version-of-example}, show that these results do not hold, 
if the assumption A3 is omitted.
\end{singlespace}
\end{rem}

\subsection{Outline of the paper}

\begin{singlespace}
\noindent The paper is organized as follows. We begin by considering
a general $\mathcal{H}-$holomorphic cylinder $u=(a,f):[-R,R]\times S^{1}\rightarrow\mathbb{R}\times M$
with harmonic perturbation $\gamma$ satisfying assumptions $A0-A3$.
\noindent In Section \ref{subsec:Notion-of-the}, this $\mathcal{H}-$holomorphic
curve is transformed, as in \cite{key-20}, by the flow $\phi^{\alpha}:\mathbb{R}\times M\rightarrow M$
of the Reeb vector field $X_{\alpha}$ into a usual pseudoholomorphic
curve with respect to a domain-dependent almost complex structure
that varies in a compact set. Here, the bound on the
conformal period in condition A3 is essential. 
The transformed curve
is a map which is pseudoholomorphic for a domain-dependent almost complex structure $\overline{J}_P$. 
The domain dependence is relatively mild, since it depends on the point $(s,t)\in [-R,R]\times S^1$ in the domain
only via the product $Ps$ and we have $|Ps|\leq C$ for all $s\in[-R,R]$ by A3, c.f. Remark \ref{rem:The-parameter-dependent} and Definition \ref{def:A-triple-}. 
The conditions imposed on the energy are transferred to the $\overline{J}_{P}-$holomorphic
curves. We then derive a notion of center action for the $\overline{J}_{P}-$holomorphic
curves by employing the same arguments as in Theorem 1.1 of \cite{key-11};
here, we distinguish the cases when the center action vanishes and
is greater than $\hbar$.

\noindent In Section \ref{subsec:Zero-Center-Action} we consider
the case of vanishing center action. First, we derive a result for
$\overline{J}_{P}-$holomorphic curves, which is similar to that established
in Theorem 1.2 of \cite{key-11}, and which basically states that
a finite energy $\overline{J}_{P}-$holomorphic curve with uniformly
small $d\alpha-$energy and having vanishing center action, is close
to a point in $\mathbb{R}\times M$. This is done by using a version
of monotonicity Lemma for $\overline{J}_{P}-$holomorphic curves given
in Appendix \ref{subsec:A-version-of}. Then we describe the asymptotic
behavior of $\overline{J}_{P}-$holomorphic curves, and finally, by
using the inverse transformation with the flow of the Reeb vector
field, we translate these results into the language of $\mathcal{H}-$holomorphic
cylinders and prove Theorem \ref{thm:With-the-same}.

\noindent In Section \ref{subsec:Positive-Center-Action} we formulate
the above findings in the case of positive center action. We prove
a result which is similar to that stated by Theorem 1.3 of \cite{key-11}
for $\overline{J}_{P}-$holomorphic curves, and then Theorem \ref{thm:With-the-same-1}.

\noindent In order to prove Theorems \ref{thm:With-the-same} and
\ref{thm:With-the-same-1} we use a compactness result for a sequence
of harmonic functions defined on cylinders and possessing certain
properties; this is established in Appendix \ref{sec:Convergence-of}. 
\end{singlespace}
\begin{acknowledgement*}
\begin{singlespace}
\noindent We thank Peter Albers and Kai Cieliebak who provided insight
and expertise that greatly assisted the research.
We would also like to thank the anonymous referee for carefully reading the manuscript and many constructive comments which helped improving the quality of the paper. \\
U.F. is supported
by the SNF fellowship 155099, a fellowship at Institut Mittag-Leffler
and the GIF Grant 1281. 
\end{singlespace}
\end{acknowledgement*}

\section{\label{subsec:Notion-of-the}$\overline{J}_{P}-$holomorphic curve
and center action}

\begin{singlespace}
\noindent In this section we transform a $\mathcal{H}-$holomorphic
curve into a pseudoholomorphic curve with domain-dependent almost
complex structure on the target space $\mathbb{R}\times M$, and introduce
a notion of center action for this curve. 
\end{singlespace}

\subsection{$\overline{J}_{P}-$holomorphic curve }

\begin{singlespace}
\noindent We consider an $\mathcal{H}-$holomorphic curve $u=(a,f):[-R,R]\times S^{1}\rightarrow\mathbb{R}\times M$
with harmonic perturbation $\gamma$ satisfying Assumptions A0-A3,
and construct a new map $\overline{u}=(\overline{a},\overline{f}):[-R,R]\times S^{1}\rightarrow\mathbb{R}\times M$
as follows. Recall that $P:=P(\gamma)$ is the period of $\gamma$ over $\{0\}\times S^1$ from (\ref{eq:period_over_the_cylinder}) and let $\phi_{t}^{\alpha}:M\rightarrow M$ be the Reeb flow
on $M$. Defining 
\begin{equation}
\overline{f}(s,t):=\phi_{Ps}^{\alpha}(f(s,t))\label{eq:bar=00003D00007Bf=00003D00007D}
\end{equation}
we find by straightforward calculation that 
\[
\pi_{\alpha}d\overline{f}=d\phi_{Ps}^{\alpha}\pi_{\alpha}df\ \text{ and }\ \overline{f}^{*}\alpha=Pds+f^{*}\alpha
\]
giving 
\begin{align}
\overline{f}^{*}\alpha\circ i & =-Pdt+f^{*}\alpha\circ i=-Pdt+da+\gamma.\label{eq:secondeq}
\end{align}

\end{singlespace}
\begin{rem}
\begin{singlespace}
\noindent \label{rem:Obviously,-as-}Obviously, as $\gamma$ is a
harmonic $1-$form, the $1-$form $-Pdt+\gamma$ is harmonic with
vanishing period over $[-R,R]\times S^{1}$. Thus $-Pdt+\gamma$ is
globally exact, i.e. there exists a harmonic function $\Gamma:[-R,R]\times S^{1}\rightarrow\mathbb{R}$
which is unique up to addition by a constant such that $-Pdt+\gamma=d\Gamma$.

Note that $$\|d\Gamma\|^2_{L^2([-R,R]\times S^1)}\leq (\|\gamma\|_{L^2([-R,R]\times S^1)}+\sqrt{2RP^2})^2\leq (\sqrt{C_0}+\sqrt{2CP})^2$$
with $C_0$ and $C$ from assumptions $A0$ resp. $A3$.
For technical reasons, which will become apparent later on, we choose
$\Gamma$ such that it has a vanishing mean value over $[-R,R]\times S^{1}$,
i.e. 
\[
\frac{1}{2R}\int_{[-R,R]\times S^{1}}\Gamma(s,t)ds\wedge dt=0.
\]
\end{singlespace}
\end{rem}
\begin{singlespace}
\noindent Set 
\begin{equation}
\overline{a}:=a+\Gamma\label{eq:bar=00003D00007Ba=00003D00007D}
\end{equation}
where $\Gamma$ was chosen as in Remark \ref{rem:Obviously,-as-}.

\noindent Define the parameter-dependent complex structure $\mathcal{J}:[-C,C]\times M\rightarrow\text{End}(\xi)$
by 
\begin{equation}
\mathcal{J}(\rho,m)=d\phi_{\rho}^{\alpha}(\phi_{-\rho}^{\alpha}(m))\circ J_{\xi}(\phi_{-\rho}^{\alpha}(m))\circ d\phi_{-\rho}^{\alpha}(m)\label{eq:domain_dependent_almost_complex_structure}
\end{equation}
for all $\rho\in[-C,C]$ and all $m\in M$, where $C>0$ is the constant
from assumption A3. We write also $J_\rho:=\mathcal{J}(\rho,\cdot)$ for the complex structure on $\xi$ obtained by fixing $\rho\in [-C,C]$. Each $J_\rho$ extends as usual to a $\mathbb{R}$-invariant 
almost complex structure on $\mathbb{R}\times M$, still denoted by $J_\rho$. Note that $J_0$ is the original almost complex structure $J$ from (\ref{eq:JJJ}).  
\end{singlespace}
\begin{prop}
\begin{singlespace}
\noindent \label{prop:The-curve-}The curve $\overline{u}=(\overline{a},\overline{f}):[-R,R]\times S^{1}\rightarrow\mathbb{R}\times M$,
where $\overline{a}$ and $\overline{f}$ are the maps defined by
(\ref{eq:bar=00003D00007Ba=00003D00007D}) and (\ref{eq:bar=00003D00007Bf=00003D00007D}),
is pseudoholomorphic with respect to a domain-dependent, $\mathbb{R}$-invariant
almost complex structure $\overline{J}_P$ on $\mathbb{R}\times M$ defined by $\overline{J}_P((s,t),(r,m)):=J_{Ps}(m)$ for $(s,t)\in [-R,R]\times S^1$ and $(r,m)\in \mathbb{R}\times M$, i.e.
\begin{align}
\pi_{\alpha}d\overline{f}(s,t)\circ i & =J_{Ps}(\overline{f}(s,t))\circ\pi_{\alpha}d\overline{f}(s,t),\label{eq:first}\\
(\overline{f}^{*}\alpha)\circ i & =d\overline{a}\label{eq:second}
\end{align}
for all $(s,t)\in[-R,R]\times S^{1}$. Moreover, for the $\alpha-$
and $d\alpha-$energies we have 
\begin{align*}
E_{d\alpha}(\overline{u};[-R,R]\times S^{1}) & =E_{d\alpha}(u;[-R,R]\times S^{1}),\\
E_{\alpha}(\overline{u};[-R,R]\times S^{1}) & \leq\int_{\{R\}\times S^{1}}|f^{*}\alpha|+\int_{\{-R\}\times S^{1}}|f^{*}\alpha|+E_{d\alpha}(u;[-R,R]\times S^{1}).
\end{align*}
\end{singlespace}
\end{prop}
\begin{proof}
\begin{singlespace}
\noindent 
 Note that for a $\mathcal{H}-$holomorphic curve
$u:[-R,R]\times S^{1}\rightarrow\mathbb{R}\times M$ satisfying assumptions
A0-A3 we have $Ps\in[-C,C]$ for all $s\in[-R,R]$ (here $P=\tau/R$ as in the introduction). Thus $\overline{J}_P$ is well-defined.\\
Next, by (\ref{eq:secondeq}) and (\ref{eq:bar=00003D00007Ba=00003D00007D})
it is obvious that (\ref{eq:second}) holds. Let us consider (\ref{eq:first}).
The left-hand side of this equation goes over in $\pi_{\alpha}d\overline{f}(s,t)\circ i=d\phi_{Ps}^{\alpha}\pi_{\alpha}df$,
while the right-hand side goes over in $J_{Ps}(\overline{f}(s,t))\circ\pi_{\alpha}d\overline{f}(s,t)=d\phi_{Ps}^{\alpha}(f(s,t))\circ J_{\xi}(f(s,t))\circ\pi_{\alpha}df(s,t)$.
Hence, (\ref{eq:first}) is satisfied. Thus $\overline{u}=(\overline{a},\overline{f}):[-R,R]\times S^{1}\rightarrow\mathbb{R}\times M$
is an $i-\overline{J}_P-$holomorphic curve, where $\overline{J}_P$ is
a parameter-dependent almost complex structure. The energies transform
as follows. The $d\alpha-$energy remains unchanged, since $d\alpha$
is invariant under the flow $\phi^{\alpha}$. For the $\alpha-$energy
we have by using (\ref{eq:second}) that 
\begin{align*}
E_{\alpha}(\overline{u};[-R,R]\times S^{1}) & =\sup_{\varphi\in\mathcal{A}}\left[-\int_{[-R,R]\times S^{1}}d(\varphi(\overline{a})d\overline{a}\circ i)+\int_{[-R,R]\times S^{1}}\varphi(\overline{a})d(d\overline{a}\circ i)\right]\\
 & \leq\left[\int_{\{R\}\times S^{1}}|f^{*}\alpha|+\int_{\{-R\}\times S^{1}}|f^{*}\alpha|\right].
\end{align*}
Here we use that the second integrand is non-positive, since $d(d\overline{a}\circ i)=-f^* d\alpha$ is a non-positive two-form on $[-R,R]\times S^1$ and $\varphi$ is a non-negative function.
\end{singlespace}
\end{proof}
\begin{rem}
\begin{singlespace}
\noindent \label{rem:For-our-sequence}The $\alpha-$energy of $\overline{u}$ constructed as above from some $\mathcal{H}$-holomorphic cylinder $u$
satisfying $A0$-$A3$ is uniformly bounded. To show this we argue as follows. Due to assumption
A1, the quantities 
\[
\int_{\{R\}\times S^{1}}|f^{*}\alpha|\ \text{ and }\ \int_{\{-R\}\times S^{1}}|f^{*}\alpha|
\]
are uniformly bounded by the constant $C_{1}>0$. Hence, we see that
\[
E(\overline{u};[-R,R]\times S^{1})=E_{\alpha}(\overline{u};[-R,R]\times S^{1})+E_{d\alpha}(\overline{u};[-R,R]\times S^{1})\leq\tilde{E}_{0}:=2C_1+E_0.
\]
\end{singlespace}
\end{rem}
\begin{singlespace}
\noindent To analyze the properties of the transformed pseudoholomorphic
curve $\overline{u}$, we consider the following additional structure
on $M$: On the contact structure $\xi=\ker(\alpha)$, let $\mathcal{J}:\left[-C,C\right]\times M\rightarrow\text{End}(\xi)$
be the parameter-dependent almost complex structure defined by (\ref{eq:domain_dependent_almost_complex_structure}). On $\mathbb{R}\times M$ we use
the following family of Riemannian metrics: 
\begin{align}
\overline{g}_{\rho,m}(v,w) & =dr\otimes dr(v,w)+\alpha\otimes\alpha(v,w)+d\alpha(v,J_{\rho}(m)w)\label{eq:metric_variational}
\end{align}
for all $\rho\in[-C,C]$ and all $m\in M$, where $r$ is the coordinate
on the $\mathbb{R}-$component of $\mathbb{R}\times M$. Before
going any further we make a remark about the metrics involved. 
\end{singlespace}
\begin{rem}
\begin{singlespace}
\noindent \label{rem:The-parameter-dependent}For
any $\rho$, the norms induced by the metrics $\overline{g}_{\rho}$
on $\mathbb{R}\times M$ that are defined by (\ref{eq:metric_variational})
are equivalent, i.e. there exists a positive constant $\overline{C}_{1}>0$
such that for every $\rho\in[-C,C]$ we have 
\begin{equation}
\frac{1}{\overline{C}_{1}}\left\Vert \cdot\right\Vert _{\overline{g}_{\rho}}\leq\left\Vert \cdot\right\Vert _{\overline{g}_{0}}
\leq\overline{C}_{1}\left\Vert \cdot\right\Vert _{\overline{g}_{\rho}}.\label{eq:equivalence_metrics}
\end{equation}
This follows from the fact that the complex structures $J_{\rho}=\mathcal{J}(\rho,\cdot)$ on $\xi$
defined by (\ref{eq:domain_dependent_almost_complex_structure}) vary
continuously in a compact set of the space of complex structures on $\xi$ tamed by
$d\alpha$ for $\rho\in[-C,C]$.
\end{singlespace}
\end{rem}
\begin{defn}
\begin{singlespace}
\noindent \label{def:A-triple-}A triple $(\overline{u},R,P)$
is called a \emph{$\overline{J}_{P}-$holomorphic
curve} if $P,R\in\mathbb{R}$ with $R>0$, $|PR|\leq C$, and $\overline{u}=(\overline{a},\overline{f}):[-R,R]\times S^{1}\rightarrow\mathbb{R}\times M$
satisfies 
\begin{align*}
\pi_{\alpha}d\overline{f}(s,t)\circ i & =J_{Ps}(\overline{f}(s,t))\circ\pi_{\alpha}d\overline{f}(s,t),\quad\text{ and }\quad\overline{f}^{*}\alpha\circ i=d\overline{a}
\end{align*}
on $[-R,R]\times S^{1}$.
\end{singlespace}
\end{defn}
\begin{rem}
\noindent \label{rem:In-the-following}In the following
we consider for $\tilde{C}_1>0$ the $\overline{J}_{P}-$homolorphic curves $(\overline{u},R,P)$
satisfying the following assumptions: 

\begin{description}
\item[\textbf{$\overline{A}$0}] $E(\overline{u};[-R,R]\times S^{1})\leq \tilde{E}_{0}:=2C_{1}+E_{0}$
(c.f. assumption A2). 
\item[$\overline{A}$1] For the constant $\delta_{1}>0$ from Assumption A1
we have $\left\Vert d\overline{f}(z)\right\Vert \leq\tilde{C}_{1}$
for all $z\in([-R,-R+\delta_{1}]\amalg[R-\delta_{1},R])\times S^{1}$. 
\item[$\overline{A}$2] $E_{d\alpha}(\overline{u};[-R,R]\times S^1)\leq\hbar/2$, where $\hbar$ is as in $A2$.
\item [{$\overline{A}$3}]  For the constant $C>0$ from Assumption A3 we have
$\left|PR\right|\leq C$ and $|SR|\leq C$. 
\end{description}

If $\tilde{C}_1\geq C$ is sufficiently large, then for any $\mathcal{H}$-holomorphic cylinder $u$ satisfying $A0$-$A3$, 
the associated $\overline{J}_{P}-$homolorphic curve $(\overline{u},R,P)$
satisfies $\overline{A}0$-$\overline{A}3$. This follows immediately from Remark \ref{rem:For-our-sequence} and the definition of $\overline{f}$ in (\ref{eq:bar=00003D00007Bf=00003D00007D}).\\
We fix for the rest of the paper such a constant $\tilde{C}_1$. When we say that \emph{a $\overline{J}_{P}-$homolorphic curve $(\overline{u},R,P)$ satisfies $\overline{A}0$-$\overline{A}3$}, 
it always means that the \emph{$\overline{J}_{P}-$homolorphic curve $(\overline{u},R,P)$ satisfies $\overline{A}0$-$\overline{A}3$ 
for this $\tilde{C}_1>0$ (and for the constants $E_{0},C_{0},C_{1}$ and $\delta_{1}$ fixed throughout the paper)}.
\end{rem}

\subsection{\label{subsec:Center-action}Center action}

\begin{singlespace}
\noindent In the following we apply the results established in \cite{key-11}
to this new curve, and introduce the notion of the center action for
the $\overline{J}_{P}-$holomorphic curve $(\overline{u},R,P)$.

\noindent The next result is similar to Theorem 1.1 of \cite{key-11}. 
\end{singlespace}
\begin{thm}
\begin{singlespace}
\noindent \label{thm:For-all-numbers}For all $\psi$ such that $0<\psi<\hbar/2$,
there exists $h_{0}>0$ such that for any $R>h_{0}$ and any $\overline{J}_{P}-$holomorphic
curve $(\overline{u},R,P)$ satisfying $\overline{A}0$- $\overline{A}3$ there exists a unique element $T\in\mathcal{P}$
such that $T\leq\tilde{E}_{0}$ and 
\[
\left|\left|\int_{S^{1}}\overline{u}(0)^{*}\alpha\right|-T\right|<\frac{\psi}{2}.
\]
\end{singlespace}
\end{thm}

\begin{singlespace}
\noindent As in \cite{key-11}, the unique element $T\in\mathcal{P}$
associated with the $\overline{J}_{P}-$holomorphic
curve $(\overline{u},R,P)$ satisfying the assumptions of Remark \ref{rem:In-the-following}
is called \emph{the center action of $\overline{u}$} and is denoted by 
\begin{equation}
T=A(\overline{u}).\label{eq:defcenteraction}
\end{equation}
If $u$ is a $\mathcal{H}$-holomorphic cylinder satisfying $A0$-$A3$, then we define \emph{the center action of $u$} as $A(u):=A(\overline{u})$, where
$\overline{u}$ is the $\overline{J}_{P}-$holomorphic
curve $(\overline{u},R,P)$ associated to $u$, which according to Remark \ref{rem:In-the-following} satisfies $\overline{A}0$-$\overline{A}3$.\\
Note that while $\int_{S^{1}}\overline{u}(0)^{*}\alpha$ may in general have an arbitrary sign, it may be assumed to be non-negative by possibly replacing $\overline{u}(s,t)$ by $\overline{u}(-s,-t)$.
\end{singlespace}
\begin{rem}
\begin{singlespace}
\noindent \label{rem:From-the-definition}From the definition of the
constant $\hbar$, the center action $A(\overline{u})$ of a curve
$\overline{u}$ fulfilling the assumptions of Theorem \ref{thm:For-all-numbers}
satisfies $A(\overline{u})=0$ or $A(\overline{u})\geq\hbar$. 
\end{singlespace}
\end{rem}

\begin{singlespace}
\noindent To prove the theorem \ref{thm:For-all-numbers} we need the following 
\end{singlespace}
\begin{lem}
\begin{singlespace}
\noindent \label{lem:For-every--1}For any $\delta>0$ there exists
a constant $C'_{1}>0$ such that the gradients of all $\overline{J}_{P}-$holomorphic
curves $(\overline{u},R,P)$ satisfying $\overline{A}0$-$\overline{A}3$ and $R>\delta$, are uniformly bounded
on $[-R+\delta,R-\delta]\times S^{1}$ by the constant $C'_{1}$,
i.e.

\noindent 
\[
\sup_{(s,t)\in[-R+\delta,R-\delta]\times S^{1}}\left\Vert d\overline{u}(s,t)\right\Vert _{g_{\text{eucl.}},\overline{g}_{Ps}}\leq C'_{1}.
\]
\end{singlespace}
\end{lem}
\begin{proof}
\begin{singlespace}
\noindent We prove this lemma by using bubbling-off analysis. Let
us assume that the assertion is not true. Then we find $\delta_{0}>0$
such that for any $C_{1,n}=n$ there exist $\overline{J}_{P_{n}}-$holomorphic
curves $(\overline{u}_{n},R_{n},P_{n})$ with $R_{n}>\delta_{0}$
such that 
\[
\sup_{(s,t)\in[-R_{n}+\delta_{0},R_{n}-\delta_{0}]\times S^{1}}\left\Vert d\overline{u}_{n}(s,t)\right\Vert _{g_{\text{eucl.}},\overline{g}_{P_{n}s}}\geq C_{1,n}=n.
\]
Therefore there exist points $(s_{n},t_{n})\in[-R_{n}+\delta_{0},R_{n}-\delta_{0}]\times S^{1}$
for which 
\[
\left\Vert d\overline{u}_{n}(s_{n},t_{n})\right\Vert _{g_{\text{eucl.}},\overline{g}_{P_{n}s_{n}}}=\sup_{(s,t)\in[-R_{n}+\delta_{0},R_{n}-\delta_{0}]\times S^{1}}\left\Vert d\overline{u}_{n}(s,t)\right\Vert _{g_{\text{eucl.}},\overline{g}_{P_{n}s}}\geq n.
\]
Set $\mathcal{R}_{n}:=\left\Vert d\overline{u}_{n}(s_{n},t_{n})\right\Vert _{g_{\text{eucl.}},\overline{g}_{P_{n}s_{n}}}$and
note that $\mathcal{R}_{n}\rightarrow\infty$. Choose a sequence $\epsilon_{n}$
such that $\epsilon_{n}>0$, $\epsilon_{n}\rightarrow0$ and $\epsilon_{n}\mathcal{R}_{n}\rightarrow+\infty$.
Now, apply Hofer's metric lemma \cite{key-18} to the continuous
sequence of functions $\left\Vert d\overline{u}_{n}(s,t)\right\Vert _{g_{\text{eucl.}},\overline{g}_{P_{n}s}}$
defined on $[-R_{n},R_{n}]\times S^{1}$. For each $(s_{n},t_{n})$
and $\epsilon_{n}$, there exist $(s_{n}',t_{n}')\in[-R_{n}+\delta_{0},R_{n}-\delta_{0}]\times S^{1}$
and $\epsilon_{n}'\in(0,\epsilon_{n}]$ with the properties:
\end{singlespace}
\begin{enumerate}
\begin{singlespace}
\item $\epsilon_{n}'\left\Vert d\overline{u}_{n}(s_{n}',t_{n}')\right\Vert _{g_{\text{eucl.}},\overline{g}_{P_{n}s_{n}'}}\geq\epsilon_{n}\left\Vert d\overline{u}_{n}(s_{n},t_{n})\right\Vert _{g_{\text{eucl.}},\overline{g}_{P_{n}s_{n}}}$; 
\item $|(s_{n},t_{n})-(s_{n}',t_{n}')|_{g_{\text{eucl.}}}\leq2\epsilon_{n}$; 
\item $\left\Vert d\overline{u}_{n}(s,t)\right\Vert _{g_{\text{eucl.}},\overline{g}_{P_{n}s}}\leq2\left\Vert d\overline{u}_{n}(s_{n}',t_{n}')\right\Vert _{g_{\text{eucl.}},\overline{g}_{P_{n}s_{n}'}}$
for all $(s,t)$ such that $|(s,t)-(s_{n}',t_{n}')|\leq\epsilon_{n}'$. 
\end{singlespace}
\end{enumerate}
\begin{singlespace}
\noindent Thus we have found the points $(s_{n}',t_{n}')$ and a sequence
$\epsilon_{n}'$ such that:
\end{singlespace}
\begin{enumerate}
\begin{singlespace}
\item $\epsilon_{n}'>0$, $\epsilon_{n}'\rightarrow0$, $\mathcal{R}_{n}':=\left\Vert d\overline{u}_{n}(s_{n}',t_{n}')\right\Vert _{g_{\text{eucl.}},\overline{g}_{P_{n}s_{n}'}}\rightarrow\infty$
and $\epsilon_{n}'\mathcal{R}_{n}'\rightarrow\infty$; 
\item $\left\Vert d\overline{u}_{n}(s,t)\right\Vert _{g_{\text{eucl.}},\overline{g}_{P_{n}s}}\leq2\mathcal{R}_{n}'$
for all $(s,t)$ such that $|(s,t)-(s_{n}',t_{n}')|\leq\epsilon_{n}'$. 
\end{singlespace}
\end{enumerate}
\begin{singlespace}
\noindent Now we do rescaling. Setting $z_{n}'=(s_{n}',t_{n}')$ and
defining the maps 
\begin{align*}
\tilde{u}_{n}(s,t) & :=\left(\overline{a}_{n}\left(z_{n}'+\frac{z}{\mathcal{R}_{n}'}\right)-\overline{a}_{n}(z_{n}'),\overline{f}_{n}\left(z_{n}'+\frac{z}{\mathcal{R}_{n}'}\right)\right)=(\tilde{a}(z),\tilde{f}(z))
\end{align*}
for $z=(s,t)\in B_{\epsilon_{n}'\mathcal{R}_{n}'}(0)$, we obtain
\[
d\tilde{u}_{n}(z)=\frac{1}{\mathcal{R}_{n}'}d\overline{u}_{n}\left(z_{n}'+\frac{z}{\mathcal{R}_{n}'}\right)
\]
and 
\[
\left\Vert d\tilde{u}_{n}(z)\right\Vert _{g_{\text{eucl.}},\overline{g}_{P_{n}\left(s_{n}'+\frac{s}{\mathcal{R}_{n}'}\right)}}=\frac{1}{\mathcal{R}_{n}'}\left\Vert d\overline{u}_{n}\left(z_{n}'+\frac{z}{\mathcal{R}_{n}'}\right)\right\Vert _{g_{\text{eucl.}},\overline{g}_{P_{n}\left(s_{n}'+\frac{s}{\mathcal{R}_{n}'}\right)}}.
\]
Thus, for all $z=(s,t)\in B_{\epsilon_{n}'\mathcal{R}_{n}'}(0)$ we
have that 
\begin{equation}
\left\Vert d\tilde{u}_{n}(z)\right\Vert _{g_{\text{eucl.}},\overline{g}_{P_{n}\left(s_{n}'+\frac{s}{\mathcal{R}_{n}'}\right)}}\leq2\label{eq:gradbdd}
\end{equation}
and $\left\Vert d\tilde{u}_{n}(0)\right\Vert _{g_{\text{eucl.}},\overline{g}_{P_{n}s_{n}'}}=1$,
and moreover, that $\tilde{u}=(\tilde{a},\tilde{f})$ solves 
\begin{align*}
\pi_{\alpha}d\tilde{f}_{n}(z)\circ i & =J_{P_{n}\left(s_{n}'+\frac{s}{\mathcal{R}_{n}'}\right)}(\tilde{f}_{n}(z))\circ\pi_{\alpha}d\tilde{f}_{n}(z),\\
\tilde{f}_{n}^{*}\alpha\circ i & =d\tilde{a}_{n}.
\end{align*}
As $P_{n}s_{n}'$ is bounded by $C$, we go over to some convergent
subsequence, i.e., $P_{n}s_{n}'\rightarrow\rho$ as $n\rightarrow\infty$.
From the uniform gradient bound (\ref{eq:gradbdd}) it follows that
there exists a subsequence converging in $C_{\text{loc}}^{\infty}$
to some curve $\tilde{u}=(\tilde{a},\tilde{f}):\mathbb{C}\rightarrow\mathbb{R}\times M$
such that:
\end{singlespace}
\begin{enumerate}
\begin{singlespace}
\item $\tilde{u}$ solves 
\[
\pi_{\alpha}d\tilde{f}(z)\circ i=J_{\rho}(\tilde{f}(z))\circ\pi_{\alpha}d\tilde{f}(z)\text{ and }\tilde{f}^{*}\alpha\circ i=d\tilde{a};
\]
\item the gradient bounds go over in $\left\Vert d\tilde{u}(z)\right\Vert _{g_{\text{eucl.}},\overline{g}_{Ps'}}\leq2$
and $\left\Vert d\tilde{u}(0)\right\Vert _{g_{\text{eucl.}},\overline{g}_{Ps'}}=1$. 
\end{singlespace}
\end{enumerate}
 From the last two results, $\tilde{u}$ is a usual non-constant
pseudoholomorphic plane with energy by $\tilde{E}_{0}$
(finite energy plane). As the $d\alpha-$energy is smaller than $\hbar$
we arrive at a contradiction (see \cite{key-12}). 
\end{proof}
\begin{proof} \emph{(of Theorem \ref{thm:For-all-numbers})} We prove
Theorem \ref{thm:For-all-numbers} by contradiction. Assume that we
find $0<\tilde{\psi}<\hbar/2$ such that for any constant $h_{0,n}=n$,
there exist $R_{n}>h_{0,n}=n$ and a $\overline{J}_{P_n}-$holomorphic
curves $(\overline{u}_{n},R_{n},P_{n})$ satisfying $\overline{A}0-\overline{A}3$ and
\[
\left|\left|\int_{S^{1}}\overline{u}_{n}(0)^{*}\alpha\right|-T\right|\geq\frac{\tilde{\psi}}{2}
\]
for any $T\in\mathcal{P}$ with $T\leq\tilde{E}_{0}$. By Lemma \ref{lem:For-every--1},
we have for $\delta=1$, 
\begin{equation}
\sup_{(s,t)\in[-R_{n}+1,R_{n}-1]\times S^{1}}\left\Vert d\overline{u}_{n}(s,t)\right\Vert _{g_{\text{eucl.}},\overline{g}_{P_{n}s}}\leq\tilde{C}_{1}.\label{eq:gradbdd-1}
\end{equation}

Since the Riemannian metrics $ \overline{g}_{P_{n}s}$ are all equivalent by Remark \ref{rem:The-parameter-dependent}
we obtain 
\[
\sup_{(s,t)\in[-R_{n}+1,R_{n}-1]\times S^{1}}\left\Vert d\overline{u}_{n}(s,t)\right\Vert _{g_{\text{eucl.}},\overline{g}_{0}}\leq\tilde{C}_{1}\overline{C}_{1}.
\]
After suitable shifts in the $\mathbb{R}$-direction in $\mathbb{R}\times M$ and passing to a subsequence, the maps $\overline{u}_{n}$ converge in $C_{\text{loc}}^{\infty}$
to some $C^\infty$-map $\overline{u}=(\overline{a},\overline{f}):\mathbb{R}\times S^{1}\rightarrow\mathbb{R}\times M$. This convergence together with $|P_n s|\leq C|s|/R_n$ from $\overline{A}3$
imply that for each fixed $(s,t)$ the endomorphism $J_{P_n s}(\overline{f}_n(s,t))\in \text{End}(T(\mathbb{R}\times M))_{\overline{u}_n(s,t)}$ converges to $J_0(\overline{f}(s,t))$.
Thus 

\begin{enumerate}
\begin{singlespace}
\item $\overline{u}$ solves 
\[
\pi_{\alpha}d\overline{f}(z)\circ i=J_{0}(\overline{f}(z))\circ\pi_{\alpha}d\overline{f}(z)\text{ and }\overline{f}^{*}\alpha\circ i=d\overline{a};
\]
\item $E(\overline{u};\mathbb{R}\times S^{1})\leq\tilde{E}_{0}$, $E_{d\alpha}(\overline{u};\mathbb{R}\times S^{1})\leq\hbar/2$ by $\overline{A}0$ and $\overline{A}2$ for $\overline{u}_n$ and Fatou's lemma,
and 
\[
\left|\left|\int_{S^{1}}\overline{u}(0)^{*}\alpha\right|-T\right|\geq\frac{\tilde{\psi}}{2}
\]
for all $T\in\mathcal{P}$ with $T\leq\tilde{E}_{0}$ by the $C^\infty$-convergence $\overline{u}_n\rightarrow \overline{u}$ on $\{0\}\times S^1$. 
\end{singlespace}
\end{enumerate}
\begin{singlespace}
\noindent The rest of the proof proceeds as in the proof of Theorem
1.1 from \cite{key-11}. For the sake of completeness
we present this proof in detail. The map $\overline{u}$ can be regarded
as a finite energy map defined on a $2-$punctured Riemann sphere.
A puncture is removable or has a periodic orbit on the Reeb vector
field as asymptotic limit. In both cases, the limits 
\[
\lim_{s\rightarrow\pm\infty}\int_{S^{1}}\overline{u}(s)^{*}\alpha\in\mathbb{R}
\]
exist. The limit is equal to $0$ if the puncture is removable, and
equal to the period of the asymptotic limit if this is not the case.
As a result and by means of Stokes' theorem, the $d\alpha-$energy
of $\overline{u}$ can be written as 
\[
\int_{\mathbb{R}\times S^{1}}\overline{u}^{*}d\alpha=T_{2}-T_{1},
\]
with $T_{2}\geq T_{1}$, where $T_{1},T_{2}\in\mathcal{P}$ and $T_{1},T_{2}\leq\tilde{E}_{0}$.
Since $\overline{u}_n$ satisfy $\overline{A}2$ we have $E_{d\alpha}(\overline{u};\mathbb{R}\times S^1)\leq\hbar/2$,
and from the definition of the constant $\hbar$ we conclude that
that $T_{1}=T_{2}$. Set $T:=T_{1}=T_{2}$. If $T=0$, both punctures
are removable, $\overline{u}$ has an extension to a $J_{0}-$holomorphic
finite energy sphere $S^{2}\rightarrow\mathbb{R}\times M$, and so,
the map $\overline{u}$ must be constant; hence 
\[
\left|\left|\int_{S^{1}}\overline{u}(0)^{*}\alpha\right|-T\right|=0<\frac{\tilde{\psi}}{2},
\]
a contradiction. If $T>0$,
the finite energy cylinder $\overline{u}$ is non-constant, has a
vanishing $d\alpha-$energy, and so, $\overline{u}$ must be a cylinder
over a periodic orbit $x(t)$ of the form $\overline{u}(\pm(s,t))=(Ts+c,x(Tt+d))$
for some constants $c$ and $d$, and with a period $T\leq\tilde{E}_{0}$;
hence
\[
\left|\left|\int_{S^{1}}\overline{u}(0)^{*}\alpha\right|-T\right|=0<\frac{\tilde{\psi}}{2},
\]
a contradiction. Thus,
there exists a constant $h_{0}>0$ such that for any $\overline{J}_{P}-$holomorphic
curve $(\overline{u},R,P)$ with $R>h_{0}$ satisfying the energy
estimates, the center loop $\overline{u}(0,\cdot)$ has an action
close to an element $T\in\mathcal{P}$ with $T\leq\tilde{E}_{0}$,
i.e. 
\begin{equation}
\left|\left|\int_{S^{1}}\overline{u}(0)^{*}\alpha\right|-T\right|<\frac{\psi}{2}.\label{eq:estimate}
\end{equation}
To deal with the uniqueness issue, we consider two elements $T_{1},T_{2}\in\mathcal{P}$
with $T_{1},T_{2}\leq\tilde{E}_{0}$ satisfying the above estimate.
Then we have 
\[
|T_{1}-T_{2}|<\frac{\psi}{2}+\frac{\psi}{2}=\psi.
\]
By assumption, $\psi<\hbar/2$, and from the definition of $\hbar$
it follows that $T_{1}=T_{2}$. Therefore there is a unique $T\in\mathcal{P}$
satisfying $T\leq\tilde{E}_{0}$ and the estimate (\ref{eq:estimate}).
\end{singlespace}
\end{proof}

\section{\label{subsec:Zero-Center-Action}Vanishing center action}

\begin{singlespace}
\noindent Let $\overline{u}_n$ be a sequence as in Theorem \ref{thm:With-the-same} with vanishing center action (as defined between Theorem
\ref{thm:For-all-numbers} and Remark \ref{rem:From-the-definition}).
We use a version of the monotonicity lemma (Corollary \ref{cor:There-exists-constants-2})
to characterize the behavior of a $\overline{J}_{P}-$holomorphic
curve $(\overline{u},P,R)$ (Theorem \ref{thm:Let--and}). Using these
results we describe the convergence of a sequence of $\overline{J}_{P}-$holomorphic
cylinders (Theorem \ref{thm:With-the-same-3}) and then prove Theorem
\ref{thm:With-the-same}. 
\end{singlespace}

\subsection{Behaviour of $\overline{J}_{P}-$holomorphic curves with vanishing
center action}

\begin{singlespace}
\noindent The following result is an adapted version of Lemma 2.1
from \cite{key-11}: 
\end{singlespace}
\begin{lem}
\begin{singlespace}
\noindent \label{lem:Recall-the-constants}For all $\delta>0$ there
exists $\overline{h}_{0}>0$ such that for any $R>\overline{h}_{0}$
and any $\overline{J}_{P}-$holomorphic curve $(\overline{u},R,P)$ satisfying $\overline{A}0$-$\overline{A}3$
and having vanishing center action, the loops $\overline{u}(s)$
satisfy 
\begin{equation}
\text{diam}_{\overline{g}_{0}}(\overline{u}(s))\leq\delta\ \text{ and }\ |\alpha(\partial_{t}\overline{u}(s))|\leq\delta\label{eq:equation1}
\end{equation}
for all $s\in[-R+\overline{h}_{0},R-\overline{h}_{0}]$. 
\end{singlespace}
\end{lem}
\begin{proof}
\begin{singlespace}
\noindent The proof is similar to that given in \cite{key-11}. Nevertheless,
for the sake of completeness we sketch it here.
By setting $\psi:=\hbar/4$ and using Theorem \ref{thm:For-all-numbers} we obtain some $h_0>0$ such that 
\[
\left|\left|\int_{S^{1}}\overline{u}(0)^{*}\alpha\right|-0\right|<\frac{\psi}{2},
\] for any $\overline{J}_{P}-$holomorphic curve $(\overline{u},R,P)$ satisfying $\overline{A}0$-$\overline{A}3$ with vanishing center action and $R>h_0$. To show that 
we have $\text{diam}_{\overline{g}_{0}}(\overline{u}(s))\leq\delta $ for $s\in [-R+\overline{h}_0,R-\overline{h}_0]$ for sufficiently large
$\overline{h}_0$, we argue by contradiction. Thus assume there is a constant $\delta_{0}>0$, a sequence
$R_{n}\geq h_{n}:=n+h_{0}$, and a sequence of $\overline{J}_{P}-$holomorphic
curves $(\overline{u}_{n},R_{n},P_{n})$ such that 
\begin{align*}
E(\overline{u}_{n};[-R_{n},R_{n}]\times S^{1}) & \leq\tilde{E}_{0},\\
E_{d\alpha}(\overline{u}_{n};[-R_{n},R_{n}]\times S^{1}) & \leq\frac{\hbar}{2},\\
\left|\int_{S^{1}}\overline{u}_{n}(0)^{*}\alpha\right| & \leq\frac{\psi}{2},\\
\text{diam}_{\overline{g}_{0}}(\overline{u}_{n}(s_{n})) & \geq\delta_{0}
\end{align*}
for a sequence $s_{n}\in[-R_{n}+n+h_{0},R_{n}-n-h_{0}]$. By
Stokes' theorem, we have 
\begin{align*}
\left|\int_{S^{1}}\overline{u}_{n}(s_{n})^{*}\alpha\right| & \leq\left|\int_{[s_{n},0]\times S^{1}}\overline{f}_{n}^{*}d\alpha\right|+\left|\int_{S^{1}}\overline{u}_{n}(0)^{*}\alpha\right|\\
 & \leq\frac{\hbar}{2}+\frac{\psi}{2}\\
 & \leq\frac{3\hbar}{4}.
\end{align*}
Now define the maps $\tilde{u}_{n}=(\tilde{a}_{n},\tilde{f}_{n}):[-R_{n}-s_{n},R_{n}+s_{n}]\times S^{1}\rightarrow\mathbb{R}\times M$
by 
\[
\tilde{u}_{n}(s,t):=(\overline{a}_{n}(s+s_{n},t),\overline{f}_{n}(s+s_{n},t)),
\]
for which, the above assumptions go over in 
\begin{align*}
E(\tilde{u}_{n};[-R_{n},R_{n}]\times S^{1}) & \leq\tilde{E}_{0},\\
E_{d\alpha}(\tilde{u}_{n};[-R_{n},R_{n}]\times S^{1}) & \leq\frac{\hbar}{2},\\
\left|\int_{S^{1}}\tilde{u}_{n}(0)^{*}\alpha\right| & \leq\frac{3\hbar}{4},\\
\text{diam}_{\overline{g}_{0}}(\tilde{u}_{n}(0)) & \geq\delta_{0}.
\end{align*}
As $s_{n}\in[-R_{n}+n+h_{0},R_{n}-n-h_{0}]$, we see that $|R_{n}+s_{n}|\rightarrow\infty$
and $|R_{n}-s_{n}|\rightarrow\infty$ as $n\rightarrow\infty$. Moreover,
$\tilde{u}_{n}$ satisfies the pseudoholomorphic curve equation
\begin{alignat*}{1}
\pi_{\alpha}d\tilde{f}_{n}(s,t)\circ i & =J_{P_{n}(s+s_{n})}(\tilde{f}_{n}(s,t))\circ\pi_{\alpha}d\tilde{f}_{n}(s,t),\\
\tilde{f}_{n}^{*}\alpha\circ i & =d\tilde{a}_{n}.
\end{alignat*}
For the new sequence 
\[
\tilde{v}_{n}(s,t)=(\tilde{b}_{n}(s,t),\tilde{f}_{n}(s,t))=(\tilde{a}_{n}(s,t)-\tilde{a}_{n}(0,0),\tilde{f}_{n}(s,t)),
\]
the $\mathbb{R}-$invariance of $J_{\tau}$ and of $\overline{g}_{0}$,
yields 
\begin{align*}
E(\tilde{v}_{n};[-R_{n}-s_{n},R_{n}-s_{n}]\times S^{1}) & \leq\tilde{E}_{0},\\
E_{d\alpha}(\tilde{v}_{n};[-R_{n}-s_{n},R_{n}-s_{n}]\times S^{1}) & \leq\frac{\hbar}{2},\\
\left|\int_{S^{1}}\tilde{v}_{n}(0)^{*}\alpha\right| & \leq\frac{3\hbar}{4},\\
\text{diam}_{\overline{g}_{0}}(\tilde{v}_{n}(0)) & \geq\delta_{0}
\end{align*}
and 
\begin{align*}
\pi_{\alpha}d\tilde{v}_{n}(s,t)\circ i & =J_{P_{n}(s+s_{n})}(\tilde{f}_{n}(s,t))\circ\pi_{\alpha}d\tilde{v}_{n}(s,t),\\
\tilde{v}_{n}^{*}\alpha\circ i & =d\tilde{b}_{n}.
\end{align*}
Since  $|s_n|\leq R_n$, we have $|P_ns_n|\leq C$ according to $\overline{A}3$; thus a subsequence of $P_n s_n$ converges to some $\tau\in [-C,C]$.
By the same bubbling-off argument as in the proof of Theorem \ref{thm:For-all-numbers},
a subsequence of $\tilde{v}_{n}$ converges (up to shifts in the $\mathbb{R}$-direction) in $C_{\text{loc}}^{\infty}$
to a map $\tilde{v}=(b,v):\mathbb{R}\times S^{1}\rightarrow\mathbb{R}\times M$, which is a usual $J_{\tau}-$holomorphic cylinder,
since $J_{P_{n}(s+s_{n})}(\tilde{f}_{n}(s,t))\rightarrow J_\tau(v(s,t))$ for each fixed $(s,t)\in \mathbb{R}\times S^1$. Since $\tilde{v}$ is a nonconstant 
$J_\tau$-holomorphic cylinder 
of finite Hofer energy, it is asymptotic to periodic Reeb orbits at the punctures and its $d\alpha$-energy can be expressed as the difference of the asymptotic periods $P_{\pm\infty}$.
Since $|P_\infty-P_{-\infty}|=E_{d\alpha}(\tilde{v};\mathbb{R}\times S^1)\leq \hbar/2$, it follows from the definition of $\hbar$ that $P_\infty=P_{-\infty}$ and thus $E_{d\alpha}(\tilde{v};\mathbb{R}\times S^1)=0$. Summarizing we have:
\begin{align*}
E_{\alpha}(\tilde{v};\mathbb{R}\times S^{1})+E_{d\alpha}(\tilde{v};\mathbb{R}\times S^{1}) & \leq\tilde{E}_{0},\\
E_{d\alpha}(\tilde{v};\mathbb{R}\times S^{1}) & =0,\\
\left|\int_{S^{1}}\tilde{v}(0)^{*}\alpha\right| & \leq\frac{3\hbar}{4},\\
\text{diam}_{\overline{g}_{0}}(\tilde{v}(0)) & \geq\delta_{0}.
\end{align*}
In particular, $\tilde{v}$ is a non-constant finite energy cylinder
having a vanishing $d\alpha-$energy. Hence $\tilde{v}$ is a cylinder
over a periodic orbit of period $0<T\leq\tilde{E}_{0}$. Consequently,
we obtain 
\[
\left|\int_{S^{1}}\tilde{v}(0)^{*}\alpha\right|=T\geq\hbar,
\]
contradicting the previous estimate $\left|\int_{S^{1}}\tilde{v}(0)^{*}\alpha\right|  \leq\frac{3\hbar}{4}$. Thus $\text{diam}_{\overline{g}_{0}}(\overline{u}(s))\leq\delta$
for all $s\in[-R+h,R-h]$. For $|\alpha(\partial_{t}\overline{u}(s))|\leq\delta$
we proceed analogously, and the proof is finished.
\end{singlespace}
\end{proof}
\begin{singlespace}
\noindent The next theorem characterizes the behavior of a $\overline{J}_{P}-$holomorphic
curve $(\overline{u},R,P)$ with vanishing center action. 
\end{singlespace}
\begin{thm}
\begin{singlespace}
\noindent \label{thm:Let--and}For any $\epsilon>0$ there exists
$h_{1}>0$ such that for any $R>h_{1}$ and any $\overline{J}_{P}-$holomorphic
curve $(\overline{u},R,P)$ satisfying $\overline{A}0$-$\overline{A}3$ and $A(\overline{u})=0$ we have
$\overline{u}([-R+h_{1},R-h_{1}]\times S^{1})\subset B_{\epsilon}^{\overline{g}_{0}}(\overline{u}(0,0))$. 
\end{singlespace}
\end{thm}
\begin{proof}
\begin{singlespace}
\noindent In the first part of the proof we employ exactly the same
arguments as in the proof of Theorem 1.2 from \cite{key-11}. With
$\epsilon>0$ as in the statement of the theorem, we choose
$0<r<\frac{\epsilon}{2}$ and $\delta>0$ sufficiently small such that 
\begin{equation}
8\delta<C_{8}r^{2}\text{ and }\delta+r\leq\frac{\epsilon}{2}.\label{eq:choice}
\end{equation}
where $C_{8}>0$ is the constant from the monotonicity
Lemma (Corollary \ref{cor:There-exists-constants-2}). For the $\overline{J}_{P}-$holomorphic
curve $(\overline{u},R,P)$ with $R>\overline{h}_{0}$ as in the Lemma
\ref{lem:Recall-the-constants} we have $\text{diam}_{\overline{g}_{0}}(\overline{u}(s))\leq\delta$
and $|\alpha(\partial_{t}\overline{f}(s))|\leq\delta$ for all $s\in[-R+\overline{h}_{0},R-\overline{h}_{0}]$.
In order to simplify notation denote $\overline{h}_{0}$
by $h$. The definition of the energy and Stokes' theorem give
\begin{align*}
E(\overline{u}|_{[-R+h,R-h]\times S^{1}};[-R+h,R-h]\times S^{1}) & =E_{\alpha}(\overline{u}|_{[-R+h,R-h]\times S^{1}};[-R+h,R-h]\times S^{1})\\
 & \qquad+E_{d\alpha}(\overline{u}|_{[-R+h,R-h]\times S^{1}};[-R+h,R-h]\times S^{1})\\
 & =\sup_{\varphi\in\mathcal{A}}\int_{[-R+h,R-h]\times S^{1}}\varphi'(\overline{a})d\overline{a}\circ j\wedge d\overline{a}\\
 & \qquad+\int_{[-R+h,R-h]\times S^{1}}\overline{f}^{*}d\alpha\\
 & =\sup_{\varphi\in\mathcal{A}}\int_{[-R+h,R-h]\times S^{1}}-\left[d(\varphi(\overline{a})d\overline{a}\circ j)-\varphi(\overline{a})d(d\overline{a}\circ j)\right]\\
 & \qquad+\int_{[-R+h,R-h]\times S^{1}}\overline{f}^{*}d\alpha\\
 & =\sup_{\varphi\in\mathcal{A}}\left[\int_{[-R+h,R-h]\times S^{1}}-d(\varphi(\overline{a})d\overline{a}\circ j)-\int_{[-R+h,R-h]\times S^{1}}\varphi(\overline{a})\overline{f}^{*}d\alpha\right]\\
 & \qquad+\int_{[-R+h,R-h]\times S^{1}}\overline{f}^{*}d\alpha\\
 & \leq\sup_{\varphi\in\mathcal{A}}\left|\int_{\{R-h\}\times S^{1}}\varphi(\overline{a})d\overline{a}\circ j-\int_{\{-R+h\}\times S^{1}}\varphi(\overline{a})d\overline{a}\circ j\right|\\
 & \qquad+\left[\int_{\{R-h\}\times S^{1}}\overline{f}^{*}\alpha-\int_{\{-R+h\}\times S^{1}}\overline{f}^{*}\alpha\right]\\
 & =\sup_{\varphi\in\mathcal{A}}\left|-\int_{\{R-h\}\times S^{1}}\varphi(\overline{a})\overline{f}^{*}\alpha+\int_{\{-R+h\}\times S^{1}}\varphi(\overline{a})\overline{f}^{*}\alpha\right|\\
 & \leq 2\left[\int_{\{R-h\}\times S^{1}}|\overline{f}^{*}\alpha|+\int_{\{-R+h\}\times S^{1}}|\overline{f}^{*}\alpha|\right]
\end{align*}
hence
\begin{equation}
E(\overline{u}|_{[-R+h,R-h]\times S^{1}};[-R+h,R-h]\times S^{1})\leq4\delta.\label{eq:monotonicity-ineq}
\end{equation}
If the conclusion of Theorem \ref{thm:Let--and} is not true for this
$h$, we find a point $(s_{0},t_{0})\in[-R+h,R-h]\times S^{1}$ for
which 
\[
\text{dist}_{\overline{g}_{0}}(\overline{u}(s_{0},t_{0}),\overline{u}(0,0))\geq\epsilon.
\]
From $\text{diam}_{\overline{g}_{0}}(\overline{u}(s))\leq\delta$
we obtain 
\[
\text{dist}_{\overline{g}_{0}}(\overline{u}(s_{0},t),\overline{u}(0,t'))\geq\epsilon-2\delta
\]
for all $t,t'\in S^{1}$. Choosing a point $s_{1}$ between $0$ and
$s_{0}$ such that 
\begin{align*}
\text{dist}_{\overline{g}_{0}}(\overline{u}(s_{1},t),\overline{u}(s_{0},t'))\geq\frac{\epsilon}{2}-\delta\ 
\text{ and }\ \text{dist}_{\overline{g}_{0}}(\overline{u}(s_{1},t),\overline{u}(0,t'))\geq\frac{\epsilon}{2}-\delta
\end{align*}
for all $t,t'\in S^{1}$, using $r\leq\epsilon/2-\delta$, and applying
the monotonicity Lemma \ref{cor:There-exists-constants-2} to the
open ball $B_{r}^{\overline{g}_{0}}(\overline{u}(s_{1},t_{1}))$,
we conclude that $$E(\overline{u};S\cap \overline{u}^{-1}(B^{\overline{g}_0}_r(\overline{u}(s_1,t_1)))\geq C_{8}r^{2}.$$
In view of (\ref{eq:monotonicity-ineq}), this implies that $C_{8}r^{2}\leq 4\delta$,
which contradicts the choice in (\ref{eq:choice}). Hence
$\overline{u}(s,t)\in B_{\epsilon}^{\overline{g}_{0}}(\overline{u}(0,0))$
for all $(s,t)\in[-R+h,R-h]\times S^{1}$ as claimed by Theorem \ref{thm:Let--and}. 
\end{singlespace}
\end{proof}

\subsection{\label{subsec:Convergence-of-holomorphic}Proof of Theorem \ref{thm:With-the-same}}

\begin{singlespace}
\noindent We are now nearly ready to describe the convergence and
the limit object of a sequence of $\mathcal{H}-$holomorphic cylinders $u_{n}$
satisfying the assumptions of Theorem \ref{thm:With-the-same}. However, befor coming to that, we prove
first an analogous result for a sequence of $J_{P_n}$-holomorphic cylinders $\overline{u}_n$. 

Thus consider a sequence $(\overline{u}_n=(\overline{a}_n,\overline{f}_n),R_n,P_n)$ of $\overline{J}_{P_{n}}-$holomorphic curves 
satisfying assumptions $\overline{A}0$-$\overline{A}3$ with $\mathcal{A}(\overline{u}_n)=0\, \forall n\in \mathbb{N}$ and $R_n\rightarrow \infty$.
For normalization purposes, we assume that $\overline{u}_n(0,0)\rightarrow w:=(0,w_f)$ for some $w_f\in M$
and that $R_n\rightarrow \infty$ (this can always be achieved by shifting in the $\mathbb{R}$-direction of $\mathbb{R}\times M$ and by passing to a subsequence).
\noindent By Theorem \ref{thm:Let--and} applied to the sequence of
$\overline{J}_{P_{n}}-$holomorphic curves $(\overline{u}_{n},R_{n},P_{n})$
we have the following 
\end{singlespace}
\begin{cor}
\begin{singlespace}
\noindent \label{cor:For-every-sequence}For every sequence $h_{n}\in\mathbb{R}_{>0}$
satisfying $h_{n}<R_{n}$ and $h_{n},R_{n}/h_{n}\rightarrow\infty$
and every $\epsilon>0$ there exists $N\in\mathbb{N}$ such that 
\[
\overline{u}_{n}([-R_{n}+h_{n},R_{n}-h_{n}]\times S^{1})\subset B_{\epsilon}^{\overline{g}_{0}}((0,w_f))
\]
for all $n\geq N$.
\end{singlespace}
\end{cor}
\begin{proof}
\begin{singlespace}
\noindent Consider a sequence $h_{n}\in\mathbb{R}_{>0}$ such that
$h_{n}<R_{n}$ and $h_{n},R_{n}/h_{n}\rightarrow\infty$ as $n\rightarrow\infty$
and let $\epsilon>0$ be given. From Theorem \ref{thm:Let--and} there
exists $h_{\epsilon}>0$ and $N_{\epsilon}\in\mathbb{N}$ such that
for all $n\geq N_{\epsilon}$, we have $R_{n}>h_{\epsilon}$ and $\overline{u}_{n}([-R_{n}+h_{\epsilon},R_{n}-h_{\epsilon}]\times S^{1})\subset B_{\epsilon}^{\overline{g}_{0}}(w)$.
By making $N_{\epsilon}$ sufficiently large and accounting of $h_{n}\rightarrow\infty$,
we may assume that for all $n\geq N_{\epsilon}$, we have that $R_{n}>h_{n}>h_{\epsilon}$,
which in turn gives $\overline{u}_{n}([-R_{n}+h_{n},R_{n}-h_{n}]\times S^{1})\subset B_{\epsilon}^{\overline{g}_{0}}(w)$. 
\end{singlespace}
\end{proof}
\begin{singlespace}
\noindent To describe the $C^{0}-$convergence of the maps $\overline{u}_{n}$
we define a sequence of diffeomorphisms, which is similar to that
constructed in Section 4.4 of \cite{key-10}. For a sequence $h_{n}\in\mathbb{R}_{>0}$
with $h_{n}<R_{n}$ and $h_{n},R_{n}/h_{n}\rightarrow\infty$ as $n\rightarrow\infty$,
let $\theta_{n}, \theta^\pm$ and $\theta_n^\pm$ be sequences of
diffeomorphisms with properties as in Remark \ref{rem:For-every-sequence}.
Given maps $\overline{u},\overline{u}_n^\pm,\overline{u}^\pm$ resp. $u,u_n^\pm,u^\pm$, we define maps 
\begin{align}
\begin{array}{rl}
\overline{v}_{n}(s,t) & =\overline{u}_{n}\circ\theta_{n}^{-1}(s,t),\ s\in[-1,1],\\
\overline{v}_{n}^{-}(s,t) & =\overline{u}_{n}^{-}\circ(\theta_{n}^{-})^{-1}(s,t),\ s\in[-1,-{\normalcolor 1/2]},\\
\overline{v}_{n}^{+}(s,t) & =\overline{u}_{n}^{+}\circ (\theta_{n}^{+})^{-1}(s,t),\ s\in[1/2,1],\\
\overline{v}^{-}(s,t) & =\overline{u}^{-}\circ(\theta^{-})^{-1}(s,t),\ s\in[-1,-1/2),\\
\overline{v}^{+}(s,t) & =\overline{u}^{+}\circ (\theta^{+})^{-1}(s,t),\ s\in(1/2,1],
\end{array}\label{eq:VkBar}
\end{align}
and 
\begin{align}
\begin{array}{rl}
v_{n}(s,t) & =u_{n}\circ \theta_{n}^{-1}(s,t),\ s\in[-1,1],\\
v_{n}^{-}(s,t) & =u_{n}^{-}\circ(\theta_{n}^{-})^{-1}(s,t),\ s\in[-1,-1/2],\\
v_{n}^{+}(s,t) & =u_{n}^{+}\circ(\theta_{n}^{+})^{-1}(s),t),\ s\in[1/2,1],\\
v^{-}(s,t) & =u^{-}\circ(\theta^{-})^{-1}(s,t),\ s\in[-1,-1/2),\\
v^{+}(s,t) & =u^{+}\circ(\theta^{+})^{-1}(s,t),\ s\in(1/2,1],
\end{array}\label{eq:Vk}
\end{align}
where, $\overline{u}_{n}^{\pm}(s,t)=\overline{u}_{n}(s\pm R_{n},t)$
and $u_{n}^{\pm}(s,t)=u_{n}(s\pm R_{n},t)$ are the left and right
shifts of the maps $\overline{u}_{n}$ and $u_{n}$, respectively.

\noindent The next theorem states a $C_{\text{loc}}^{\infty}-$ and
a $C^{0}-$convergence result for the maps $\overline{u}_{n}$.
\end{singlespace}
\begin{thm}
\begin{singlespace}
\noindent \label{thm:With-the-same-3}Let $(\overline{u}_{n},R_{n},P_{n})$ be a sequence of $J_{P_n}$-holomorphic curves satisfying $\overline{A}0$-$\overline{A}3$ with vanishing center action
and such that $\overline{u}_n(0,0)\rightarrow w:=(0,w_f)$, $R_n\rightarrow \infty$ and $R_nP_n\rightarrow \tau$. Then there exists a subsequence,
denoted again by $(\overline{u}_{n},R_{n},P_{n}),$ and pseudoholomorphic
half cylinders $\overline{u}^{\pm}$ defined on $(-\infty,0]\times S^{1}$
and $[0,\infty)\times S^{1}$ respectively, such that for every sequence
$h_{n}\in\mathbb{R}_{>0}$ and every sequence of diffeomorphisms $\theta_{n}:[-R_{n},R_{n}]\times S^{1}\rightarrow[-1,1]\times S^{1}$
satisfying the assumptions of Remark \ref{rem:For-every-sequence},
the following convergence results hold:

\noindent $C_{\text{loc}}^{\infty}-$convergence: 
\end{singlespace}
\begin{enumerate}
\begin{singlespace}
\item For any sequence $s_{n}\in[-R_{n}+h_{n},R_{n}-h_{n}]$ the shifted
maps $\overline{u}_{n}(s+s_{n},t)$, defined on $[-R_{n}+h_{n}-s_{n},R_{n}-h_{n}-s_{n}]\times S^{1}$,
converge in $C_{\text{loc}}^{\infty}$ on $\mathbb{R}\times S^1$ to the constant map $(s,t)\mapsto w$. 
\item The left shifts $\overline{u}_{n}^{-}(s,t):=\overline{u}_{n}(s-R_{n},t)$,
defined on $[0,h_{n})\times S^{1}$, converge
in $C_{\text{loc}}^{\infty}$ to a pseudoholomorphic half cylinder
$\overline{u}^{-}=(\overline{a}^{-},\overline{f}^{-})$ defined on
$[0,+\infty)\times S^{1}$. We have $\lim_{s\rightarrow \infty}\overline{u}^{-}(s,\cdot)=w$ in $C^\infty(S^1,\mathbb{R}\times M)$. 
The maps $\overline{v}_{n}^{-}:[-1,-1/2)\times S^{1}\rightarrow\mathbb{R}\times M$
converge in $C_{\text{loc}}^{\infty}$ to $\overline{v}^{-}:[-1,-1/2)\times S^{1}\rightarrow\mathbb{R}\times M$
such that $\lim_{s\rightarrow -1/2}\overline{v}^{-}(s,\cdot)=w$ in $C^\infty(S^1,\mathbb{R}\times M)$. 
\item The right shifts $\overline{u}_{n}^{+}(s,t):=\overline{u}_{n}(s+R_{n},t)$,
defined on $(-h_{n},0]\times S^{1}$, converge
in $C_{\text{loc}}^{\infty}$ to a pseudoholomorphic half cylinder
$\overline{u}^{+}=(\overline{a}^{+},\overline{f}^{+})$ defined
on $(-\infty,0]\times S^{1}$. We have $\lim_{s\rightarrow -\infty}\overline{u}^{+}(s,\cdot)=w$ in $C^\infty(S^1,\mathbb{R}\times M)$. 
The maps $\overline{v}_{n}^{+}:(1/2,1]\times S^{1}\rightarrow\mathbb{R}\times M$
converge in $C_{\text{loc}}^{\infty}$ to $\overline{v}:(1/2,1]\times S^{1}\rightarrow\mathbb{R}\times M$
such that $\lim_{s\rightarrow 1/2}\overline{v}^{+}(s,\cdot)=w$ in $C^\infty(S^1,\mathbb{R}\times M)$. 
\end{singlespace}
\end{enumerate}
\begin{singlespace}
\noindent $C^{0}-$convergence: 
\end{singlespace}
\begin{enumerate}
\begin{singlespace}
\item The maps $\overline{v}_{n}:[-1/2,1/2]\times S^{1}\rightarrow\mathbb{R}\times M$
converge in $C^{0}$ to the constant map $\overline{v}(s,t)=w$. 
\item The maps $\overline{v}_{n}^{-}:[-1,-1/2]\times S^{1}\rightarrow\mathbb{R}\times M$
converge in $C^{0}$ to $\overline{v}^{-}:[-1,-1/2]\times S^{1}\rightarrow\mathbb{R}\times M$
such that $\overline{v}^{-}(-1/2,t)=w$. 
\item The maps $\overline{v}_{n}^{+}:[1/2,1]\times S^{1}\rightarrow\mathbb{R}\times M$
converge in $C^{0}$ to $\overline{v}^{+}:[1/2,1]\times S^{1}\rightarrow\mathbb{R}\times M$
such that $\overline{v}(1/2,t)=w$. 
\end{singlespace}
\end{enumerate}
\end{thm}
\begin{proof}
\begin{singlespace}
\noindent We prove only the first and second statements of the $C_{\text{loc}}^{\infty}$-
and $C^{0}$- convergences because the proofs of the third statements
are exactly the same as those of the second statements. For the
sequence $h_{n}\in\mathbb{R}_{>0}$ with the property $h_{n},R_{n}/h_{n}\rightarrow\infty$
as $n\rightarrow\infty$. Let $s_{n}\in[-R_{n}+h_{n},R_{n}-h_{n}]$ be an arbitrary sequence.
Then shifted maps $\overline{u}_{n}(\cdot+s_{n},\cdot)$, defined on
$[-R_{n}+h_{n}-s_{n},R_{n}-h_{n}-s_{n}]\times S^{1}$, converge
in $C_{\text{loc}}^{\infty}$ to $w$. Indeed, due to Corollary \ref{cor:For-every-sequence}
and Lemma \ref{lem:For-every--1}, any subsequence has a further subsequence
converging to $w.$ To prove the second statement of the $C_{\text{loc}}^{\infty}-$convergence
we consider the shifted maps $\overline{u}_{n}^{-}:[0,h_{n}]\times S^{1}\rightarrow\mathbb{R}\times M$,
defined by $\overline{u}_{n}^{-}(s,t)=\overline{u}_{n}(s-R_{n},t)$.
By Lemma \ref{lem:For-every--1}, these maps have bounded gradients,
and hence, after going over to some subsequence, (at least after suitable shifts in the $a$-coordinate) they converge in
$C_{\text{loc}}^{\infty}$ on $[0,\infty)\times S^{1}$ to a pseudoholomorphic
curve $\overline{u}^{-}:[0,+\infty)\times S^{1}\rightarrow\mathbb{R}\times M$
with respect to the standard complex structure $i$ on $[0,+\infty)\times S^{1}$
and the parameter-independent almost complex structure $J_{-\tau}$ on the target $\mathbb{R}\times M$. Here
we use that $|P_nh_n|\leq h_n/R_n\rightarrow 0$ and that
$\tau:=\lim_{n\rightarrow \infty} P_{n}R_{n}$. 
However, since we know that $u_n^-(h_n)\rightarrow w$ from the first part, there are actually no shifts in the $a$-coordinate needed.
Let
us show that $\overline{u}^{-}$ is asymptotic to $w\in\mathbb{R}\times M$,
i.e. let us show that $\lim_{r\rightarrow\infty}\overline{u}^{-}(r,\cdot)=w$ in $C^\infty(S^1,\mathbb{R}\times M)$. Since we have uniform gradient bounds, it is sufficient
to show that this limit holds pointwise (any subsequence has a further subsequence converging in $C^\infty$ to the pointwise limit).
We prove pointwise convergence by contradiction: Assume that there exists a sequence $(s_{k},t_{k})\in[0,\infty)\times S^{1}$
with $s_{k}\rightarrow\infty$ as $k\rightarrow\infty$ such that
$\lim_{k\rightarrow\infty}\overline{u}^{-}(s_{k},t_{k})=w'\in\mathbb{R}\times M$
with $w'\not=w$. Let $\epsilon:=\text{dist}_{\overline{g}_{0}}(w,w')>0$.
For any $k\in\mathbb{N}$ there exists $N_{k}\in\mathbb{N}$ such
that for any $n\geq N_{k}$, $(s_{k},t_{k})\in[0,h_{n}]$. Thus for
arbitrary $k$ and $n$ such that $n\geq N_{k}$ we have 
\begin{align*}
\text{dist}_{\overline{g}_{0}}(w,w') & \leq\text{dist}_{\overline{g}_{0}}(w,\overline{u}_{n}^{-}(s_{k},t_{k}))+\text{dist}_{\overline{g}_{0}}(\overline{u}_{n}^{-}(s_{k},t_{k}),\overline{u}^{-}(s_{k},t_{k}))\\
 & +\text{dist}_{\overline{g}_{0}}(\overline{u}^{-}(s_{k},t_{k}),w').
\end{align*}
By Theorem \ref{thm:Let--and}, there exists $h>0$ such that $\text{dist}_{\overline{g}_{0}}(\overline{u}_{n}^{-}(s,t),w)<\epsilon/10$
for all $(s,t)\in[h,h_{n}]\times S^{1}$. Now choose $k$ and $n\geq N_{k}$
sufficiently large such that $(s_{k},t_{k})\in[h,h_{n}]\times S^{1}$.
Hence, $\text{dist}_{\overline{g}_{0}}(\overline{u}_{n}^{-}(s_{k},t_{k}),w)<\epsilon/10$.
Making $k$ and $n\geq N_{k}$ larger we may also assume that $\text{dist}_{\overline{g}_{0}}(\overline{u}^{-}(s_{k},t_{k}),w')<\epsilon/10$.
After fixing $k$ and making $n\geq N_{k}$ sufficiently large we
get $\text{dist}_{\overline{g}_{0}}(\overline{u}_{n}^{-}(s_{k},t_{k}),\overline{u}^{-}(s_{k},t_{k}))<\epsilon/10$.
As a result, we find $\text{dist}_{\overline{g}_{0}}(w,w')\leq3\epsilon/10$,
which is a contradiction to $\text{dist}_{\overline{g}_{0}}(w,w')=\epsilon$.\\
For any sequence of diffeomorphisms
$\theta_{n}:[-R_{n},R_{n}]\times S^1\rightarrow[-1,1]\times S^1$ fulfilling the assumptions
of Remark \ref{rem:For-every-sequence},
the maps $\overline{v}_{n}^{-}(s,t)=\overline{u}_{n}^{-}\circ(\theta_{n}^{-})^{-1}(s,t)$
also converge now in $C_{\text{loc}}^{\infty}$ to the map $\overline{v}^{-}(s,t)=\overline{u}^{-}\circ(\theta^{-})^{-1}(s,t)$.
This follows from the fact that $(\theta_{n}^{-})^{-1}:[-1,-1/2]\times S^1\rightarrow[0,h_{n}]\times S^1$
converge in $C_{\text{loc}}^{\infty}$ to the diffeomorphism $(\theta^{-})^{-1}:[-1,-1/2)\times S^1\rightarrow[0,+\infty)\times S^1$.
By the asymptotics of $\overline{u}^{-}$, $\overline{v}^{-}$ can
be continuously extended to the compact cylinder $[-1,-1/2]\times S^1$ by setting
$\overline{v}^{-}(-1/2,t)=w$. This finishes the proof of the second
statement, and hence of the $C_{\text{loc}}^{\infty}-$convergence.

\noindent Now we consider the first statement of the $C^{0}-$convergence.
From Corollary \ref{cor:For-every-sequence} it follows that $\text{dist}_{\overline{g}_{0}}(\overline{v}_{n}(s,t),w)\rightarrow0$
as $n\rightarrow\infty$ for all $(s,t)\in[-1/2,1/2]\times S^{1}$,
and the proof of the first statement is complete. The proof of the
second statement of the $C^{0}-$convergence is exactly the same as
the proof of Lemma 4.16 in \cite{key-10}.\textbf{ } 
\end{singlespace}
\end{proof}
\begin{singlespace}
\noindent Now we are in the position to prove Theorem \ref{thm:With-the-same}. 
\end{singlespace}
\begin{proof}
\begin{singlespace}
\noindent \emph{(of Theorem \ref{thm:With-the-same})}

Consider a sequence of $\mathcal{H}-$holomorphic cylinders $u_{n}=(a_{n},f_{n}):[-R_{n},R_{n}]\times S^{1}\rightarrow\mathbb{R}\times M$
satisfying the assumptions of Theorem \ref{thm:With-the-same}.
That is, they satisfy $A0$-$A3$ with harmonic
perturbation $1-$forms $\gamma_{n}$, $R_{n}\rightarrow\infty$ and they have vanishing center action.
As in Section \ref{subsec:Notion-of-the} we transform the map $u_{n}$
into a $\overline{J}_{P_n}-$holomorphic curve $\overline{u}_{n}$ with
respect to the domain-dependent almost complex structure $\overline{J}_{P_n}$.
We consider the new sequence of maps $\overline{f}_{n}$ defined by
$\overline{f}_{n}(s,t):=\phi_{P_{n}s}^{\alpha}(f_{n}(s,t))$ for all
$n\in\mathbb{N}$. Thus $\overline{u}_{n}=(\overline{a}_{n},\overline{f}_{n}):[-R_{n},R_{n}]\times S^{1}\rightarrow\mathbb{R}\times M$
is a $\overline{J}_{P_{n}}-$holomorphic curve. 
Due to Remark \ref{rem:In-the-following} the triple $(\overline{u}_{n},R_{n},P_{n})$
is a $\overline{J}_{P_{n}}-$holomorphic curve as in Definition \ref{def:A-triple-} satisfying $\overline{A}0$-$\overline{A}3$ with vanishing center action.
After shifting $u_{n}$ by $-a_{n}(0,0)$ we obtain that $a_{n}(0,0)=0$.\\
Now by (\ref{eq:bar=00003D00007Ba=00003D00007D})
$a_{n}$ can be written as $a_{n}=\overline{a}_{n}-\Gamma_{n}$, where
$\Gamma_{n}:[-R_{n},R_{n}]\times S^{1}\rightarrow\mathbb{R}$ is a
normalized harmonic function with $\|d\Gamma_n\|^2_{L^2([-R_n,R_n]\times S^1)}\leq (\sqrt{C_0}+\sqrt{2|P_n|C})^2$ (c.f. Remark \ref{rem:Obviously,-as-}) and $R_n\rightarrow \infty$. 
Thus the functions $\Gamma_n$ satisfy (if $R_n\geq 1$) the assumptions $C1$-$C5$ from
Appendix \ref{sec:Convergence-of} with $B=(\sqrt{C_0}+\sqrt{2C^2})^2$, since $|P_n|\leq C$ by $A3$ if $R_n\geq 1$. 
We can thus by Lemma \ref{lem:decompose_harmonic} further write $\Gamma_n=S_ns+\tilde{\Gamma}_n$, where the harmonic functions $\tilde{\Gamma}_n$ 
have uniform $C^k$-bounds on $[-R_n+1,R_n-1]\times S^1$ 
by Lemma \ref{lem:kbounds_tilde_Gamma}.\\
Actually Proposition \ref{prop:For-every-} shows that $\tilde{\Gamma}_n(0,0)\rightarrow 0$ and thus $\overline{a}_{n}(0,0)\rightarrow 0$.
Therefore we can, after passing
to a subsequence and by using that $M$ is compact, assume that $\overline{u}_{n}(0,0)\rightarrow w=(0,w_{f})\in\mathbb{R}\times M$
as $n\rightarrow\infty$.\\
We pick a subsequence, such that the convergence results from Theorem \ref{thm:With-the-same-3} hold for this sequence $\overline{u}_n$ and such that
that $P_{n}R_{n}\rightarrow\tau$ and $S_{n}R_{n}\rightarrow\sigma$
for $\tau,\sigma\in\mathbb{R}$ (the latter statements are possible by assumption A3).

As before,
we focus only on the proofs of the first and second statements of
the $C_{\text{loc}}^{\infty}-$ and $C^{0}-$convergences, because
the proofs of the third statements are similar to those of the second
statements. For the sequence $h_{n}\in\mathbb{R}_{>0}$ with the property
$h_{n},R_{n}/h_{n}\rightarrow\infty$ as $n\rightarrow\infty$, consider
the sequence of diffeomorphisms $\theta_{n}:[-R_{n},R_{n}]\rightarrow[-1,1]$
fulfilling the assumptions of Remark \ref{rem:For-every-sequence}.
By the construction described in Section \ref{subsec:Notion-of-the},
we have 
\[
\overline{f}_{n}(s,t)=\phi_{P_{n}s}^{\alpha}(f_{n}(s,t))\text{ and }d\overline{a}_{n}=d\Gamma_{n}+da_{n},
\]
where $(s,t)\in[-R_{n}+h_{n},R_{n}-h_{n}]\times S^{1}$ and $\Gamma_{n}:[-R_{n},R_{n}]\times S^{1}\rightarrow\mathbb{R}$
is a sequence of harmonic functions such that $d\Gamma_{n}$ has a
uniformly bounded $L^{2}-$norm. Then we obtain 
\begin{equation}
f_{n}(s,t)=\phi_{-P_{n}s}^{\alpha}(\overline{f}_{n}(s,t))\text{ and }a_{n}(s,t)=\overline{a}_{n}(s,t)-\Gamma_{n}(s,t).\label{eq:f-part-1}
\end{equation}
For the sequence of harmonic functions $\Gamma_{n}(s,t)$, the $L^{2}-$norms
of $d\Gamma_{n}$ are uniformly bounded, while by Remark \ref{rem:Obviously,-as-},
the functions $\Gamma_{n}$ can be chosen to have vanishing average.\\
For any sequence $s_{n}\in[-R_{n}+h_{n},R_{n}-h_{n}]$
for which $s_{n}/R_{n}\rightarrow\kappa\in[-1,1]$, the shifted map
$u_{n}(\cdot+s_{n},\cdot)$ defined on $[-R_{n}+h_{n}-s_{n},R_{n}-h_{n}-s_{n}]\times S^{1}$,
coincides with $(\overline{a}_n(\cdot+s_n,\cdot)-\Gamma_n(\cdot+s_n,\cdot),\phi_{-P_{n}(\cdot+s_n)}^{\alpha}(\overline{f}(\cdot+s_n,\cdot))$, as recalled above.
By the first $C^\infty_{\text{loc}}$-convergence result 
in Theorem \ref{thm:The-maps--1}, the functions $\Gamma_n(\cdot+s_n,\cdot)$ converge to the constant function $\sigma \kappa$.
Combining this with the first $C^\infty_{\text{loc}}$-convergence result of Theorem \ref{thm:With-the-same-3} it follows that
the shifted maps converge in $C_{\text{loc}}^{\infty}$ to the constant map $(s,t)\mapsto (-\sigma\kappa,\phi_{-\tau\kappa}^{\alpha}(w_{f}))$ on $\mathbb{R}\times S^1$.\\
The shifted harmonic $1-$form defined on $[-R_{n}+h_{n}-s_{n},R_{n}-h_{n}-s_{n}]\times S^{1}$
takes the form $\gamma_{n}(s+s_{n},t)=d\Gamma_{n}(s+s_{n},t)+P_{n}dt$ by Remark \ref{rem:Obviously,-as-}.
Since $|P_n|\leq C/R_n\rightarrow 0$ and, as before $\Gamma_n(s+s_n,\cdot)$ converges in $C^\infty_{loc}$ to the constant $\sigma\kappa$ by the first $C^\infty_{\text{loc}}$-convergence result 
in Theorem \ref{thm:The-maps--1}, 
we see $\gamma_{n}(s+s_{n},t)\rightarrow0$
in $C_{\text{loc}}^{\infty}$ as $n\rightarrow\infty$. This completes
the proof of the first statement. \\
To prove the second statement of
the $C_{\text{loc}}^{\infty}-$convergence we transfer the convergence
results for the shifted maps $\overline{u}_{n}^{-}:[0,h_{n}]\times S^{1}\rightarrow\mathbb{R}\times M$,
$\overline{u}_{n}^{-}(s,t)=\overline{u}_{n}(s-R_{n},t)$ of Theorem
\ref{thm:With-the-same-3} to the maps $u_{n}^{-}$, and use the convergence
results of the harmonic functions established in Theorem \ref{thm:The-maps--1}
of Appendix \ref{sec:Convergence-of}. The maps $\overline{u}_n^-$ are obtained from the shifted maps $u_{n}^{-}=(a_{n}^{-},f_{n}^{-}):[0,h_{n}]\times S^{1}\rightarrow\mathbb{R}\times M$ 
given by $u_n^-(s,t)=u_n(s-R_n,t)$, by applying the definitions (\ref{eq:bar=00003D00007Bf=00003D00007D}) and (\ref{eq:bar=00003D00007Ba=00003D00007D});
thus the components $\overline{u}_n^-$ satisfy
\begin{equation}
\overline{f}_{n}^{-}(s,t)=\phi_{P_{n}(s-R_{n})}^{\alpha}(f_{n}^{-}(s,t))\text{ and }\overline{a}_{n}^{-}(s,t)=a_{n}^{-}(s,t)+\Gamma_{n}^{-}(s,t),\label{eq:fnminusbar}
\end{equation}
where $\Gamma_{n}^{-}:[0,h_{n}]\times S^{1}\rightarrow\mathbb{R}$
is the left shifted harmonic function, defined by $\Gamma_{n}^{-}(s,t)=\Gamma_{n}(s-R_{n},t)$.
Hence we obtain 
\begin{equation}
f_{n}^{-}(s,t)=\phi_{-P_{n}(s-R_{n})}^{\alpha}(\overline{f}_{n}^{-}(s,t))\text{ and }a_{n}^{-}(s,t)=\overline{a}_{n}^{-}(s,t)-\Gamma_{n}^{-}(s,t).\label{eq:fnminus}
\end{equation}
Thus, by the second $C^\infty_{loc}$-convergence results in Theorem \ref{thm:With-the-same-3} and Theorem \ref{thm:The-maps--1} 
(the latter implying that the harmonic functions $\Gamma_n^-$ converge to a harmonic function $\Gamma^-$ with $\lim_{s\rightarrow \infty}\Gamma^-(s,\cdot)=-\sigma$),
the maps $u_{n}^{-}$ converge in $C_{\text{loc}}^{\infty}$ to
a curve $u^{-}(s,t)=(a^{-}(s,t),f^{-}(s,t))=(\overline{a}^{-}(s,t)-\Gamma^{-}(s,t),\phi_{\tau}^{\alpha}(\overline{f}^{-}(s,t)))$,
defined on $[0,\infty)\times S^{1}$. The map $u^{-}$ is asymptotic
to $(\sigma,\phi_{\tau}^{\alpha}(w_{f}))$, and it is (as the limit of the maps $u_n^-$)
a $\mathcal{H}-$holomorphic map with harmonic perturbation $d\Gamma^{-}$.
This finishes the proof of the second statement. The third statement
is proven analogously.

\noindent To prove the first statement of the $C^{0}-$convergence,
we consider the maps $v_{n}$ and recall that 
\[
\overline{f}_{n}(s,t)=\phi_{P_{n}s}^{\alpha}(f_{n}(s,t)),\,\overline{a}_{n}(s,t)=a_{n}(s,t)+\Gamma_{n}(s,t),
\]
and 
\begin{align*}
v_{n}(s,t)=(\overline{a}_{n}\circ\theta_{n}^{-1}(s,t)-\Gamma_{n}\circ\theta_{n}^{-1}(s,t),\phi_{-P_{n}(\check{\theta}_{n}^{-1})(s)}^{\alpha}(\overline{f}_{n}\circ\theta_{n}^{-1}(s,t)))
\end{align*}
for $s\in[-1/2,1/2]$. Since $S_{n}R_{n}\rightarrow\sigma$
as $n\rightarrow\infty$ we have, using the first $C^0$-convergence result of Theorem \ref{thm:With-the-same-3} and the first result of Theorem \ref{thm:The-maps--2},
that 
\[
|\overline{a}_{n}\circ\theta_{n}^{-1}(s,t)-\Gamma_{n}\circ\theta_{n}^{-1}(s,t)+2\sigma s|\rightarrow0
\]
for all $s\in[-1/2,1/2]$ as $n\rightarrow\infty$. Moreover, since $\phi_{\rho}^\alpha$ for $\rho\in [-C,C]$ is a compact family of diffeomorphisms of $M$, they
are uniformly Lipschitz: there exists a constant $c>0$ such that for all $(s,t)\in[-1/2,1/2]$ 
\begin{align}\label{eq:unif_Lip}
\text{dist}_{\overline{g}_{0}}(\overline{f}_{n}\circ\theta_{n}^{-1}(s,t),w_{f})\geq c\text{dist}_{\overline{g}_{0}}(f_{n}\circ\theta_{n}^{-1}(s,t),
\phi_{-P_{n}(\check{\theta}_{n})^{-1}(s)}^{\alpha}(w_{f})).
\end{align}
Noting that 
\begin{equation}
P_{n}\check{\theta}_{n}^{-1}(s)=2(P_{n}R_{n}-P_{n}h_{n})s\label{eq:PnmM}
\end{equation}
for $s\in[-1/2,1/2]$, and that $P_{n}R_{n}\rightarrow\tau$ and $P_{n}h_{n}\rightarrow0$
as $n\rightarrow\infty$, it follows that $P_{n}\check{\theta}_{n}^{-1}(s)\rightarrow2\tau s$
in $C^{0}([-1/2,1/2])$.
By the triangle inequality we have for all $(s,t)\in[-1/2,1/2]\times S^{1}$
\[
\text{dist}_{\overline{g}_{0}}(f_{n}\circ\theta_{n}^{-1}(s,t),\phi_{-2\tau s}^{\alpha}(w_{f}))\leq \text{dist}_{\overline{g}_{0}}(f_{n}\circ\theta_{n}^{-1}(s,t),
\phi_{-P_{n}\check{\theta}_{n}^{-1}(s)}^{\alpha}(w_{f}))+
\text{dist}_{\overline{g}_{0}}(\phi_{-P_{n}\check{\theta}_{n}^{-1}(s)}^{\alpha}(w_{f}),
\phi_{-2\tau s}^{\alpha}(w_{f}))
\]
and the estimate (\ref{eq:unif_Lip}) together with the first $C^0$-convergence result in Theorem \ref{thm:With-the-same-3} 
resp. (\ref{eq:PnmM}) together with the subsequent discussion show that the terms
on the right hand side tend to $0$, as $n\rightarrow\infty$. Thus $v_{n}$ converge in $C^{0}$ on $[-1/2,1/2]\times S^1$
to the map $(s,t)\mapsto (-2\sigma s,\phi_{-2\tau s}^{\alpha}(w_{f}))$ whose projection to $M$ is a
segment of a Reeb trajectory; this completes the proof of the first statement.\\
To prove the second statement, we consider the maps $v_{n}^{-}$ which is given by 
\begin{align*}
v_{n}^{-}(s,t)=(\overline{a}_{n}^{-}\circ(\theta_{n}^{-})^{-1}(s,t)-\Gamma_{n}^{-}\circ(\theta_{n}^{-})^{-1}(s,t),\phi_{-P_{n}(\check{\theta}_{n}^{-})^{-1}(s)+P_{n}R_{n}}^{\alpha}
(\overline{f}_{n}^{-}\circ(\theta_{n}^{-})^{-1}(s,t))).
\end{align*}
The second $C^0$-convergence result in Theorem
\ref{thm:With-the-same-3} shows that the maps $(\overline{a}_{n}^{-}\circ(\theta_{n}^{-})^{-1}(s,t), f_n^-\circ (\theta_n^-)^{-1}(s,t))$  converge
to a continuous map $\overline{\nu}^-$ on $[-1,-1/2]\times S^1$ with value $w$ at $\{-1/2\}\times S^1$.
Moreover the maps $\Gamma_{n}^{-}\circ(\theta_{n}^{-})^{-1}(s,t)$
converge by Theorem \ref{thm:The-maps--1} in $C^{0}$ to a function $\Gamma^{-}\circ (\theta^{-})^{-1}(s,t)$
on $[-1,-1/2]\times S^1$, which equals to the constant $-\sigma$ on $\{-1/2\}\times S^1$. Since $|P_n(\check{\theta}_{n}^{-})^{-1}(s)|\leq Ch_n/R_n$
(as $|P_n|\leq C/R_n$), it follows that $v_{n}^{-}(s,t)$
converge in $C^{0}([-1,-1/2]\times S^1)$ to the map 
\[
(\overline{a}^{-}\circ(\theta^{-})^{-1}(s,t)-\Gamma^{-}\circ(\theta^{-})^{-1}(s,t),\phi_{\tau}^{\alpha}(\overline{f}^{-}\circ(\theta^{-})^{-1}(s,t))).
\]
By the above discussion this has the desired value $(\sigma,\phi^\alpha_\tau(w_f))$ on $\{-1/2\}\times S^1$ and hence we proved the second statement of the $C^{0}-$convergence result,
and thus Theorem \ref{thm:With-the-same}. 
\end{singlespace}
\end{proof}

\section{\label{subsec:Positive-Center-Action}Positive center action}

\begin{singlespace}
\noindent In this section we consider the case when there is no subsequence
of $\overline{u}_{n}$ with vanishing center action. Note that in
this case, due to Remark \ref{rem:From-the-definition}, the center
action of $\overline{u}_{n}$ is bounded from below by the constant
$\hbar>0$ defined in assumption A2. As in the previous
section we first characterize the asymptotic behavior of the $\overline{J}_{P}-$holomorphic
curves with positive center action (Theorem \ref{thm:Let--and1}).
We then prove a convergence result for the transformed psudoholomorphic
curves $\overline{u}_{n}$ and induce a convergence result on the
$\mathcal{H}-$holomorphic curves $u_{n}$ with harmonic perturbations
$\gamma_{n}$ by undoing the transformation. Theorem \ref{thm:With-the-same-1-1}
establishes the convergence of the transformed pseudoholomorphic curves
$\overline{u}_{n}$. 
\end{singlespace}

\subsection{Behavior of $\overline{J}_{P}-$holomorphic curves with positive
center action}

\begin{singlespace}

Since the contact form is assumed to be non-degenerate, there are only
finitely many ($S^1$-families of) periodic orbits $x(t)$ of the Reeb flow generated by $X_{\alpha}$ period $0<T\leq \tilde{E}_{0}$. 
We can therefore as in \cite{key-11}
choose a (arbitrary small) neighborhood $\mathcal{W}$ of the (normalized) loops $t\mapsto x(Tt), \,t\in S^1=\mathbb{R}/\mathbb{Z}$ 
in the loop space $C^\infty(S^1,M)$,
such that
the periodic orbits lie in different components of $\mathcal{W}$.
We can moreover assume that $\mathcal{W}$ is $S^{1}-$invariant 
for the natural $S^1$-action on $C^{\infty}(S^{1},M)$,
given by $(\vartheta\star y)(t):=y(t+\vartheta)$ for
$\vartheta\in S^{1}$.
The following result, which
is similar to Lemma 3.1 of \cite{key-11}, ensures that ``long''
$\overline{J}_{P}-$holomorphic curves $(\overline{u},R,P)$ with
small $d\alpha-$energies and positive center action are close to
some periodic orbit of the Reeb vector field $X_{\alpha}$. 
\end{singlespace}
\begin{lem}
\begin{singlespace}
\noindent \label{lem:Let--and-1}Given any $S^{1}-$invariant neighborhood
$\mathcal{W}\subset C^{\infty}(S^{1},M)$ of the
loops defined by the periodic solutions of $X_{\alpha}$ with periods
$T\leq\tilde{E}_{0}$ as above, there exists $h>0$ such that the
following hold: For any $R>h$ and any $\overline{J}_{P}-$holomorphic
curve $(\overline{u},R,P)$ satisfying $\overline{A}0$-$\overline{A}3$ such that $A(\overline{u})>0$ (up to possibly replacing $\overline{u}(s,t)$ by $\overline{u}(-s,-t)$)) the loops
$\overline{f}(s,\cdot):t\mapsto\overline{f}(s,t)$ satisfy $\overline{f}(s,\cdot)\in\mathcal{W}$
for all \textup{$s\in[-R+h,R-h]$}. Moreover, with $T=A(\overline{u})$
being the center action, the loops $\overline{f}(s,\cdot), s\in [-R+h,R-h],$ will all lie in the same component of $\mathcal{W}$
determined by a loop $t\mapsto x(Tt)$ corresponding
to a $T-$periodic orbit $x(t)$ of the Reeb vector field. 
\end{singlespace}
\end{lem}
\begin{singlespace}
For the loops associated to periodic orbits $x(t)$ of $X_\alpha$, we find from $\alpha(X_{\alpha})=1$
\[
T=\int_{S^{1}}x(T\cdot)^{*}\alpha,
\]
and so, given $\epsilon>0$ we can choose $\mathcal{W}$ so small such 
that for any $\gamma\in \mathcal{W}$ 
\begin{equation}
\left|\int_{S^{1}}\gamma^{*}\alpha-T\right|\leq\epsilon.\label{eq:estimate_center_action_arbitrary_loop}
\end{equation}
This holds then in particular for all $\gamma:=\overline{f}(s,\cdot)$ for $s\in[-R+h,R-h]$, where $\overline{u}=(\overline{a},\overline{f})$ is as above. 
\end{singlespace}
\begin{proof}
\begin{singlespace}
\noindent \emph{(of Lemma \ref{lem:Let--and-1})} The proof is almost
the same as that of Lemma 3.1 from \cite{key-11}. For completeness we outline the parts which are different. 
By applying Theorem \ref{thm:For-all-numbers} for $\psi=\hbar/4$ we find an $h_0>0$ such that $\left|\left|\int \overline{f}(s,\cdot)^*\alpha\right|-T\right|\leq \psi$
for all $s\in [-R+h_0,R-h_0]$ and all $\overline{u}$ as above.
Arguing now indirectly, we find a sequence $R_{n}$ with $R_{n}\geq n+h_{0}$ and a sequence of $\overline{J}_{P_{n}}-$holomorphic
curves $(\overline{u}_{n},R_{n},P_{n})$ with positive center actions
and satisfying $\overline{f}_{n}(\pm(s_{n},\cdot))\not\in\mathcal{W}$
for some sequence $s_{n}\in[-R_{n}+n,R_{n}-n]$. By assumption, the
center actions are positive. Hence $A(\overline{u}_{n})=T_{n}\geq\hbar$ by Remark \ref{rem:From-the-definition},
and we find by using Stokes' theorem that 
\begin{equation}
\left|\int_{S^{1}}\overline{f}_{n}(s_n,\cdot)^{*}\alpha\right|\geq 
\left|\int_{S^{1}}\overline{f}_{n}(0,\cdot)^{*}\alpha\right|-\left|\int_{[s_{n},0]\times S^{1}}\overline{f}_{n}^{*}d\alpha\right|
\geq
\hbar-\psi-\frac{\hbar}{2}=\frac{\hbar}{4}=:\epsilon_{0}>0\label{eq:Thm11}
\end{equation}
for all $n$ and all $s\in[-R_{n},R_{n}]$.

\noindent We define new curves $\overline{v}_{n}=(\overline{b}_{n},\overline{g}_{n}):[-R_{n}-s_{n},R_{n}-s_{n}]\times S^{1}\rightarrow\mathbb{R}\times M$
by 
\begin{align*}
\overline{v}_{n}(s,t) & =(\overline{b}_{n}(s,t),\overline{g}_{n}(s,t))=(\overline{a}_{n}(s+s_{n},t),\overline{f}_{n}(s+s_{n},t)).
\end{align*}
These curves have bounded total energies, small $d\alpha-$energies,
and satisfy 
\begin{align*}
\pi_{\alpha}d\overline{g}_{n}(s,t)\circ i & =J_{P_{n}(s_{n}+s)}(\overline{g}_{n}(s,t))\circ\pi_{\alpha}d\overline{g}_{n}(s,t),\\
(\overline{g}_{n}^{*}\alpha)\circ i & =d\overline{b}_{n}
\end{align*}
and $\overline{g}_{n}(\pm(0,\cdot))\not\in\mathcal{W}$ for all $n$. The
left and right ends of the interval $[-R_{n}-s_{n},R_{n}-s_{n}]$
converge to $-\infty$ and $+\infty$, respectively. Now define the
sequence of maps $\tilde{v}_{n}=(b_{n},v_{n}):[-R_{n}-s_{n},R_{n}-s_{n}]\times S^{1}\rightarrow\mathbb{R}\times M$
by setting $\tilde{v}_{n}(s,t)=(\overline{b}_{n}(s,t)-\overline{b}_{n}(0,0),\overline{g}_{n}(s,t))$.
The maps $\tilde{v}_{n}$ solve 
\begin{align*}
\pi_{\alpha}dv_{n}(s,t)\circ i & =J_{P_{n}(s_{n}+s)}(v_{n}(s,t))\circ\pi_{\alpha}dv_{n}(s,t),\\
(v_{n}^{*}\alpha)\circ i & =db_{n}.
\end{align*}
As in the proof of Theorem \ref{thm:For-all-numbers}, the gradients
of $\tilde{v}_{n}$ are uniformly bounded. Hence, by Arzelà\textendash Ascoli's
theorem, a subsequence of $\tilde{v}_{n}$ converges in $C_{\text{loc}}^{\infty}$,
i.e. 
\[
\tilde{v}_{n}\rightarrow\tilde{v}\text{ in }C_{\text{loc}}^{\infty}(\mathbb{R}\times S^{1},\mathbb{R}\times M),
\]
where $\tilde{v}=(b,v):\mathbb{R}\times S^{1}\rightarrow\mathbb{R}\times M$
is a usual $J_{\tau}$\textendash holomorphic curve for
some $\tau\in[-C,C]$ satisfying 
\begin{align*}
E_{\alpha}(\tilde{v};\mathbb{R}\times S^{1})+E_{d\alpha}(\tilde{v};\mathbb{R}\times S^{1}) & \leq\tilde{E}_{0},\\
E_{d\alpha}(\tilde{v};\mathbb{R}\times S^{1}) & \leq\frac{\hbar}{2},\\
\left|\int_{S^{1}}\tilde{v}(s,\cdot)^{*}\alpha\right| & \geq\epsilon_{0},\text{ for all }s\in\mathbb{R}.
\end{align*}
As in the proof of Lemma \ref{lem:Recall-the-constants}, one can now conclude that $E_{d\alpha}(\tilde{v};\mathbb{R}\times S^{1})$
vanishes and that $\tilde{v}$ is a trivial cylinder over a Reeb orbit $\gamma$. But then the convergence
$\overline{g}_n(\pm(0,\cdot))\rightarrow \gamma(\pm \cdot)$ contradicts the assumption that $\overline{g}_n(\pm(0,\cdot))=\overline{f}_n(\pm(0,\cdot))\notin \mathcal{W}$.
Some more details on this proof can be found in Lemma 3.1 of \cite{key-11}. 
\end{singlespace}
\end{proof}
\begin{singlespace}
\noindent In view of Lemma \ref{lem:Let--and-1} we fix a non-degenerate
periodic solution $x(t)$ of period $T\leq\tilde{E}_{0}$ and analyze
the curves $(\overline{u}=(\overline{a},\overline{f}),R,P)$ with
$\overline{f}([-R,R]\times S^{1})\subset\mathcal{U}$, where $\mathcal{U}$
is a small tubular neighborhood of $x(\mathbb{R})$.

\noindent To study long cylinders with positive center action we use
some special coordinates centered at the periodic orbit. Denote by $\alpha_{0}$ the standard contact
form $\alpha_{0}=d\vartheta+xdy$ on $S^{1}\times\mathbb{R}^{2}$
with coordinates $(\vartheta,x,y)$. The next lemma introduces the
``standard coordinates'' near a periodic orbit of the Reeb vector
field. For a proof we refer to \cite{key-12}. 
\end{singlespace}
\begin{lem}
\begin{singlespace}
\noindent \label{lem:Let--be}Let $(M,\alpha)$ be a $3-$dimensional
manifold equipped with a contact form, and let $x(t)$ be the $T-$periodic
solution of the corresponding Reeb vector field $\dot{x}=X_{\alpha}(x)$
on $M$. Let $\tau_{0}$ be the minimal period such that $T=k\tau_{0}$
for some positive integer $k$. Then there exists an open neighborhood
$U\subset S^{1}\times\mathbb{R}^{2}$ of $S^{1}\times\{0\}$, an open
neighborhood $V\subset M$ of $P=x(\mathbb{R})$, and a diffeomorphism
$\varphi:U\rightarrow V$ mapping $S^{1}\times\{0\}$ onto $P$ such
that 
\[
\varphi^{*}\alpha=f\cdot\alpha_{0}
\]
for a positive smooth function $f:U\rightarrow\mathbb{R}$ satisfying
\begin{equation}
f\equiv\tau_{0}\text{ and }df\equiv0\label{eq:special_coordinates_contact}
\end{equation}
on $S^{1}\times\{0\}$. 
\end{singlespace}
\end{lem}
\begin{singlespace}
\noindent The following description is borrowed from \cite{key-11}.
As $S^{1}=\mathbb{R}/\mathbb{Z}$ we work in the covering space and
denote by $(\vartheta,x,y)$ the coordinates, where $\vartheta$ is
mod 1. In these coordinates, the contact form $\alpha$ is $\alpha=f\cdot\alpha_{0}$
for a smooth function $f:\mathbb{R}^{3}\rightarrow(0,\infty)$ defined
near $S^{1}\times\{0\}$, being periodic in $\vartheta$, i.e. $f(\vartheta+1,x,y)=f(\vartheta,x,y)$,
and satisfying (\ref{eq:special_coordinates_contact}). The Reeb vector field
$X_{\alpha}=(X_{0},X_{1},X_{2})$ has the components 
\[
X_{0}=\frac{1}{f^{2}}(f+x\partial_{x}f),\,\,\,X_{1}=\frac{1}{f^{2}}(\partial_{y}f-x\partial_{\vartheta}f),\,\,\,X_{2}=-\frac{1}{f^{2}}\partial_{x}f.
\]
The vector field $X_{\alpha}$ is periodic in $\vartheta$ of period
1 and constant along the periodic orbit $x(\mathbb{R})$, i.e. $X_{\alpha}(\vartheta,0,0)=(\tau_{0}^{-1},0,0)$.
The periodic solution is represented as $x(Tt)=(kt,0,0)$, where $T=k\tau_{0}$
is the period, $\tau_{0}$ the minimal period, and $k$ the covering
number of the periodic solution. The subsequent lemma is rather technical
and describes the behavior of a long $\overline{J}_{P}-$holomorphic
curve $(\overline{u},R,P)$ in the coordinates introduced by Lemma
\ref{lem:Let--be}. 
\end{singlespace}
\begin{lem}
\begin{singlespace}
\noindent \label{lem:Let--and-2}For any $N\in\mathbb{N}$, $\delta>0$,
there exists $h>0$ such that for any $R>h$ and any $\overline{J}_{P}-$holomorphic
curve $(\overline{u},R,P)$ satisfying $\overline{A}0$-$\overline{A}3$ and $\mathcal{A}(u)>0$, the
representation 
\[
\overline{u}(s,t)=(\overline{a}(s,t),\vartheta(s,t),z(s,t)=(x(s,t),y(s,t)))
\]
of the cylinder in the above local coordinates satisfies (up to possibly replacing $\overline{u}(s,t)$ by $\overline{u}(-s,-t)$) the following:
For all $(s,t)\in[-R+h,R-h]\times S^{1}$ we have 
\[
|\partial^{\alpha}(\overline{a}(s,t)-Ts)|\leq\delta\text{ and }|\partial^{\alpha}(\vartheta(s,t)-kt)|\leq\delta
\]
for $1\leq|\alpha|\leq N$, and 
\[
|\partial^{\alpha}z(s,t)|\leq\delta
\]
for all $0\leq|\alpha|\leq N$. Here, $T$ is the period and $k$
the covering number of the distinguished periodic solution lying in
the center of the tubular neighborhood. 
\end{singlespace}
\end{lem}
\begin{proof}
\begin{singlespace}
\noindent The proof is more or less the same as that of Lemma 3.3
in \cite{key-11}. We argue by contradiction. Assume there exist $N\in\mathbb{N}$,
$\delta_{0}>0$ such that for any $h_{n}=2n$ we find $R_{n}>2n$
and $\overline{J}_{P_{n}}-$holomorphic curves $(\overline{u}_{n},R_{n},P_{n})$
satisfying the following. Representing the maps $\overline{u}_{n}$
in local coordinates by 
\[
\overline{u}_{n}(s,t)=(\overline{a}_{n}(s,t),\vartheta_{n}(s,t).z_{n}(s,t)),
\]
we assume the existence of a sequence $(s_{n},t_{n})\in[-R_{n}-n,R_{n}-n]\times S^{1}$
and a multiindex $\alpha$ with $1\leq|\alpha|\leq N$ such that we have (even after potentially replacing $\overline{u}_n(s,t)$ by $\overline{u}_n(-s,-t)$)
\begin{equation}
\left|\partial^{\alpha}\left[(\overline{a}_{n}-Ts,\vartheta_{n}-kt)\right](s_{n},t_{n})\right|\geq\delta_{0}.\label{eq:non-ct.}
\end{equation}
We define the translated sequence $\tilde{v}_{n}:[-n,n]\times S^{1}\rightarrow\mathbb{R}\times M$
by 
\[
\tilde{v}_{n}(s,t)=(b_{n}(s,t),v_{n}(s,t))=(\overline{a}_{n}(s+s_{n},t)-\overline{a}_{n}(s_{n},t_{n}),\overline{f}_{n}(s+s_{n},t)).
\]
These maps satisfy have the same $\alpha$-energy (resp. $d\alpha$-energy) as
the maps $\overline{u}_{n}$, and they solve
\begin{align*}
\pi_{\alpha}dv_{n}(s,t)\circ i & =J_{P_{n}(s+s_{n})}\circ\pi_{\alpha}dv_{n}(s,t),\\
(v_{n}^{*}\alpha)\circ i & =db_{n}.
\end{align*}
Thus the maps have uniform gradient bounds and converge after passing to a subsequence to some usual $J_\sigma$-holomorphic cylinder $u:\mathbb{R}\times S^1\rightarrow \mathbb{R}\times M$, 
which has small $d\alpha$-energy, $\alpha$ energy bounded by $E_0$ but is nonconstant. Thus it is a trivial cylinder over (an $S^1$-shift) of the periodic Reeb orbit $x$.
However the convergence of $v_n$ to a trivial cylinder contradicts the assumption (\ref{eq:non-ct.}). Similarly one proves the second statement of the Lemma.
More details for an analogous proof can be found in the proof of Lemma 3.3 in \cite{key-11}. 
\end{singlespace}
\end{proof}
\begin{singlespace}
\noindent We compute the Cauchy-Riemann equations for the representation
\begin{align*}
\overline{u}(s,t) & =(\overline{a}(s,t),\overline{f}(s,t))=(\overline{a}(s,t),\vartheta(s,t),z(s,t))\\
 & =(\overline{a}(s,t),\vartheta(s,t),x(s,t),y(s,t))
\end{align*}
of a $\overline{J}_{P}-$holomorphic curve $(\overline{u},R,P)$ in
the local coordinates $\mathbb{R}\times\mathbb{R}^{3}$ of the tubular
neighborhood given in Lemma \ref{lem:Let--be}. In the following,
we adopt the same constructions as in \cite{key-11}. On $\mathbb{R}^{3}$
we have the contact form $\alpha=f\alpha_{0}$. At point $m=(t,x,y)\in\mathbb{R}^{3}$,
the contact structure $\xi_{m}=\ker(\alpha_{m})$ is spanned by the
vectors 
\[
E_{1}=\left(\begin{array}{c}
0\\
1\\
0
\end{array}\right)\textrm{ and }E_{2}=\left(\begin{array}{c}
-x\\
0\\
1
\end{array}\right).
\]
We denote by $\mathbb{J}_{\rho}(m)$ the $2\times2$ matrix representing
the compatible almost complex structure on the plane $\xi_{m}$ in
the basis $\{E_{1},E_{2}\}$ for all $\rho\in[-C,C]$. In the basis
$\{E_{1},E_{2}\}$, the symplectic structure $d\alpha|_{\xi_{m}}$
is given by the skew symmetric matrix function $f(m)\mathbb{J}_0$, where
\[
\mathbb{J}_0=\left(\begin{array}{cc}
0 & -1\\
1 & 0
\end{array}\right).
\]
Therefore, in view of the compatibility requirement, the complex multiplication
$\mathbb{J}_{\rho}(m)$ has the properties $\mathbb{J}_{\rho}(m)^{2}=-\text{id}$,
$\mathbb{J}_{\rho}(m)^{T}\mathbb{J}_0\mathbb{J}_{\rho}(m)=\mathbb{J}_0$ and
$-\mathbb{J}_0\mathbb{J}_{\rho}(m)>0$. In particular, $\mathbb{J}_0\mathbb{J}_{\rho}(m)$
is a symmetric matrix for all $\rho\in[-C,C]$, and it follows that
\[
\left\langle x,y\right\rangle _{\rho}:=\Bigl\langle x,-\mathbb{J}_0\mathbb{J}_{\rho}(m)y\Bigr\rangle
\]
is an inner product on $\mathbb{R}^{2}$ which is left invariant under
$\mathbb{J}_{\rho}(m)$, i.e. $\left\langle \mathbb{J}_{\rho}(m)x,\mathbb{J}_{\rho}(m)y\right\rangle _{\rho}=\left\langle x,y\right\rangle _{\rho}$
for all $\rho\in[-C,C]$. The Reeb vector field $X_{\alpha}$ can
be written as $X_{\alpha}=(X_{0},X_{1},X_{2})\in\mathbb{R}\times\mathbb{R}^{2}$.
Setting $z=(x,y)\in\mathbb{R}^{2}$ we define $Y(t,z)=(X_{1}(t,z),X_{2}(t,z))\in\mathbb{R}^{2}$.
Since $X(t,0)=(1/\tau_{0},0)$ we have $Y(t,z)=D(t,z)z$, where 
\[
D(t,z)=\int_{0}^{1}dY(t,\rho z)d\rho,
\]
and $d$ is the derivative with respect to the $z-$variable. In particular,
if $z=0$ we obtain 
\[
D(t,0)=dY(t,0)=\frac{1}{\tau_{0}^{2}}\left(\begin{array}{cc}
\partial_{xy}f & \partial_{yy}f\\
-\partial_{xx}f & -\partial_{xy}f
\end{array}\right).
\]
We introduce the $2\times2$ matrices depending on $\overline{u}(s,t)$
and $Ps$ by 
\begin{align*}
J(s,t) & =J_{Ps}(\overline{f}(s,t))=J_{Ps}(\vartheta(s,t).z(s,t)),\\
S(s,t) & =\left[\partial_{t}\overline{a}-\partial_{s}\overline{a}\cdot J(s,t)\right]D(\overline{f}(s,t)).
\end{align*}
In the basis $\{E_{1},E_{2}\}$ of the contact plane $\xi_{m}$ at
$m=\overline{f}(s,t)$ and for the representation $\overline{u}(s,t)=(\overline{a}(s,t),\vartheta(s,t),z(s,t))\in\mathbb{R}\times\mathbb{R}\times\mathbb{R}^{2}$,
we have
\[
\pi_{\alpha}\partial_{s}\overline{f}(s,t)+J_{Ps}(\overline{f}(s,t))\pi_{\alpha}\partial_{t}\overline{f}(s,t)=0.
\]
We find 
\[
z_{s}+J(s,t)z_{t}+S(s,t)z=0
\]
and further on, with $z(s,t)=(x(s,t),y(s,t))$, 
\begin{align*}
\overline{a}_{s}=(\vartheta_{t}+xy_{t})f(\overline{f})\ \text{ and }\ \overline{a}_{t}=-(\vartheta_{s}+xy_{s})f(\overline{f}).
\end{align*}
It is convenient to decompose the matrix $S(s,t)$ into its symmetric
and anti-symmetric parts with respect to the inner product $\left\langle \cdot,-\mathbb{J}_0J(s,t)\cdot\right\rangle =\left\langle \cdot,-\mathbb{J}_0J_{Ps}(\overline{f}(s,t))\cdot\right\rangle $
on $\mathbb{R}^{2}$ by introducing 
\begin{align*}
B(s,t)=\frac{1}{2}\left[S(s,t)+S^{*}(s,t)\right]\ \text{ and }\ C(s,t)=\frac{1}{2}\left[S(s,t)-S^{*}(s,t)\right],
\end{align*}
where $S^{*}$ is the transpose of $S$ with respect to the inner
product $\left\langle \cdot,-\mathbb{J}_0J(s,t)\cdot\right\rangle $. Explicitly
we have $S^{*}=J\mathbb{J}_0S^{T}\mathbb{J}_0J$, where $S^{T}$ is the transpose
matrix of $S$ with respect to the Euclidean inner product $\left\langle \cdot,\cdot\right\rangle $
in $\mathbb{R}^{2}$. In terms of $B$ and $C$, the above equation
becomes 
\[
z_{s}+J(s,t)z_{t}+B(s,t)z+C(s,t)z=0.
\]
The operator $A(s):W^{1,2}(S^{1},\mathbb{R}^{2})\subset L^{2}(S^{1},\mathbb{R}^{2})\rightarrow L^{2}(S^{1},\mathbb{R}^{2})$,
given by 
\[
A(s)=-J(s,t)\frac{d}{dt}-B(s,t),
\]
is self-adjoint with respect to the inner product $\left\langle \cdot,\cdot\right\rangle _{s}$
in $L^{2}$, defined for $x,y\in L^{2}(S^{1},\mathbb{R}^{2})$ by
\[
\left\langle x,y\right\rangle _{s}:=\int_{0}^{1}\left\langle x(t),-\mathbb{J}_0J(s,t)y(t)\right\rangle dt.
\]
The norms $\left\Vert x\right\Vert _{s}^{2}:=\left\langle x,x\right\rangle _{s}$
are equivalent to the standard $L^{2}(S^{1},\mathbb{R}^{2})-$norms
(denoted by $\left\Vert \cdot\right\Vert $) in the following sense: 
\end{singlespace}
\begin{lem}
\begin{singlespace}
\noindent \label{lem:There-exist-the}There exist constants $h,c>0$
such that for all $R>h$ and all $\overline{J}_{P}-$holomorphic curves
$(\overline{u},R,P)$ satisfying $\overline{A}0$-$\overline{A}3$ and $\mathcal{A}(u)>0$,
all $x\in L^{2}(S^{1},\mathbb{R}^{2})$, and all $s\in[-R,R]$, we
have 
\[
\frac{1}{c}\left\Vert x\right\Vert _{s}\leq\left\Vert x\right\Vert \leq c\left\Vert x\right\Vert _{s}.
\]
\end{singlespace}
\end{lem}
\begin{proof}
\begin{singlespace}
\noindent The Euclidean inner product on the coordinate neighborhood is equivalent to the metric $\overline{g}_0$.
The inequalities follow,  since the
 almost complex structures $J_{\rho}$,
$\rho\in[-C,C]$ are uniformly bounded w.r.t $\overline{g_0}$ and since $-\mathbb{J}_0J(s,t)$ is uniformly positive definite, 
i.e. since $J(s,t)$ are compatible on $\xi$ by $d\alpha$ and thus uniformly tamed rel. to $\overline{g}_0$
(compare also appendix \ref{subsec:A-version-of}). Both these facts rely on the compactness of the family of almost complex structures
$J_\rho, \rho\in [-C,C]$ and therefore on the assumption $A3$.
\end{singlespace}
\end{proof}
\begin{lem}
\begin{singlespace}
\noindent \label{lem:There-exists-a}There exists a constant $h>0$
such that for every $R>h$ and every $\overline{J}_{P}-$holomorphic
curve $(\overline{u},R,P)$ satisfying $\overline{A}0$-$\overline{A}3$ and $\mathcal{A}(u)>0$,
the following holds true. If $\overline{u}=(\overline{a},\overline{f})$
is the reparametrization in local coordinates and $A(s)$ the associated
operator, then there exists a constant $\eta>0$ such that 
\[
\left\Vert A(s)\xi\right\Vert _{s}\geq\eta\left\Vert \xi\right\Vert _{s}
\]
for all $s\in[-R+h,R-h]$ and all $\xi\in W^{1,2}(S^{1},\mathbb{R}^{2})$. 
\end{singlespace}
\end{lem}
\begin{proof}
\begin{singlespace}
\noindent We prove the inequality by contradiction. Assume that this
is not true. Then for any $h_{n}=2n$ we assume the existence of $R_{n}\in\mathbb{R}_{>0}$
such that $R_{n}>2n$ and a sequence of $\overline{J}_{P_{n}}-$holomorphic
curves $(\overline{u}_{n},R_{n},P_{n})$ satisfying $\overline{A}0$-$\overline{A}3$ and $\mathcal{A}(u)>0$ and thus
\[
\left|\int_{S^{1}}\overline{f}_{n}(s)^{*}\alpha\right|\geq\epsilon_{0}
\]
for all $s\in[-R_{n},R_{n}]$. Here $\epsilon_{0}>0$ is the constant
defined by (\ref{eq:Thm11}). Representing $\overline{u}_{n}$ in
local coordinates as $\overline{u}_{n}(s,t)=(\overline{a}_{n}(s,t),\vartheta_{n}(s,t),z_{n}(s,t))$
we have the associated operator 
\[
A_{n}(s)=-J_{n}(s,t)\frac{d}{dt}-B_{n}(s,t)
\]
where $J_{n}(s,t)=J_{P_{n}s}(\overline{f}_{n}(s,t))$ and
where $S_{n}(s,t)$ and $B_{n}(s,t)$ are defined as above. We assume
that there exists a sequence $s_{n}\in[-R_{n}-n,R_{n}+n]$ and $\xi_{n}\in W^{1,2}(S^{1},\mathbb{R}^{2})$
such that 
\begin{equation}
\left\Vert \xi_{n}\right\Vert _{s_{n}}=1\text{ and }\left\Vert A_{n}(s_{n})\xi_{n}\right\Vert _{s_{n}}\rightarrow0.\label{eq:}
\end{equation}
Now we consider the translated maps 
\[
\tilde{v}_{n}(s,t)=(b_{n}(s,t),v_{n}(s,t))=(\overline{a}_{n}(s+s_{n},t)-\overline{a}_{n}(s_{n},0),\overline{f}_{n}(s+s_{n},t))
\]
for all $n$ and $(s,t)\in[-n,n]\times S^{1}$. Arguing as in the proof of Lemma \ref{lem:Let--and-1}
we find that $\tilde{v}_{n}\rightarrow\tilde{v}$ in $C_{\text{loc}}^{\infty}(\mathbb{R}\times S^{1},\mathbb{R}\times M)$,
where $\tilde{v}$ is a cylinder over a (shift of) the distinguished periodic orbit
$x(t)$ lying in the center of the tubular neighborhood. Hence in
local coordinates we can write $\tilde{v}(s,t)=(Ts+a_{0},kt+\vartheta_{0},0)$
with two constants $a_{0}$ and $\vartheta_{0}$. Setting $s=0$ we
obtain 
\begin{align*}
\frac{\partial}{\partial t}\overline{a}_{n}(s_{n},t) & \rightarrow0,\\
\frac{\partial}{\partial s}\overline{a}_{n}(s_{n},t) & \rightarrow T,\\
\vartheta_{n}(s_{n},t) & \rightarrow kt+\vartheta_{0},\\
z_{n}(s_{n},t) & \rightarrow0
\end{align*}
as $n\rightarrow\infty$, uniformly in $t$. Consequently, it follows
that 
\begin{align}
B_{n}(s_{n},t) & \rightarrow TJ_{\tau_{\{s_{n}\}}}(kt+\vartheta_{0},0)\cdot dY(kt+\vartheta_{0},0),\label{eq:B}\\
J_{n}(s_{n},t) & \rightarrow J_{\tau_{\{s_{n}\}}}(kt+\vartheta_{0},0)\label{eq:J}
\end{align}
as $n\rightarrow\infty$, uniformly in $t$ and for some $\tau_{\{s_{n}\}}$
given by $P_{n}s_{n}\rightarrow\tau_{\{s_{n}\}}$. Since $\left\Vert J_{n}(s,\cdot)\xi\right\Vert _{s}=\left\Vert \xi\right\Vert _{s}$
for every $\xi\in L^{2}(S^{1},\mathbb{R}^{2})$ and from Lemma \ref{lem:There-exist-the},
there exists a constant $c>0$ such that for all $n\in\mathbb{N}$
and $\xi\in W^{1,2}(S^{1},\mathbb{R}^{2})$ we have 
\begin{equation}
\left\Vert \dot{\xi}\right\Vert \leq c\left(\left\Vert A_{n}(s_{n})\xi\right\Vert +\left\Vert B_{n}(s_{n},\cdot)\xi\right\Vert \right).\label{eq:Lem}
\end{equation}
Consequently, the sequence $\xi_{n}$ as constructed in (\ref{eq:})
is bounded in $W^{1,2}$. Since $W^{1,2}$ is compactly embedded in
$L^{2}$, a subsequence of $\xi_{n}$ converges in $L^{2}$. Therefore,
by assumption (\ref{eq:}), the limits (\ref{eq:B}) and (\ref{eq:J}),
and the estimate (\ref{eq:Lem}) we have that after going over to
a subsequence, $\xi_{n}$ is a Cauchy sequence in $W^{1,2}(S^{1},\mathbb{R}^{2})$
so that 
\[
\xi_{n}\rightarrow\xi\text{ in }W^{1,2}(S^{1},\mathbb{R}^{2}).
\]
From 
\[
A_{n}(s_{n})\xi_{n}=-J_{n}(s_{n},t)\dot{\xi}_{n}-B_{n}(s_{n},t)\xi_{n}\rightarrow0\text{ in }L^{2}(S^{1},\mathbb{R}^{2})
\]
and (\ref{eq:B}) and (\ref{eq:J}) one concludes that $\xi$ solves
the equation 
\[
\frac{d}{dt}\xi(t)=TdY(kt+\vartheta_{0},0)\xi(t)
\]
which (since $\|\xi\|\geq 1/c$) results in a contradiction to the fact that the periodic orbits $x(t)=(kt+\vartheta_{0},0)$
was assumed to be non-degenerate. 
\end{singlespace}
\end{proof}
\begin{singlespace}
\noindent The next theorem is similar to Theorem 1.3 of \cite{key-11};
the only difference is that it is formulated for pseudoholomorphic
curves with respect to a domain-dependent almost complex structure
on the target space $\mathbb{R}\times M$. 
\end{singlespace}
\begin{thm}
\begin{singlespace}
\noindent \label{thm:Let--and1} For  $0<\psi<\hbar/2$, let $h_{0}>0$ be the corresponding constant appearing
in Theorem \ref{thm:For-all-numbers}.
Then there exist positive constants $\delta_{0}$, $\mu$, and
$\nu<\min\{4\pi,2\mu\}$ such that the following holds: Given $0<\delta\leq\delta_{0}$,
there exists $h>h_0$ such that for any $R>h$ and any $\overline{J}_{P}-$holomorphic
curve $(\overline{u},R,P)$ such that $A(\overline{u})>0$, there
exists a unique (up to a phase shift) periodic orbit $x(t)$ of the
Reeb vector field $X_{\alpha}$ with period $T=A(\overline{u})\leq\tilde{E}_{0}$
satisfying 
\[
\left|\int_{S^{1}}\overline{f}(0)^{*}\alpha-T\right|<\frac{\psi}{2}\text{ and }\left|\int_{S^{1}}\overline{f}(s)^{*}\alpha-T\right|<\hbar,\text{ for all }s\in[-R,R].
\]
In addition, there exists a tubular neighborhood $U\cong S^{1}\times\mathbb{R}^{2}$
around the periodic orbit $x(\mathbb{R})\cong S^{1}\times\{0\}$ such
that $\overline{f}(s,t)\in U$ for all $(s,t)\in[-R+h,R-h]\times S^{1}$.
Using the covering $\mathbb{R}$ of $S^{1}=\mathbb{R}/\mathbb{Z}$,
the map $\overline{u}$ is represented in local coordinates $\mathbb{R}\times U$
as 
\begin{align*}
\overline{u}(s,t) & =(\overline{a}(s,t),\vartheta(s,t),z(s,t))\\
 & =(Ts+a_{0}+\tilde{a}(s,t),kt+\vartheta_{0}+\tilde{\vartheta}(s,t),z(s,t)),
\end{align*}
where $(a_{0},\vartheta_{0})\in\mathbb{R}^{2}$ are constants. The
functions $\tilde{a}$, $\tilde{\vartheta}$, and $z$ are $1-$periodic
in $t$, and the positive integer $k$ is the covering number of the
$T-$periodic orbit represented by $x(Tt)=(kt,0,0)$. For all multiindices
$\alpha$ there exists a constant $C_{\alpha}$ such that for all
$(s,t)\in[-R+h,R-h]\times S^{1}$ the following estimates hold: 
\[
|\partial^{\alpha}z(s,t)|^{2}\leq C_{\alpha}\delta^{2}\frac{\cosh(\mu s)}{\cosh(\mu(R-h))}
\]
and 
\[
|\partial^{\alpha}\tilde{a}(s,t)|^{2},|\partial^{\alpha}\tilde{\vartheta}(s,t)|^{2}\leq C_{\alpha}\delta^{2}\frac{\cosh(\nu s)}{\cosh(\nu(R-h))}.
\]
\end{singlespace}
\end{thm}
\begin{singlespace}
\noindent For the proof of the theorem we need the following 
\end{singlespace}
\begin{rem}
\begin{singlespace}
\noindent \label{rem:By-Lemma-,}By Lemma \ref{lem:Let--and-2}, which
is similar to Lemma 3.3 from \cite{key-11}, we have $|\partial_{s}^{\alpha}\overline{f}(s,t)|\leq\delta$
for all $\alpha\geq1$ and all $(s,t)\in[-R+h,R-h]\times S^{1}$.
As a result, the derivatives with respect to the $s$ coordinate of
$J(s,t)$ and $B(s,t)$ contain factors estimated by $\delta$. This
can be seen as follows. Recalling that $J(s,t)=J_{Ps}(\vartheta(s,t),z(s,t))$
we find 
\[
\partial_{s}J(s,t)=P\partial_{\rho}J_{Ps}(\overline{f}(s,t))+\partial_{\vartheta}J_{Ps}(\overline{f}(s,t))\partial_{s}\vartheta+\partial_{z}J_{Ps}(\overline{f}(s,t))\partial_{s}z.
\]
For $R$ sufficiently large, the assumption $A3$ on the universal bound
of the conformal co-period gives $|P|\leq\delta$; consequently, $|\partial_{s}J(s,t)|=\mathcal{O}(\delta)$.
In a similar way it can be shown that $|\partial_{s}^{2}J(s,t)|,|\partial_{s}B(s,t)|=\mathcal{O}(\delta)$. 
\end{singlespace}
\end{rem}
\begin{singlespace}
\noindent Now, the proof of Theorem \ref{thm:Let--and1} proceeds
as in \cite{key-11} by using Lemma \ref{lem:There-exists-a} and
Remark \ref{rem:By-Lemma-,}, and for this reason, it is omitted here. 
\end{singlespace}

\subsection{Proof of Theorem \ref{thm:With-the-same-1}}

The goal of this section is to describe the convergence and
the limit object of the $\mathcal{H}-$holomorphic cylinders $u_{n}$
with harmonic perturbations $\gamma_{n}$ whose center
action, defined as in (\ref{eq:defcenteraction}) is positive. In a first step we prove
an analogous result for a sequence of $J_{P_n}$-holomorphic cylinders.

\begin{singlespace}
\noindent Applying Theorem \ref{thm:Let--and1} to a sequence of
$\overline{J}_{P_{n}}-$holomorphic curves $(\overline{u}_{n},R_{n},P_{n})$ satisfying $\overline{A}0$-$\overline{A}3$
satisfying $\mathcal{A}(u)>0$ and $R_n\rightarrow\infty$, we find the following (after passing to a subsequence, and possibly replacing $\overline{u}_n(s,t)$ by $\overline{u}_n(-s,-t)$):
\end{singlespace}
\begin{cor}
\begin{singlespace}
\noindent There exists a $T$-periodic orbit $x(t)$ of $X_\alpha$ for $T\leq \tilde{E}_0$ such that for any $\epsilon>0$ there exist $h>0$ and $N_{\epsilon,h}\in\mathbb{N}$
such that we have for any $n\geq N_{\epsilon,h}$,
\begin{equation}
d(\overline{f}_{n}(s,t),x(Tt))<\epsilon\text{ and }|\overline{a}_{n}(s,t)-Ts|<\epsilon \, \text{ for all }(s,t)\in[-R_{n}+h,R_{n}-h]\times S^{1}.\label{eq:a-coordinate-diverges}
\end{equation}
\end{singlespace}
\end{cor}
\begin{singlespace}
\noindent We continue to denote the cylinder $[-R_{n}+h,R_{n}-h]\times S^{1}$
by $[-R_{n},R_{n}]\times S^{1}$. In view of (\ref{eq:a-coordinate-diverges}) the quantities 
\[
\overline{r}_{n}^{-}:=\inf_{t\in S^{1}}\overline{a}_{n}(-R_{n},t)\text{ and }\overline{r}_{n}^{+}:=\sup_{t\in S^{1}}\overline{a}_{n}(R_{n},t)
\]
satisfy $\overline{r}_{n}^{+}-\overline{r}_{n}^{-}\rightarrow\infty$
as $n\rightarrow\infty$.

\noindent Recalling that $P_{n},S_{n}$ and $1/R_{n}$ are zero sequences
we reformulate the above findings as in Corollary \ref{cor:For-every-sequence}. 
\end{singlespace}
\begin{cor}
\begin{singlespace}
\noindent \label{cor:There-exists-a-1}For every sequence $h_{n}\in\mathbb{R}_{>0}$
satisfying $h_{n}<R_{n}$ and $h_{n},R_{n}/h_{n}\rightarrow\infty$
and every $\epsilon>0$ there exists $N\in\mathbb{N}$ such that for
all $(s,t)\in[-R_{n}+h_{n},R_{n}-h_{n}]\times S^{1}$ 
\[
\text{dist}_{\overline{g}_{0}}(\overline{f}_{n}(s,t),x(Tt))<\epsilon\text{ and }|\overline{a}_{n}(s,t)-Ts|<\epsilon
\]
for all $n\geq N$.
\end{singlespace}
\end{cor}
\begin{proof}
\begin{singlespace}
\noindent The proof follows as in
Corollary \ref{cor:For-every-sequence}. 
\end{singlespace}
\end{proof}
\begin{singlespace}
\noindent The next theorem which states a $C_{\text{loc}}^{\infty}-$
and a $C^{0}-$convergence result for the maps $\overline{u}_{n}$
with positive center action is the analogue of Theorem \ref{thm:With-the-same-3}. 
\end{singlespace}
\begin{thm}
\begin{singlespace}
\noindent \label{thm:With-the-same-1-1} Let $(\overline{u}_{n},R_{n},P_{n})$ be a sequence of $J_{P_n}$-holomorphic curves satisfying $\overline{A}0$-$\overline{A}3$ with positive center action
and such that $\overline{u}_n(0,0)\rightarrow w:=(0,w_f)$, $R_n\rightarrow \infty$ and $R_nP_n\rightarrow \tau$. Then there exists a subsequence
also denoted by $(\overline{u}_{n},R_{n},P_{n})$, $T\in \mathbb{R}\setminus \{0\}$, a $|T|$-periodic Reeb orbit $x$ and pseudoholomorphic
half cylinders $u^{\pm}$ defined on $(-\infty,0]\times S^{1}$ and
$[0,\infty)\times S^{1}$ respectively, such that for every sequence
$h_{n}\in\mathbb{R}_{>0}$ and every sequence of diffeomorphisms $\theta_{n}:[-R_{n},R_{n}]\times S^{1}\rightarrow[-1,1]\times S^{1}$
satisfying the assumptions of Remark \ref{rem:For-every-sequence},
the following convergence results hold:
\noindent $C_{\text{loc}}^{\infty}-$convergence: 
\end{singlespace}
\begin{enumerate}
\begin{singlespace}
\item For any sequence $s_{n}\in[-R_{n}+h_{n},R_{n}-h_{n}]$ the shifted
maps $\overline{u}_{n}(s+s_{n},t)-Ts_{n}$, defined on $[-R_{n}+h_{n}-s_{n},R_{n}-h_{n}-s_{n}]\times S^{1}$,
converge in $C_{\text{loc}}^{\infty}$ to $(Ts,x(Tt))$. 
\item The left shifts $\overline{u}_{n}^{-}(s,t):=\overline{u}_{n}(s-R_{n},t)+TR_n$
defined on $[0,h_{n})\times S^{1}$, converge
in $C_{\text{loc}}^{\infty}$ to a pseudoholomorphic half cylinder
$\overline{u}^{-}=(\overline{a}^{-},\overline{f}^{-})$ defined on
$[0,+\infty)\times S^{1}$. The curve $\overline{u}^{-}$ has the same asymptotes
as the map $(s,t)\mapsto (Ts,x(Tt))$ for $s\rightarrow \infty$. The maps $\overline{v}_{n}^{-}$ converge
in $C_{\text{loc}}^{\infty}$ on $[-1,-1/2)\times S^{1}$ to $\overline{v}^{-}$.
\item The right shifts $\overline{u}_{n}^{+}(s,t):=\overline{u}_{n}(s+R_{n},t)-TR_n$,
defined on $(-h_{n},0]\times S^{1}$, converge
in $C_{\text{loc}}^{\infty}$ to a pseudoholomorphic half
cylinder $\overline{u}^{+}=(\overline{a}^{+},\overline{f}^{+})$,
defined on $(-\infty,0]\times S^{1}$. The curve $\overline{u}^{+}$
has the same asymptotes as $(s,t)\mapsto (Ts,x(Tt))$ for $s\rightarrow -\infty$. The maps $\overline{v}_{n}^{+}$
converge in $C_{\text{loc}}^{\infty}$ on $(1/2,1]\times S^{1}$ to
$\overline{v}^{+}$.
\end{singlespace}
\end{enumerate}
\begin{singlespace}
\noindent $C^{0}-$convergence: 
\end{singlespace}
\begin{enumerate}
\begin{singlespace}
\item The maps $\overline{f}_{n}\circ\theta_{n}^{-1}:[-1/2,1/2]\times S^{1}\rightarrow M$
converge in $C^{0}$ to the map $(s,t)\rightarrow x(Tt)$. 
\item The maps $\overline{f}_{n}^{-}\circ(\theta_{n}^{-})^{-1}:[-1,-1/2]\times S^{1}\rightarrow M$
converge in $C^{0}$ to a map $\overline{f}^{-}\circ(\theta^{-})^{-1}:[-1,-1/2]\times S^{1}\rightarrow M$
such that $\overline{f}^{-}\circ(\theta^{-})^{-1}(-1/2,t)=x(Tt)$. 
\item The maps $\overline{f}_{n}^{+}\circ(\theta_{n}^{+})^{-1}:[1/2,1]\times S^{1}\rightarrow M$
converge in $C^{0}$ to a map $\overline{f}^{+}\circ(\theta^{+})^{-1}:[1/2,1]\times S^{1}\rightarrow M$
such that $\overline{f}^{+}\circ(\theta^{+})^{-1}(1/2,t)=x(Tt)$. 
\item Let $\overline{r}_n^-:=\inf_{t\in S^1} \overline{a}_n(-sgn(T)R_n,t)$ and $\overline{r}_n^+:=\sup_{t\in S^1} \overline{a}_n(sgn(T)R_n,t)$, where $sgn(T):=T/|T|\in \{\pm 1\}$. 
Then $\overline{r}_n^+-\overline{r}_n^-\rightarrow \infty$ and for any $R>0$, there exist $\rho>0$ and $N\in\mathbb{N}$ such that
$\overline{a}_{n}\circ\theta_{n}^{-1}(s,t)\in[\overline{r}_{n}^{-}+R,\overline{r}_{n}^{+}-R]$
for all $n\geq N$ and all $(s,t)\in[-\rho,\rho]\times S^{1}$. 
\end{singlespace}
\end{enumerate}
\end{thm}
\begin{proof}
\begin{singlespace}
\noindent As in Theorem \ref{thm:With-the-same-3} we prove only the
first and second statements of the $C_{\text{loc}}^{\infty}-$convergence.
Let $h_{n}\in\mathbb{R}_{>0}$ be a sequence satisfying $h_{n}<R_{n}$
and $h_{n},R_{n}/h_{n}\rightarrow\infty$ as $n\rightarrow\infty$. To
prove the first statement we consider the shifted maps $\overline{u}_{n}(\cdot+s_{n},\cdot)$,
defined on $[-R_{n}+h_{n}-s_{n},R_{n}-h_{n}-s_{n}]\times S^{1}$,
for any sequence $s_{n}\in[-R_{n}+h_{n},R_{n}-h_{n}]$. If
$s_{n}/R_{n}$ converges, $\overline{u}_{n}(s+s_{n},t)-Ts_{n}$ converges
in $C_{\text{loc}}^{\infty}$ to the trivial cylinder $(Ts,x(Tt))$
over the Reeb orbit $x(Tt)$. Indeed, any subsequence has by Corollary
\ref{cor:There-exists-a-1} and Lemma \ref{lem:For-every--1} a further subsequence converging to this
trivial cylinder. To prove the second statement, we consider the
shifted maps $\overline{u}_{n}^{-}:[0,h_{n}]\times S^{1}\rightarrow\mathbb{R}\times M$,
defined by $\overline{u}_{n}^{-}(s,t)=\overline{u}_{n}(s-R_{n},t)+R_n T$.
By Lemma \ref{lem:For-every--1}, a further subsequence $\overline{u}_{n}^{-}$ converges
in $C_{\text{loc}}^{\infty}([0,\infty)\times S^{1})$ to a usual pseudoholomorphic
curve $\overline{u}^{-}=(\overline{a}^{-},\overline{f}^{-}):[0,+\infty)\times S^{1}\rightarrow\mathbb{R}\times M$
with respect to the standard complex structure $i$ on $[0,+\infty)\times S^{1}$
and the almost complex structure $J_{-\tau}$ on the domain,\textbf{
}where \textbf{$\tau=\lim_{n\rightarrow\infty}P_{n}R_{n}$}. (Note that the $a$-coordinate of $\overline{u}_n^-$ converges by Corollary \ref{cor:There-exists-a-1}
without any further shifts.) We show
that $\overline{u}^{-}$ has the same asymptotics as the trivial cylinder $(Ts,x(Tt))$ over
the Reeb orbit $x$. Since the occuring maps have uniform gradient bounds, it is sufficient to show pointwise convergence 
\begin{equation}
\lim_{s\rightarrow\infty}\left(\overline{a}^{-}(s,t)-Ts,\overline{f}^{-}(s,t)\right)=(0,x(Tt))\label{eq:asymptotic_behaviour}
\end{equation}
to establish convergence in $C^\infty(S^1,\mathbb{R}\times M)$ (indeed, any subsequence has a further subsequence converging in $C^\infty$ to the pointwise limit).
To show the pointwise convergence, we argue by contradiction. Assume that there exists a sequence $(s_{k},t_{k})\in[0,\infty)\times S^{1}$
with $s_{k}\rightarrow\infty$ as $k\rightarrow\infty$, and since
$S^{1}$ is compact, also assume that $t_{k}\rightarrow t^{*}$ as
$k\rightarrow\infty$ such that $\lim_{k\rightarrow\infty}\overline{f}^{-}(s_{k},t_{k})=x'(T't^{*})\in M$,
where $x'$ is some Reeb orbit with $w':=x'(T't^{*})\not=w:=x(Tt^{*})$.
Letting $\epsilon:=\text{dist}_{\overline{g}_{0}}(w,w')>0$, using
Corollary \ref{cor:There-exists-a-1} and employing the same arguments
as in Theorem \ref{thm:With-the-same-3} we are led to the contradiction
$\text{dist}_{\overline{g}_{0}}(w,w')\leq3\epsilon/10$. Now consider
 the $\mathbb{R}-$coordinate $\overline{a}_{n}$. To prove (\ref{eq:asymptotic_behaviour})
for the $\mathbb{R}-$coordinate it is sufficient to replace $\overline{f}^{-}$
by the function $\overline{a}^{-}(s,t)-Ts$ and to repeat the
above arguments.\\
If now $\theta_{n}:[-R_{n},R_{n}]\times S^{1}\rightarrow[-1,1]\times S^{1}$ is any sequence of diffeomorphisms as in Remark \ref{rem:For-every-sequence}, then the maps 
$(\theta_{n}^{-})^{-1}:[-1,-1/2)\times S^1\rightarrow[0,h_{n}]\times S^1$
converge in $C_{\text{loc}}^{\infty}$ to the diffeomorphism $(\theta^{-})^{-1}:[-1,-1/2)\times S^1\rightarrow[0,+\infty)\times S^1$,
the maps $\overline{u}_{n}^{-}\circ(\theta_{n}^{-})^{-1}(s,t)$ converge
in $C_{\text{loc}}^{\infty}$ to the map $\overline{u}^{-}\circ(\theta^{-})^{-1}(s,t)$
on $[-1,-1/2)\times S^{1}$. This finishes the proof of the $C_{\text{loc}}^{\infty}-$convergence.

\noindent To prove the first statement of the $C^{0}-$convergence,
we use Corollary \ref{cor:There-exists-a-1} which yields for any $\varepsilon>0$ some $N\in \mathbb{N}$ so that for $n\geq N$ 
$\text{dist}_{\overline{g}_{0}}(\overline{f}_{n}\circ \theta_{n}^{-1}(s,t),x(Tt))<\varepsilon$
for all $(s,t)\in[-1/2,1/2]\times S^{1}$, thus
$\overline{f}_{n}$ converge in $C^{0}([-1/2,1/2]\times S^{1},M)$ to
$x(T\cdot)$.\\
For the second statement we take into account
that the maps $\overline{f}_{n}^{-}\circ (\theta_{n}^{-})^{-1}(s,t)$
converge in $C_{\text{loc}}^{\infty}$ to $\overline{f}^{-}\circ(\theta^{-})^{-1}(s,t)$
on $[-1,-1/2)\times S^{1}$, so that by the asymptotics of $\overline{f}^{-}$,
$\overline{f}^{-}$ can be continuously extended to the compact cylinder
$[-1,-1/2]\times S^1$ by setting $\overline{v}^{-}(-1/2,t)=x(Tt)$. Now, the
proof of the convergence of $\overline{f}_{n}^{-}$ in $C^{0}([-1,-1/2])$
to $\overline{f}^{-}$ is exactly the same as the proof of Lemma 4.16
in \cite{key-10}. For the maps $\overline{v}_{n}^{+}$ we proceed
analogously, while for the fourth statement we refer to Proposition \ref{prop:For-every-,} below.
Thus the proof of the $C^{0}-$convergence is complete. 
\end{singlespace}
\end{proof}
\begin{prop}
\begin{singlespace}
\noindent \label{prop:For-every-,}Let $\overline{r}_n^-:=\inf_{t\in S^1} \overline{a}_n(-sgn(T)R_n,t)$ and $\overline{r}_n^+:=\sup_{t\in S^1} \overline{a}_n(sgn(T)R_n,t)$, where \\$sgn(T):=T/|T|\in \{\pm 1\}$. 
Then $\overline{r}_n^+-\overline{r}_n^-\rightarrow \infty$ and for any $R>0$, there exist $\rho>0$
and $N\in\mathbb{N}$ such that $\overline{a}_{n}\circ\theta_{n}^{-1}(s,t)\in[\overline{r}_{n}^{-}+R,\overline{r}_{n}^{+}-R]$
for all $n\geq N$ and all $(s,t)\in[-\rho,\rho]\times S^{1}$. 
\end{singlespace}
\end{prop}
\begin{proof}
\begin{singlespace}
\noindent The first statement was observed above when discussing (\ref{eq:a-coordinate-diverges}) and the rest of the result is obtained by following exactly the steps from Lemma 4.10, Lemma
4.13, and Lemma 4.17 of \cite{key-10}. 
\end{singlespace}
\end{proof}
\begin{singlespace}
\noindent Now we give a proof of Theorem \ref{thm:With-the-same-1},
which closely follows the proof of Theorem \ref{thm:With-the-same}. 
\end{singlespace}
\begin{proof}
\begin{singlespace}
\noindent \emph{(of Theorem \ref{thm:With-the-same-1})} 

Consider a sequence of $\mathcal{H}-$holomorphic cylinders $u_{n}=(a_{n},f_{n}):[-R_{n},R_{n}]\times S^{1}\rightarrow\mathbb{R}\times M$
satisfying $A0$-$A3$ with harmonic
perturbation $1-$forms $\gamma_{n}$, $R_{n}\rightarrow\infty$ and positive center action.
As in Section \ref{subsec:Notion-of-the} we transform the map $u_{n}$
into a $\overline{J}_{P_n}-$holomorphic curve $\overline{u}_{n}$ with
respect to the domain-dependent almost complex structure $\overline{J}_{P_n}$.
We consider the new sequence of maps $\overline{f}_{n}$ defined by
$\overline{f}_{n}(s,t):=\phi_{P_{n}s}^{\alpha}(f_{n}(s,t))$ for all
$n\in\mathbb{N}$. Thus $\overline{u}_{n}=(\overline{a}_{n},\overline{f}_{n}):[-R_{n},R_{n}]\times S^{1}\rightarrow\mathbb{R}\times M$
is a $\overline{J}_{P_{n}}-$holomorphic curve. 
Due to Remark \ref{rem:In-the-following} the triple $(\overline{u}_{n},R_{n},P_{n})$
is a $\overline{J}_{P_{n}}-$holomorphic curve as in Definition \ref{def:A-triple-} satisfying $\overline{A}0$-$\overline{A}3$ with positive center action.
After shifting $u_{n}$ by $-a_{n}(0,0)$ we obtain that $a_{n}(0,0)=0$.\\
Now by (\ref{eq:bar=00003D00007Ba=00003D00007D})
$a_{n}$ can be written as $a_{n}=\overline{a}_{n}-\Gamma_{n}$, where
$\Gamma_{n}:[-R_{n},R_{n}]\times S^{1}\rightarrow\mathbb{R}$ is a
normalized harmonic function with $\|d\Gamma_n\|^2_{L^2([-R_n,R_n]\times S^1)}\leq (\sqrt{C_0}+\sqrt{2|P_n|C})^2$ (c.f. Remark \ref{rem:Obviously,-as-}) and $R_n\rightarrow \infty$. 
Thus if $R_n\geq 1$ the functions $\Gamma_n$ satisfy the assumptions $C1$-$C5$ from
Appendix \ref{sec:Convergence-of} with $B=(\sqrt{C_0}+\sqrt{2C^2})^2$, since $|P_n|\leq C$ by $A3$ if $R_n\geq 1$. 
We can thus by Lemma \ref{lem:decompose_harmonic} further write $\Gamma_n=S_ns+\tilde{\Gamma}_n$, where the harmonic functions $\tilde{\Gamma}_n$ 
have uniform $C^k$-bounds on $[-R_n+1,R_n-1]\times S^1$ 
by Lemma \ref{lem:kbounds_tilde_Gamma}.\\
Actually Proposition \ref{prop:For-every-} shows that $\tilde{\Gamma}_n(0,0)\rightarrow 0$ and thus $\overline{a}_{n}(0,0)\rightarrow 0$.
Therefore after passing
to a subsequence and by using that $M$ is compact, we can assume that $\overline{u}_{n}(0,0)\rightarrow w=(0,w_{f})\in\mathbb{R}\times M$
as $n\rightarrow\infty$.\\
We pick a subsequence, such that the convergence results from Theorem \ref{thm:With-the-same-1-1} hold for this sequence $\overline{u}_n$ and such that
that $P_{n}R_{n}\rightarrow\tau$ and $S_{n}R_{n}\rightarrow\sigma$
for $\tau,\sigma\in\mathbb{R}$ (the latter statements are possible by assumption A3).

We start
by proving the first statement of the $C_{\text{loc}}^{\infty}-$convergence.
Let $h_{n}\in\mathbb{R}_{>0}$ be a sequence satisfying $h_{n}<R_{n}$
and $h_{n},R_{n}/h_{n}\rightarrow\infty$ as $n\rightarrow\infty$. As
in the proof of Theorem \ref{thm:With-the-same} we consider for $(s,t)\in[-R_{n}+h_{n},R_{n}-h_{n}]\times S^{1}$
the maps (cf. (\ref{eq:f-part-1})) 
\begin{equation}
f_{n}(s,t)=\phi_{-P_{n}s}^{\alpha}(\overline{f}_{n}(s,t))\text{ and }a_{n}(s,t)=\overline{a}_{n}(s,t)-\Gamma_{n}(s,t),\label{eq:f-part-1-1}
\end{equation}
and that by Remark \ref{rem:Obviously,-as-}, the functions $\Gamma_{n}$
can be chosen to have vanishing average.
By the first $C^\infty_{loc}$ convergence results in Theorem \ref{thm:With-the-same-1-1} and Theorem \ref{thm:The-maps--1},
for any sequence $s_{n}\in[-R_{n}+h_{n},R_{n}-h_{n}]$ with $s_{n}/R_{n}\rightarrow\kappa\in[-1,1]$,
the shifted maps $u_{n}(\cdot+s_{n},\cdot)-Ts_{n}+S_{n}s_{n}$, defined
on $[-R_{n}+h_{n}-s_{n},R_{n}-h_{n}-s_{n}]\times S^{1}$, converge
in $C_{\text{loc}}^{\infty}$ to the trivial cylinder $(Ts,\phi_{-\tau\kappa}^{\alpha}(x(Tt))=x(Tt-\tau\kappa))$
over the Reeb orbit $x(Tt-\kappa \tau)$.\\
To prove the second statement of the $C_{\text{loc}}^{\infty}-$convergence,
we consider the shifted maps $u_n^-(s,t):=u_n(s-R_n,t)+R_n T$ and the shifted maps $\overline{u}_{n}^{-}:[0,h_{n}]\times S^{1}\rightarrow\mathbb{R}\times M$
which are defined by $\overline{u}_{n}^{-}(s,t)=\overline{u}_{n}(s-R_{n},t)+R_nT$. They are related by (\ref{eq:fnminusbar})
and thus by (\ref{eq:fnminus}). By Theorems \ref{thm:With-the-same-1-1}
and \ref{thm:The-maps--1}, $u_{n}^{-}$ converge therefore in $C_{\text{loc}}^{\infty}$
to a curve $u^{-}(s,t)=(a^{-}(s,t),f^{-}(s,t))=(\overline{a}^{-}(s,t)-\Gamma^{-}(s,t),\phi_{\tau}^{\alpha}(\overline{f}^{-}(s,t)))$,
defined on $[0,\infty)\times S^{1}$. The map $u^{-}$ has the same asymptotes as the
trivial cylinder $(Ts+\sigma,\phi_{\tau}^{\alpha}(x(Tt))=x(Tt+\tau))$ as $s\rightarrow \infty$,
and it is (as the limit of the maps $u_n$) a $\mathcal{H}-$holomorphic map with harmonic
perturbation $d\Gamma^{-}$. 

\noindent Now fix a sequence
of diffeomorphisms $\theta_{n}=\check{\theta}_n\times id_{S^1}:[-R_{n},R_{n}]\times S^1\rightarrow[-1,1] \times S^1$
as in Remark \ref{rem:For-every-sequence}. To prove the first statement of the $C^{0}-$convergence,
we consider the maps $f_{n}$ satisfying $\overline{f}_{n}(s,t)=\phi_{P_{n}s}^{\alpha}(f_{n}(s,t))$
and 
\[
f_{n}(s,t)=f_{n}\circ \theta_{n}^{-1}(s,t)=\phi_{-P_{n}\check{\theta}_{n}^{-1}(s)}^{\alpha}(\overline{f}_{n}\circ(\theta_{n}^{-1}(s,t))
\]
for $s\in[-1/2,1/2]$. There exists a constant $c>0$ such that for
all $(s,t)\in[-1/2,1/2]$ we have 
\begin{align}
\text{dist}_{\overline{g}_{0}}(\overline{f}_{n}\circ\theta_{n}^{-1}(s,t),x(Tt))\geq c\text{dist}_{\overline{g}_{0}}(f_{n}\circ \theta_{n}^{-1}(s,t),\phi_{-P_{n}\check{\theta}_{n}^{-1}(s)}^{\alpha}(x(Tt)))
\label{eq:estimate_bla}\end{align}
The triangle inequality shows that we have for $(s,t)\in[-1/2,1/2]\times S^{1}$
\begin{align*}
\text{dist}_{\overline{g}_{0}}(f_{n}\circ \theta_{n}^{-1}(s,t),\phi_{-2\tau s}^{\alpha}(x(Tt)))\leq 
\text{dist}_{\overline{g}_{0}}(\overline{f}_{n}\circ \theta_{n}^{-1}(s,t),\phi_{-P_{n}\check{\theta}_{n}^{-1}(s)}^{\alpha}(x(Tt)))+
\text{dist}_{\overline{g}_{0}}(\phi_{-P_{n}\check{\theta}_{n}^{-1}(s)}^{\alpha}(x(Tt)),\phi_{-2\tau s}^{\alpha}(x(Tt))).
\end{align*}
The estimate (\ref{eq:estimate_bla}) together with the first $C^0$-convergence result in Theorem  \ref{thm:With-the-same-1-1} 
resp. (\ref{eq:PnmM}) together with the subsequent discussion show, that the terms
on the right hand side tend to $0$ as $n\rightarrow\infty$. Thus $f_{n}$ converge in $C^{0}$ on $[-1/2,1/2]\times S^1$
to $(s,t)\mapsto \phi_{-2\tau s}^{\alpha}(x(Tt))$, which is a segment of shifts of the closed Reeb $x$. \\
To prove the second statement, we consider the maps $f_n^-\circ \theta_n^-$
satisfying 
\[
f_{n}^{-}\circ(\theta_{n}^{-})^{-1}(s,t)=\phi_{-P_{n}(\check{\theta}_{n}^{-})^{-1}(s)+P_{n}R_{n}}^{\alpha}(\overline{f}_{n}^{-}\circ (\theta_{n}^{-})^{-1}(s,t)),
\]
and use the second $C^0$-convergence result in Theorem \ref{thm:With-the-same-1-1} together with $|P_{n}(\check{\theta}_{n}^{-})^{-1}(s)|\leq |P_nh_n|\leq h_n/R_n\rightarrow 0$ 
and $R_nP_n\rightarrow \tau$ to conclude that  $f_n^-\circ \theta_n^-$ converge
in $C^{0}$ on $[-1,-1/2]\times S^1$ to the map $\phi_{\tau}^{\alpha}(\overline{f}^{-}\circ(\theta^{-})^{-1}(s,t))$. It follows from the second $C^0$-convergence result
from Theorem \ref{thm:With-the-same-1-1}, that this map has the desired value $x(T\cdot)$ on $\{-1/2\}\times S^1$.
The third statement is proved in a similar manner. \\
The last
statement follows from Proposition \ref{prop:For-every-,} and the
fact that the harmonic functions $\tilde{\Gamma}_{n}$ are uniformly bounded
in $C^{0}$ by Lemma \ref{lem:kbounds_tilde_Gamma}. To see this in some more detail, with $\overline{\Gamma}_{n}:=\tilde{\Gamma}_n\circ \theta_n^{-1}$, we can write 
\[
a_{n}\circ \theta_{n}^{-1}(s,t)=\overline{a}_{n}\circ \theta_{n}^{-1}(s,t)-S_{n}\check{\theta}_{n}^{-1}(s)-\overline{\Gamma}_{n}(s,t)
\]
for $(s,t)\in[-1,1]\times S^{1}$. First Lemma \ref{lem:kbounds_tilde_Gamma} implies that the functions $\tilde{\Gamma}_n$ and thus the functions $\overline{\Gamma}_{n}$
are uniformly bounded in $C^{0}([-1,1]\times S^{1})$ by some $K<\infty$.
We get
\[
-|S_{n}|R_{n}-K\leq S_{n}\check{\theta}_{n}^{-1}(s)+\overline{\Gamma}_{n}(s,t)\leq |S_{n}|R_{n}+K
\]
for all $s\in[-1,1]$. On the other hand, from Proposition \ref{prop:For-every-,}
we have that for every $R>0$ there exist $\rho>0$ and $N\in\mathbb{N}$
such that $\overline{a}_{n}\circ \theta_{n}^{-1}(s,t)\in[\overline{r}_{n}^{-}+R,\overline{r}_{n}^{+}-R]$
for all $n\geq N$ and all $(s,t)\in[-\rho,\rho]\times S^{1}$. Thus
we obtain 
\[
a_{n}\circ \theta_{n}^{-1}(s,t)\in[\overline{r}_{n}^{-}-|S_{n}|R_{n}-K+R,\overline{r}_{n}^{+}+|S_{n}|R_{n}+K-R]
\]
for all $n\geq N$ and all $(s,t)\in[-\rho,\rho]\times S^{1}$. Since $S_{n}R_{n}\rightarrow\sigma$ as $n\rightarrow\infty$
we find for $K':=|\sigma|+1+K$
\[
a_{n}\circ \theta_{n}^{-1}(s,t)\in[\overline{r}_{n}^{-}-K'+R,\overline{r}_{n}^{+}+K'-R]
\]
for all $n\geq N$ and all $(s,t)\in[-\rho,\rho]\times S^{1}$. 
Since the constant $K'>0$ is independent of $R>0$ we have that for
every $R>0$ there exists a $\rho>0$ and $N\in\mathbb{N}$ such that
\begin{equation}
a_{n}\circ \theta_{n}^{-1}(s,t)\in[\overline{r}_{n}^{-}+R,\overline{r}_{n}^{+}-R]\label{eq:barr}
\end{equation}
for all $n\geq N$ and every $(s,t)\in[-\rho,\rho]\times S^{1}$.
Let
\begin{align}
r_{n}^{-} & :=\inf_{t\in S^{1}}a_{n}(\check{\theta}_{n}^{-1}(-sgn(T)),t)=\inf_{t\in S^{1}}\left[\overline{a}_{n}(\check{\theta}_{n}^{-1}(-sgn(T)),t)-
(S_{n}\check{\theta}_{n}^{-1}(-sgn(T))+\overline{\Gamma}_{n}(-sgn(T),t))\right],\label{eq:tildeR}\\
r_{n}^{+} & :=\sup_{t\in S^{1}}a_{n}(\check{\theta}_{n}^{-1}(sgn(T)),t)=\sup_{t\in S^{1}}\left[\overline{a}_{n}(\check{\theta}_{n}^{-1}(sgn(T)),t)-(S_{n}
\check{\theta}_{n}^{-1}(sgn(T))+\overline{\Gamma}_{n}(sgn(T),t))\right].\nonumber 
\end{align}
Since for large $n$, $\sup_{t\in S^1}|S_{n}\check{\theta}_{n}^{-1}(\pm 1)+\overline{\Gamma}_{n}(\pm 1,t)|\leq K'$, the last statement of the $C^{0}-$convergence follows.
The proof of Theorem \ref{thm:With-the-same-1} is finished. 
\end{singlespace}
\end{proof}

\begin{proof}
\begin{singlespace}
\noindent \emph{(of Corollary \ref{cor:Under-the-same} and Corollary \ref{cor:underthesame_1})}
The convergence $v_n^-\rightarrow v^-$ follows directly from the corresponding results in Theorem \ref{thm:With-the-same} and Theorem \ref{thm:With-the-same-1}.
The convergence of the harmonic perturbations $\gamma_{n}$ follows
from Corollary \ref{cor:After-going-over}, while the convergence
of $v_{n}^{+}$ is proved in an analogous manner. 
\end{singlespace}
\end{proof}

\appendix

\section{\label{sec:Convergence-of}Compactness of harmonic cylinders}

\begin{singlespace}
\noindent In this section we describe the $C_{\text{loc}}^{\infty}$-
and $C^{0}$- convergence of a sequence of harmonic functions $\Gamma_{n}$
on cylinders $[-R_{n},R_{n}]\times S^{1}$. This result is used in
the proof of Theorems \ref{thm:With-the-same} and \ref{thm:With-the-same-1}.
The analysis is performed in the following setting: 
\end{singlespace}
\begin{description}
\begin{singlespace}
\item [{C1}]\noindent $R_{n}\rightarrow\infty$;
\item [{C2}] \noindent $\Gamma_{n}$ is a harmonic function on $[-R_{n},R_{n}]\times S^{1}$.
(This implies that $d\Gamma_n$ is a harmonic $1-$form with respect to the
standard complex structure $i$ on $\mathbb{R}\times S^{1}$.)
\item [{C3}] \noindent $\Gamma_{n}$ has vanishing average over the cylinder
$[-R_{n},R_{n}]\times S^{1}$, i.e. for all $n\in\mathbb{N}$ we have
\[
\frac{1}{2R_{n}}\int_{[-R_{n},R_{n}]\times S^{1}}\Gamma_{n}(s,t)dsdt=0;
\]
\item [{C4}] \noindent the $L^{2}-$norm of $d\Gamma_{n}$ is uniformly
bounded, i.e. there exists a constant $B>0$ such that 
\[
\left\Vert d\Gamma_{n}\right\Vert _{L^{2}([-R_{n},R_{n}]\times S^{1})}^{2}:=\int_{[-R_{n},R_{n}]\times S^{1}}d\Gamma_{n}\circ i\wedge d\Gamma_{n}\leq B
\]
for all $n\in\mathbb{N}$. 

\item [{C5}] \noindent The co-period $S_n$ of $d\Gamma_n$ satisfy $ |R_nS_n|\leq C$.

\end{singlespace}
\end{description}
\begin{singlespace}
\noindent The subsequent lemma gives a decomposition of $\Gamma_{n}$
in a linear term and a harmonic function satisfying properties C1-C4
and having a uniformly bounded $L^{2}-$norm. 
\end{singlespace}
\begin{lem}\label{lem:decompose_harmonic}
\begin{singlespace}
\noindent \label{lem:For-the-harmonic}There exists a sequence $S_{n}\in\mathbb{R}$
with $|S_{n}|\leq\sqrt{B/2R_{n}}$ such that the harmonic function
$\Gamma_{n}:[-R_{n},R_{n}]\times S^{1}\rightarrow\mathbb{R}$ can
be decomposed as $\Gamma_{n}(s,t)=S_{n}s+\tilde{\Gamma}_{n}(s,t)$,
where $S_n$ is the co-period and \\$\tilde{\Gamma}_{n}:[-R_{n},R_{n}]\times S^{1}\rightarrow\mathbb{R}$
is a harmonic function satisfying properties C1-C4 and additionally
\begin{equation}
\left\Vert \tilde{\Gamma}_{n}\right\Vert _{L^{2}([-R_{n},R_{n}]\times S^{1})}^{2}\leq\left\Vert d\tilde{\Gamma}_{n}\right\Vert _{L^{2}([-R_{n},R_{n}]\times S^{1})}^{2}.\label{eq:L2estimate}
\end{equation}
\end{singlespace}
\end{lem}
\begin{proof}
\begin{singlespace}
\noindent We consider the Fourier series of the harmonic function
$\Gamma_{n}$, i.e. 
\begin{align*}
\Gamma_{n}(s,t) & =\sum_{k\in\mathbb{Z}}c_{k}^{n}(s)e^{2\pi ikt}=c_{0}^{n}(s)+\sum_{k\in\mathbb{Z}\backslash\{0\}}c_{k}^{n}(s)e^{2\pi ikt}.
\end{align*}
Because $\Gamma_{n}$ has vanishing mean value, we have

\noindent 
\begin{equation}
0=\int_{[-R_{n},R_{n}]\times S^{1}}\Gamma_{n}(s,t)dsdt=\int_{-R_{n}}^{R_{n}}\int_{0}^{1}\Gamma_{n}(s,t)dtds=\int_{-R_{n}}^{R_{n}}c_{0}^{n}(s)ds.\label{eq:coefficient_vanishing_average}
\end{equation}
As $\Gamma_{n}$ is a harmonic function, the coefficients $c_{k}^{n}$
can be readily computed; we find 
\[
c_{k}^{n}(s)=\begin{cases}
A_{k}^{n}\sinh(2\pi ks)+B_{k}^{n}\cosh(2\pi ks), & k\in\mathbb{Z}\backslash\{0\}\\
S_{n}s+d_{n}, & k=0
\end{cases},
\]
where $A_{k}^{n},B_{k}^{n},S_{n},d_{n}\in\mathbb{C}$. By (\ref{eq:coefficient_vanishing_average}),
$d_{n}=0$, and the Fourier expansion of $\Gamma_{n}$ takes the form
\begin{align*}
\Gamma_{n}(s,t) & =S_{n}s+\sum_{k\in\mathbb{Z}\backslash\{0\}}c_{k}^{n}(s)e^{2\pi ikt}=S_{n}s+\tilde{\Gamma}_{n}(s,t),
\end{align*}
where 
\begin{equation}
\tilde{\Gamma}_{n}(s,t)=\Gamma_{n}(s,t)-S_{n}s=\sum_{k\in\mathbb{Z}\backslash\{0\}}c_{k}^{n}(s)e^{2\pi ikt}.\label{eq:tilde_gamma}
\end{equation}
For every $s\in[-R_{n},R_{n}]$ we have 
\begin{align*}
S_{n}=\int_{\{s\}\times S^{1}}\partial_{s}\Gamma_{n}(s,t)dt\in\mathbb{R}
\end{align*}
and so, $\tilde{\Gamma}_{n}$ is a real valued harmonic function.
On the other hand by Hölder inequality we find the estimate 
\begin{align*}
|S_{n}|\leq\frac{1}{2R_{n}}\int_{[-R_{n},R_{n}]\times S^{1}}|\partial_{s}\Gamma_{n}(s,t)|dsdt\leq\sqrt{\frac{B}{2R_{n}}}.
\end{align*}
Now we show that $d\tilde{\Gamma}_{n}$ has a uniform $L^{2}-$bound.
By (\ref{eq:tilde_gamma}) and Hölder inequality we get 
\begin{align*}
\left\Vert d\tilde{\Gamma}_{n}\right\Vert _{L^{2}([-R_{n},R_{n}]\times S^{1})}^{2} & =
\left\Vert d\Gamma_{n}\right\Vert _{L^{2}([-R_{n},R_{n}]\times S^{1})}^{2}-2S_{n}\int_{[-R_{n},R_{n}]\times S^{1}}d\Gamma_{n}\circ i\wedge ds\\
 & \ \ +2S_{n}^{2}R_{n}\\
 & \leq4B.
\end{align*}
Thus $\tilde{\Gamma}_{n}$ satisfies the property C4 from above, and
obviously, properties C1-C3. Next we prove estimate (\ref{eq:L2estimate}).
By (\ref{eq:tilde_gamma}), the $L^{2}-$norm of $\tilde{\Gamma}_{n}$
can be computed as follows 
\begin{align*}
\left\Vert \tilde{\Gamma}_{n}\right\Vert _{L^{2}([-R_{n},R_{n}]\times S^{1})}^{2} & =\sum_{k\in\mathbb{Z}\backslash\{0\}}\left\Vert c_{k}^{n}\right\Vert _{L^{2}([-R_{n},R_{n}])}^{2}.
\end{align*}
On the other hand we have 
\[
\partial_{t}\tilde{\Gamma}_{n}(s,t)=\sum_{k\in\mathbb{Z}\backslash\{0\}}2\pi ikc_{k}^{n}(s)e^{2\pi ikt}
\]
and 
\begin{align*}
\left\Vert \partial_{t}\tilde{\Gamma}_{n}\right\Vert _{L^{2}([-R_{n},R_{n}]\times S^{1})}^{2}=\sum_{k\in\mathbb{Z}\backslash\{0\}}4\pi^{2}k^{2}\left\Vert c_{k}^{n}\right\Vert _{L^{2}([-R_{n},R_{n}])}^{2}\geq\left\Vert \tilde{\Gamma}_{n}\right\Vert _{L^{2}([-R_{n},R_{n}]\times S^{1})}^{2},
\end{align*}
while from 
\[
\left\Vert \partial_{t}\tilde{\Gamma}_{n}\right\Vert _{L^{2}([-R_{n},R_{n}]\times S^{1})}^{2}\leq\left\Vert d\tilde{\Gamma}_{n}\right\Vert _{L^{2}([-R_{n},R_{n}]\times S^{1})}^{2}
\]
we end up with 
\[
\left\Vert \tilde{\Gamma}_{n}\right\Vert _{L^{2}([-R_{n},R_{n}]\times S^{1})}^{2}\leq\left\Vert d\tilde{\Gamma}_{n}\right\Vert _{L^{2}([-R_{n},R_{n}]\times S^{1})}^{2}.
\]
\end{singlespace}
\end{proof}

\begin{singlespace}
\noindent In particular, we see that for all $n$ we have 
\begin{equation}
\left\Vert \tilde{\Gamma}_{n}\right\Vert _{L^{2}([-R_{n},R_{n}]\times S^{1})}^{2}\leq4B.\label{eq:L2boundsGamma_n}
\end{equation}
The next lemma establishes uniform bounds on the derivatives of $\tilde{\Gamma}_{n}$. 
\end{singlespace}
\begin{lem}\label{lem:kbounds_tilde_Gamma}
\begin{singlespace}
\noindent \label{lem:For-every-}For any $\delta>0$ and $k\in\mathbb{N}_{0}$
there exists a constant $\tilde{K}=\tilde{K}(\delta,k,B)>0$ such
that 
\[
\left\Vert \tilde{\Gamma}_{n}\right\Vert _{C^{k}([-R_{n}^{\delta},R_{n}^{\delta}]\times S^{1})}\leq\tilde{K}
\]
for all $n\in\mathbb{N}$ and $R_{n}^{\delta}:=R_{n}-\delta$. 
\end{singlespace}
\end{lem}
\begin{proof}
\begin{singlespace}
\noindent The function $F_{n}:=\partial_{s}\tilde{\Gamma}_{n}+i\partial_{t}\tilde{\Gamma}_{n}:[-R_{n},R_{n}]\times S^{1}\rightarrow\mathbb{C}$ is holomorphic 
since $\Gamma_n $ is harmonic.
Note that the $ L^{2}-$norm of $F_{n}$ is  uniformly bounded, i.e. 
\begin{equation}
\int_{[-R_{n},R_{n}]\times S^{1}}|F_{n}|^{2}dsdt\leq4B\label{eq:L2normOfF}
\end{equation}
for all $n\in\mathbb{N}$. As $\tilde{\Gamma}_{n}$ is harmonic it
is obvious that 
\[
\Delta|F_{n}|^{2}=2|\nabla F_{n}|^{2}\geq0.
\]
Hence $|F_{n}|^{2}$ is subharmonic. By using the mean value property
for subharmonic functions we conclude that for any $\delta>0$ and
any $z=(s,t)\in[-R_{n}^{\delta/2},R_{n}^{\delta/2}]\times S^{1}$,
\begin{align*}
|F_{n}(z)|^{2} & \leq\frac{32}{\pi\delta^{2}}\int_{B_{\frac{\delta}{4}}(z)}|F_{n}(s,t)|^{2}dsdt
\leq\frac{32}{\pi\delta^{2}}\left\Vert F_{n}\right\Vert_{L^{2}([-R_{n}^{\frac{\delta}{2}},R_{n}^{\frac{\delta}{2}}]\times S^{1})}^{2}.
\end{align*}
Since these estimates hold for all $z\in[-R_{n}^{\delta/2},R_{n}^{\delta/2}]\times S^{1}$
we obtain 
\[
\left\Vert F_{n}\right\Vert _{C^{0}([-R_{n}^{\frac{\delta}{2}},R_{n}^{\frac{\delta}{2}}]\times S^{1})}^{2}\leq\frac{32}{\pi\delta^{2}}\left\Vert F_{n}\right
\Vert _{L^{2}([-R_{n}^{\frac{\delta}{2}},R_{n}^{\frac{\delta}{2}}]\times S^{1})}^{2}.
\]
In particular, by using (\ref{eq:L2normOfF}), we find 
\begin{equation}
\left\Vert F_{n}\right\Vert _{C^{0}([-R_{n}^{\frac{\delta}{2}},R_{n}^{\frac{\delta}{2}}]\times S^{1})}\leq\frac{8\sqrt{2B}}{\delta\sqrt{\pi}}\label{eq:C0normOfF}
\end{equation}
for all $n\in\mathbb{N}$. By the Cauchy integral formula for holomorphic
functions and (\ref{eq:C0normOfF}) we deduce that the derivatives
of $F_{n}$ are uniformly bounded on $[-R_{n}^{\delta},R_{n}^{\delta}]\times S^{1}$.
Indeed, for $k\in\mathbb{N}$ we have 
\begin{align*}
|F_{n}^{(k)}(z)| & =\frac{k!}{2\pi}\left|\int_{\partial B_{\frac{\delta}{2}}(z)}\frac{F_{n}(\xi)}{(\xi-z)^{k+1}}d\xi\right|=\frac{k!}{2\pi}
\left|\int_{0}^{2\pi}2^{k}i\frac{F_{n}(z+\delta e^{it})}{\delta^{k}e^{ikt}}dt\right|\leq\frac{2^{k+3}k!\sqrt{2B}}{\delta^{k+1}\sqrt{\pi}}
\end{align*}
for all $z\in[-R_{n}^{\delta},R_{n}^{\delta}]\times S^{1}$ and $n\in\mathbb{N}$.
Since $z\in[-R_{n}^{\delta},R_{n}^{\delta}]\times S^{1}$ was arbitrary,
we obtain 
\[
\left\Vert F_{n}^{(k)}\right\Vert _{C^{0}([-R_{n}^{\delta},R_{n}^{\delta}]\times S^{1})}\leq\frac{2^{k+3}k!\sqrt{2B}}{\delta^{k+1}\sqrt{\pi}}.
\]
Using (\ref{eq:L2boundsGamma_n}) and the mean value property and
Hölder inequality for $\tilde{\Gamma}_{n}$ we find that for all $z\in[-R_{n}^{\delta},R_{n}^{\delta}]\times S^{1}$,
\begin{align*}
|\tilde{\Gamma}_{n}(z)|\leq\frac{4}{\pi\delta^{2}}\int_{B_{\frac{\delta}{2}}(z)}|\tilde{\Gamma}_{n}(s,t)|dsdt\leq 
\frac{4}{\pi\delta^{2}}\sqrt{\text{Area}(B_{\frac{\delta}{2}}(z))}\|\tilde{\Gamma}_n\|_{L^2(B_{\frac{\delta}{2}}(z))}\leq \frac{4\sqrt{B}}{\delta\sqrt{\pi}}.
\end{align*}
Hence we get 
\[
\left\Vert \tilde{\Gamma}_{n}\right\Vert _{C^{0}([-R_{n}^{\delta},R_{n}^{\delta}]\times S^{1})}\leq\frac{4\sqrt{B}}{\delta\sqrt{\pi}} \quad 
\forall n\in\mathbb{N}. 
\]
\end{singlespace}
\end{proof}
\begin{rem}
\begin{singlespace}
\noindent \label{rem:From-the-proof}
\end{singlespace}
\begin{enumerate}
\begin{singlespace}
\item From the proof of Lemma \ref{lem:For-every-} the following result
can be established: By Arzelà-Ascoli theorem, for any sequence $s_{n}\in[-R_{n}^{\delta},R_{n}^{\delta}]$,
the sequence of maps $F_{n}(\cdot+s_{n},\cdot)$ defined on $[-R_{n}^{\delta}-s_{n},R_{n}^{\delta}-s_{n}]\times S^{1}$,
where $F_{n}=\partial_{s}\tilde{\Gamma}_{n}+i\partial_{t}\tilde{\Gamma}_{n}$,
contains a subsequence, also denoted by $F_{n}(\cdot+s_{n},\cdot)$,
that converges in $C_{\text{loc}}^{\infty}$ to some holomorphic map
$F$; $F$ depends on the sequence $\{s_{n}\}$, has bounded $L^{2}$-
and $C^{0}$- norms, and is defined either on a half cylinder or on
$\mathbb{R}\times S^{1}$. In the later case, when $R_{n}^{\delta}-s_{n}$
and $R_{n}^{\delta}+s_{n}$ diverge, $F$ has to be $0$. Indeed,
by Liouville theorem, $F$ has to be constant, while from the boundedness
of the $L^{2}-$norm we conclude that $F$ is $0$. 
\item By Lemma \ref{lem:For-every-}, (\ref{eq:L2boundsGamma_n}) and the
Liouville theorem for harmonic functions, $\tilde{\Gamma}_{n}(0,\cdot)$
converges to $0$. By Lemma \ref{lem:For-every-} and Remark \ref{rem:From-the-proof},
the sequence of harmonic functions $\tilde{\Gamma}_{n}(\cdot+s_{n},\cdot)$
with $s_{n}\in[R_{n}^{\delta},R_{n}^{\delta}]$, contains a subsequence
that converges in $C_{\text{loc}}^{\infty}$ to some harmonic function
defined either on a half cylinder or on $\mathbb{R}\times S^{1}$.
In the later case the limit harmonic function has to be $0$ by the
same arguments as above. 
\end{singlespace}
\end{enumerate}
\end{rem}
\begin{singlespace}
\noindent To simplify notation we drop the index $\delta$. We define
the harmonic functions $\tilde{\Gamma}_{n}^{-}:[0,2R_{n}]\times S^{1}\rightarrow\mathbb{R}$
and $\tilde{\Gamma}_{n}^{+}:[-2R_{n},0]\times S^{1}\rightarrow\mathbb{R}$
by $\tilde{\Gamma}_{n}^{-}(s,t):=\tilde{\Gamma}_{k}(s-R_{n},t)$ and
$\tilde{\Gamma}_{n}^{+}(s,t):=\tilde{\Gamma}_{n}(s+R_{n},t)$, respectively.
By Lemma \ref{lem:For-every-}, there exist harmonic functions
$\tilde{\Gamma}^{-}:[0,+\infty)\times S^{1}\rightarrow\mathbb{R}$
and $\tilde{\Gamma}^{+}:(-\infty,0]\times S^{1}\rightarrow\mathbb{R}$
such that $\tilde{\Gamma}_{n}^{-}\stackrel{C_{\text{loc}}^{\infty}}{\longrightarrow}\tilde{\Gamma}^{-}$
and $\tilde{\Gamma}_{n}^{+}\stackrel{C_{\text{loc}}^{\infty}}{\longrightarrow}\tilde{\Gamma}^{+}$.
The next proposition plays an important role in establishing a $C_{\text{loc}}^{\infty}$\textendash{}
and $C^{0}$\textendash convergence of the harmonic functions $\tilde{\Gamma}_{n}$. 
\end{singlespace}
\begin{prop}
\begin{singlespace}
\noindent \label{prop:For-every-}For any $\epsilon>0$ there exists
$h>0$ such that for any $R_{n}>h$ we have 
\[
\left\Vert \tilde{\Gamma}_{n}\right\Vert _{C^{0}([-R_{n}+h,R_{n}-h]\times S^{1})}<\epsilon.
\]
\end{singlespace}
\end{prop}
\begin{proof}
\begin{singlespace}
\noindent Assume that this is not the case. Then there exist $\epsilon_{0},C_{0}>0$
such that for any $h_{k}:=k$ there exist $R_{n_{k}}>k$ and a sequence
$(s_{k},t_{k})\in[-R_{n_{k}}+k,R_{n_{k}}-k]\times S^{1}$ such that
$|\tilde{\Gamma}_{n_{k}}(s_{k},t_{k})|\geq\epsilon_{0}$. From $s_{k}\in[-R_{n_{k}}+k,R_{n_{k}}-k]$
it follows that $|R_{n_{k}}-s_{k}|\rightarrow\infty$ as $k\rightarrow\infty$.
Consider the harmonic functions $H_{k}:[-R_{n_{k}}-s_{k},R_{n_{k}}-s_{k}]\times S^{1}\rightarrow\mathbb{R}$
defined by $H_{k}(s,t)=\tilde{\Gamma}_{n_{k}}(s+s_{k},t)$. Obviously,
we have $H_{k}(0,t_{k})=\tilde{\Gamma}_{n_{k}}(s_{k},t_{k})$ and
by Remark \ref{rem:From-the-proof} we conclude that the $H_{k}$
converge in $C_{\text{loc}}^{\infty}$ to some harmonic function $H:\mathbb{R}\times S^{1}\rightarrow\mathbb{R}$
with bounded $L^{2}$ and $C^{0}-$norms. By the Liouville theorem
for harmonic function, $H\equiv0$. This gives a contradiction to
$|H_{k}(0,t_{k})|=|\tilde{\Gamma}_{n_{k}}(s_{k},t_{k})|\geq\epsilon_{0}$,
and the proof is finished. 
\end{singlespace}
\end{proof}
\begin{cor}
\begin{singlespace}
\noindent \label{cor:Thus-we-may}For every sequence $h_{n}\in\mathbb{R}_{>0}$
satisfying $h_{n}<R_{n}$ and $h_{n},R_{n}/h_{n}\rightarrow\infty$
and every $\epsilon>0$ there exists $N\in\mathbb{N}$ such that 
\[
\left\Vert \tilde{\Gamma}_{n}\right\Vert _{C^{0}([-R_{n}+h_{n},R_{n}-h_{n}]\times S^{1})}<\epsilon
\]
for all $n\geq N$. Moreover, for the co-period $S_{n}$ we obtain
that $h_{n}S_{n}\rightarrow0$ as $n\rightarrow\infty$. 
\end{singlespace}
\end{cor}
\begin{proof}
\begin{singlespace}
\noindent Consider a sequence $h_{n}\in\mathbb{R}_{>0}$ with $h_{n}<R_{n}$
and $h_{n},R_{n}/h_{n}\rightarrow\infty$ as $n\rightarrow\infty$
and let $\epsilon>0$ be given. By Proposition \ref{prop:For-every-},
there exist $h_{\epsilon}>0$ and $N_{\epsilon}\in\mathbb{N}$ such
that for all $n\geq N_{\epsilon}$ we have $R_{n}>h_{\epsilon}$ and
$\left\Vert \tilde{\Gamma}_{n}\right\Vert _{C^{0}([-R_{n}+h_{\epsilon},R_{n}-h_{\epsilon}]\times S^{1})}<\epsilon$.
By taking $N_{\epsilon}$ sufficiently large and since $h_{n}\rightarrow\infty$
we may assume that for all $n\geq N_{\epsilon}$, we have $R_{n}>h_{n}>h_{\epsilon}$,
giving $\left\Vert \tilde{\Gamma}_{n}\right\Vert _{C^{0}([-R_{n}+h_{n},R_{n}-h_{n}]\times S^{1})}<\epsilon$.
By assumption C5 it follows from $R_{n}S_{n}\rightarrow\sigma$
and $R_{n}/h_{n}\rightarrow\infty$ as $n\rightarrow\infty$ that
$h_{n}S_{n}\rightarrow0$ as $n\rightarrow\infty$. 
\end{singlespace}
\end{proof}
\begin{singlespace}
\noindent  
Now we focus on the $C^{0}-$convergence
of the maps $\tilde{\Gamma}_{n}$ as $n\rightarrow\infty$.
Let $h_{n}\in\mathbb{R}_{>0}$ be a sequence with
$h_{n}<R_{n}$ and $h_{n},R_{n}/h_{n}\rightarrow\infty$ and
fix a sequence of diffeomorphisms $\theta_{n}$ as
in Remark \ref{rem:For-every-sequence}. Further on, let
us introduce the maps 
\begin{align*}
\overline{\Gamma}_{n}(s,t) & =\tilde{\Gamma}_{n}\circ\theta_{n}^{-1}(s,t),\ s\in[-1,1],\\
\overline{\Gamma}_{n}^{-}(s,t) & =\tilde{\Gamma}_{n}^{-}\circ(\theta_{n}^{-})^{-1}(s,t),\ s\in[-1,-1/2],\\
\overline{\Gamma}_{n}^{+}(s,t) & =\tilde{\Gamma}_{n}^{+}\circ (\theta_{n}^{+})^{-1}(s,t),\ s\in[1/2,1],\\
\overline{\Gamma}^{-}(s,t) & =\tilde{\Gamma}^{-}\circ(\theta^{-})^{-1}(s,t),\ s\in[-1,-1/2),\\
\overline{\Gamma}^{+}(s,t) & =\tilde{\Gamma}^{+}\circ(\theta^{+})^{-1}(s,t),\ s\in(1/2,1].
\end{align*}
We prove the following 
\end{singlespace}
\begin{thm}
\begin{singlespace}
\noindent \label{thm:The-maps-}For every sequence $h_{n}\in\mathbb{R}_{>0}$
satisfying $h_{n}<R_{n}$ and $h_{n},R_{n}/h_{n}\rightarrow\infty$
as $n\rightarrow\infty$, the following convergence results hold for
the maps $\tilde{\Gamma}_{n}$ and $\overline{\Gamma}_{n}$ and their
left and right shifts $\tilde{\Gamma}_{n}^{\pm}$ and $\overline{\Gamma}_{n}^{\pm}$,
respectively.

\noindent $C_{\text{loc}}^{\infty}-$convergence: 
\end{singlespace}
\begin{enumerate}
\begin{singlespace}
\item For any sequence $s_{n}\in[-R_{n}+h_{n},R_{n}-h_{n}]$ there exists
a subsequence of the sequence of shifted harmonic functions $\tilde{\Gamma}_{n}(\cdot+s_{n},\cdot)$,
also denoted by $\tilde{\Gamma}_{n}(\cdot+s_{n},\cdot)$, which is
defined on $[-R_{n}+h_{n}-s_{n},R_{n}-h_{n}-s_{n}]\times S^{1}$ and
converges in $C_{\text{loc}}^{\infty}$ to $0$. 
\item The harmonic functions $\tilde{\Gamma}_{n}^{-}:[0,h_{n}]\times S^{1}\rightarrow\mathbb{R}$
converge in $C_{\text{loc}}^{\infty}$ to a harmonic function $\tilde{\Gamma}^{-}:[0,+\infty)\times S^{1}\rightarrow\mathbb{R}$
with $\lim_{s\rightarrow \infty}\tilde{\Gamma}^-(s,\cdot)=0$. Furthermore, $\overline{\Gamma}_{n}^{-}:[-1,-1/2]\times S^{1}\rightarrow\mathbb{R}$
converge in $C_{\text{loc}}^{\infty}([-1,-1/2)\times S^{1})$ to the
map $\overline{\Gamma}^{-}:[-1,-1/2)\times S^{1}\rightarrow\mathbb{R}$ with
$\lim_{s\rightarrow -1/2}\overline{\Gamma}^-(s,\cdot)=0$ at $\{-1/2\}\times S^{1}$. 
\item The harmonic functions $\tilde{\Gamma}_{n}^{+}:[-h_{n},0]\times S^{1}\rightarrow\mathbb{R}$
converge in $C_{\text{loc}}^{\infty}$ to a harmonic function $\tilde{\Gamma}^{+}:(-\infty,0]\times S^{1}\rightarrow\mathbb{R}$
with $\lim_{s\rightarrow -\infty}\tilde{\Gamma}^+(s,\cdot)=0$. Furthermore, $\overline{\Gamma}_{n}^{+}:[1/2,1]\times S^{1}\rightarrow\mathbb{R}$
converge in $C_{\text{loc}}^{\infty}((1/2,1]\times S^{1})$ to the
map $\overline{\Gamma}^{+}:(1/2,1]\times S^{1}\rightarrow\mathbb{R}$ with
$\lim_{s\rightarrow 1/2}\overline{\Gamma}^+(s,\cdot)=0$. 
\end{singlespace}
\end{enumerate}
\begin{singlespace}
\noindent $C^{0}-$convergence: 
\end{singlespace}
\begin{enumerate}
\begin{singlespace}
\item The functions $\overline{\Gamma}_{n}$ converge in $C^{0}([-1/2,1/2]\times S^{1})$
to $0$. 
\item The functions $\overline{\Gamma}_{n}^{-}$ converge in $C^{0}([-1,-1/2]\times S^{1})$
to a function $\overline{\Gamma}^{-}:[-1,-1/2]\times S^{1}\rightarrow\mathbb{R}$
with $\overline{\Gamma}^{-}(-1/2,t)=0$ for all $t\in S^{1}$. 
\item The functions $\overline{\Gamma}_{n}^{+}$ converge in $C^{0}([1/2,1]\times S^{1})$
to a function $\overline{\Gamma}^{+}:[1/2,1]\times S^{1}\rightarrow\mathbb{R}$
with $\overline{\Gamma}^{+}(1/2,t)=0$ for all $t\in S^{1}$. 
\end{singlespace}
\end{enumerate}
\end{thm}
\begin{proof}
\begin{singlespace}
\noindent First we prove the $C_{\text{loc}}^{\infty}-$convergence
of the harmonic functions $\Gamma_{n}$.
\end{singlespace}
\begin{enumerate}
\begin{singlespace}
\item By Remark \ref{rem:From-the-proof}, for any sequence $s_{n}\in[-R_{n}+h_{n},R_{n}-h_{n}]$
the sequence of shifted harmonic maps $\Gamma_{n}(\cdot+s_{n},\cdot)$
contains a subsequence, also denoted by $\Gamma_{n}(\cdot+s_{n},\cdot)$,
which is defined on $[-R_{n}+h_{n}-s_{n},R_{n}-h_{n}-s_{n}]\times S^{1}$
and converges in $C_{\text{loc}}^{\infty}$ to $0$. 
\item Consider the shifted maps $\Gamma_{n}^{-}:[0,h_{n}]\times S^{1}\rightarrow\mathbb{R}\times M$.
By Lemma \ref{lem:For-every-}, these maps have uniformly bounded
derivatives, and hence after passing a subsequence, they
converge in $C_{\text{loc}}^{\infty}([0,\infty)\times S^{1})$ to
a harmonic function $\Gamma^{-}:[0,+\infty)\times S^{1}\rightarrow\mathbb{R}$.
In the following we show that $\Gamma^{-}$ is asymptotic to $0$,
i.e. that $\lim_{s\rightarrow\infty}\Gamma^{-}(s,\cdot)=0$ in $C^\infty(S^1,\mathbb{R})$. Since the derivatives of the harmonic function $\Gamma^-$
are uniformly bounded, it suffices to prove pointwise convergence. For this we argue
by contradiction. Because the $\Gamma_{n}$ are uniformly bounded
in $C^{0}$, we assume that there exists a sequence $(s_{k},t_{k})\in[0,\infty)\times S^{1}$
with $s_{k}\rightarrow\infty$ as $k\rightarrow\infty$ such that
$\lim_{k\rightarrow\infty}\Gamma^{-}(s_{k},t_{k})=w\in\mathbb{R}\backslash\{0\}$.
Putting $\epsilon:=|w|>0$, using Proposition \ref{prop:For-every-},
and arguing as in Theorem \ref{thm:With-the-same} we are led to the
contradiction $\epsilon=|w|\leq3\epsilon/10$. As $(\theta_{n}^{-})^{-1}:[-1,-1/2]\times S^1\rightarrow[0,h_{n}]\times S^1$
converge in $C_{\text{loc}}^{\infty}$ to the diffeomorphism $(\theta^{-})^{-1}:[-1,-1/2)\times S^1\rightarrow[0,+\infty)\times S^1$,
the maps $\Gamma_{n}^{-}\circ(\theta_{n}^{-})^{-1}(s,t)$ converge in
$C_{\text{loc}}^{\infty}$ to the map $\Gamma^{-}\circ(\theta^{-})^{-1}(s,t)$.
By the asymptotics of $\Gamma^{-}$, we have $\lim_{s\rightarrow \infty}\overline{\Gamma}^{-}(s,\cdot)=0$
as $s\rightarrow -1/2$. 
\item The proof for the maps $\Gamma_{n}^{+}$ proceeds as in Case 2. 
\end{singlespace}
\end{enumerate}
\begin{singlespace}
\noindent To prove the $C^{0}-$convergence of the harmonic functions
$\Gamma_{n}$, we prove that the functions $\overline{\Gamma}_{n}$,
$\overline{\Gamma}_{n}^{-}$ and $\overline{\Gamma}_{n}^{+}$ converge
in $C^{0}$.
\end{singlespace}
\begin{enumerate}
\begin{singlespace}
\item From Corollary \ref{cor:Thus-we-may} it follows that $\left\Vert \overline{\Gamma}_{n}\right\Vert _{C^{0}([-1/2,1/2]\times S^{1})}\rightarrow0$
as $n\rightarrow\infty$. 
\item Now consider the maps $\overline{\Gamma}_{n}^{-}(s,t)$. The $\overline{\Gamma}_{n}^{-}$
converge in $C_{\text{loc}}^{\infty}$ to $\overline{\Gamma}^{-}$
on $[-1,-1/2)\times S^{1}$. Since $\lim_{s\rightarrow \infty}\Gamma^{-}(s,\cdot)=0$ by the discussion above, the function
$\overline{\Gamma}^{-}$ can be continuously extended to the whole
cylinder $[-1,-1/2]\times S^{1}$ by setting $\overline{\Gamma}^{-}(-1/2,t)=0$.
As a matter of fact, the maps $\overline{\Gamma}_{n}^{-}$ converge
in $C^{0}([-1,-1/2])$ to $\overline{\Gamma}^{-}$. The proof of this
statement is as in Lemma 4.16 of \cite{key-10}, but for completeness we review it here: Let $\delta>0$ be given. By the $C_{\text{loc}}^{\infty}-$convergence
of the functions $\overline{\Gamma}_{n}^{-}$ to $\overline{\Gamma}^{-}$
on $[-1,-1/2)\times S^{1}$ it suffices to find $\sigma>0$ and $N\in\mathbb{N}$
such that $|\overline{\Gamma}_{n}^{-}(s,t)|\leq\delta$ for all $(s,t)\in[-(1/2)-\sigma,-1/2]\times S^{1}$
and $n\geq N$. From Proposition \ref{prop:For-every-} there exist
$N\in\mathbb{N}$ and $h>0$ such that for all $n\geq N$ and $(s,t)\in[-R_{n}+h,R_{n}-h]\times S^{1}$,
we have $|\Gamma_{n}(s,t)|\leq\delta$. Recall that $(\check{\theta}^{-})^{-1}$
maps $[-1,-1/2)$ diffeomorphically onto $[0,\infty)$. Thus we find
$\sigma>0$ such that $(\check{\theta}^{-})^{-1}(-\sigma)\geq h+1$. By the
$C_{\text{loc}}^{\infty}-$convergence, we obtain $\check{\theta}_{n}^{-1}(-\sigma)+R_{n}=(\check{\theta}_{n}^{-})^{-1}(-\sigma)\geq h$
for $n$ sufficiently large; hence, $\check{\theta}_{n}^{-1}(-\sigma)\geq-R_{n}+h$.
Therefore, by the monotonicity of $\check{\theta}_{n}$, we have $\check{\theta}_{n}^{-}([-(1/2)-\sigma,-1/2])\subset[-R_{n}+h,R_{n}-h]$
and we end up with $|\overline{\Gamma}_{n}^{-}(s,t)|=|\Gamma_{n}^{-}\circ(\theta_{n}^{-})^{-1}(s,t)|\leq\delta$. 
\item For the maps $\overline{\Gamma}_{n}^{+}$ we proceed analogously. 
\end{singlespace}
\end{enumerate}
\end{proof}
\begin{singlespace}
\noindent In the following, we establish a convergence result for
the harmonic functions $\Gamma_{n}$. For this purpose, we define
the harmonic functions $\Gamma_{n}^{-}:[0,2R_{n}]\times S^{1}\rightarrow\mathbb{R}$
and $\Gamma_{n}^{+}:[-2R_{n},0]\times S^{1}\rightarrow\mathbb{R}$
by $\Gamma_{n}^{-}(s,t):=\Gamma_{n}(s-R_{n},t)=S_{n}(s-R_{n})+\tilde{\Gamma}_{n}^{-}(s,t)$
and $\Gamma_{n}^{+}(s,t):=\Gamma_{n}(s+R_{n},t)=S_{n}(s+R_{n})+\tilde{\Gamma}_{n}^{+}(s,t)$,
respectively. Since $\tilde{\Gamma}_{n}^{-}\rightarrow\tilde{\Gamma}^{-}$
and $\tilde{\Gamma}_{n}^{+}\rightarrow\tilde{\Gamma}^{+}$ in $C_{\text{loc}}^{\infty}$,
$\Gamma_{n}^{-}+S_{n}R_{n}\rightarrow\tilde{\Gamma}^{-}$ converge
in $C_{\text{loc}}^{\infty}$ on $[0,+\infty)\times S^{1}$ and $\Gamma_{n}^{+}-S_{n}R_{n}\rightarrow\tilde{\Gamma}^{+}$
converge in $C_{\text{loc}}^{\infty}$ on $(-\infty,0]\times S^{1}$.
Moreover, by means of the homeomorphism $\theta_{n}$, we define the
maps 
\begin{align*}
\underline{\Gamma}_{n}(s,t) & =\Gamma_{n}\circ\theta_{n}^{-1}(s,t)=S_{n}\check{\theta}_{n}^{-1}(s)+\overline{\Gamma}_{n}(s,t),\ s\in[-1,1],\\
\underline{\Gamma}_{n}^{-}(s,t) & =\Gamma_{n}^{-}\circ(\theta_{n}^{-})^{-1}(s,t)=S_{n}((\check{\theta}_{n}^{-})^{-1}(s)-R_{n})+\overline{\Gamma}_{n}^{-}(s,t),\ s\in[-1,-1/2],\\
\underline{\Gamma}_{n}^{+}(s,t) & =\Gamma_{n}^{+}\circ(\theta_{n}^{+})^{-1}(s,t)=S_{n}((\check{\theta}_{n}^{+})^{-1}(s)+R_{n})+\overline{\Gamma}_{n}^{+}(s,t),\ s\in[1/2,1].
\end{align*}
Now we are in the position to derive a convergence result for the
sequence of harmonic functions $\Gamma_{n}$. 
\end{singlespace}
\begin{thm}
\begin{singlespace}
\noindent \label{thm:The-maps--1}For every sequence $h_{n}\in\mathbb{R}_{>0}$
satisfying $h_{n}<R_{n}$ and $h_{n},R_{n}/h_{n}\rightarrow\infty$
as $n\rightarrow\infty$, the following $C_{\text{loc}}^{\infty}-$convergence
results hold for the functions $\Gamma_{n}$ and $\underline{\Gamma}_{n}$
and their left and right shifts $\Gamma_{n}^{\pm}$ and $\underline{\Gamma}_{n}^{\pm}$,
respectively: 
\end{singlespace}
\begin{enumerate}
\begin{singlespace}
\item For any sequence $s_{n}\in[-R_{n}+h_{n},R_{n}-h_{n}]$ there exists
a subsequence of the sequence of shifted harmonic functions $\Gamma_{n}(\cdot+s_{n},\cdot)$,
also denoted by $\Gamma_{n}(\cdot+s_{n},\cdot)$ and defined on $[-R_{n}+h_{n}-s_{n},R_{n}-h_{n}-s_{n}]\times S^{1}$,
such that $\Gamma_{n}(\cdot+s_{n},\cdot)-S_{n}s_{n}$ converges in
$C_{\text{loc}}^{\infty}$ to $0$. 
\item The harmonic functions $\Gamma_{n}^{-}+S_{n}R_{n}:[0,h_{n}]\times S^{1}\rightarrow\mathbb{R}$
converge in $C_{\text{loc}}^{\infty}$ to a harmonic function $\tilde{\Gamma}^{-}:[0,+\infty)\times S^{1}\rightarrow\mathbb{R}$
with $\lim_{s\rightarrow \infty}\tilde{\Gamma}^{-}(s,\cdot) =0$. Furthermore, $\underline{\Gamma}_{n}^{-}+S_{n}R_{n}:[-1,-1/2]\times S^{1}\rightarrow\mathbb{R}$
converge in $C_{\text{loc}}^{\infty}([-1,-1/2)\times S^{1})$ to a
function $\overline{\Gamma}^{-}:[-1,-1/2)\times S^{1}\rightarrow\mathbb{R}$
such that 
\[
\lim_{s\rightarrow-\frac{1}{2}}\overline{\Gamma}^{-}(s,t)=0 \text{ in }C^{\infty}(S^{1}). 
\]

\item The harmonic functions $\Gamma_{n}^{+}-S_{n}R_{n}:[-h_{n},0]\times S^{1}\rightarrow\mathbb{R}$
converge in $C_{\text{loc}}^{\infty}$ to a harmonic function $\tilde{\Gamma}^{+}:(-\infty,0]\times S^{1}\rightarrow\mathbb{R}$
$\lim_{s\rightarrow -\infty}\tilde{\Gamma}^{+}(s,\cdot) =0$. Furthermore, $\underline{\Gamma}_{n}^{+}-S_{n}R_{n}:[1/2,1]\times S^{1}\rightarrow\mathbb{R}$
converge in $C_{\text{loc}}^{\infty}((1/2,1]\times S^{1})$ to a
function $\overline{\Gamma}^{+}:(1/2,1]\times S^{1}\rightarrow\mathbb{R}$
such that 
\[
\lim_{s\rightarrow\frac{1}{2}}\overline{\Gamma}^{+}(s,t)=0 \text{ in } C^{\infty}(S^{1}). 
\]

\end{singlespace}
\end{enumerate}
\end{thm}
\begin{proof}
\begin{singlespace}
\noindent For $(s,t)\in[-R_{n}+h_{n}-s_{n},R_{n}-h_{n}-s_{n}]\times S^{1}$,
we have $\Gamma_{n}(s+s_{n},t)-S_{n}s_{n}=S_{n}s+\tilde{\Gamma}_{n}(s+s_{n},t)$.
By Theorem \ref{thm:The-maps-}, the first assertion readily follows.
Putting $\Gamma_{n}^{-}(s,t)-S_{n}R_{n}=S_{n}s+\tilde{\Gamma}_{n}^{-}(s,t)$,
using the fact that $\underline{\Gamma}_{n}^{-}(s,t)+S_{n}R_{n}=S_{n}(\theta_{n}^{-})^{-1}(s)+\overline{\Gamma}_{n}^{-}(s,t)$
converge in $C_{\text{loc}}^{\infty}$ to $\overline{\Gamma}^{-}:[-1,-1/2)\times S^{1}\rightarrow\mathbb{R}$
since $|S_{n}(\theta_{n}^{-})^{-1}(s)|\leq |S_n h_n|\leq h_n/R_n\rightarrow 0$. Now $\lim_{s\rightarrow \infty}\overline{\Gamma}^{-}(s,\cdot)=0$ as $s\rightarrow-1/2$, and applying Theorem
\ref{thm:The-maps-} finishes the proof of the second assertion.\textbf{
}The third assertion is proved in a similar manner. 
\end{singlespace}
\end{proof}
\begin{singlespace}
\noindent To derive a notion of $C^{0}$ convergence we assume that
the sequence $S_{n}R_{n}$ converges, i.e. $S_{n}R_{n}\rightarrow\sigma$
as $n\rightarrow\infty$. Note that this always holds
after passing to a suitable subsequence by assumption C5. 
\end{singlespace}
\begin{thm}
\begin{singlespace}
\noindent \label{thm:The-maps--2}For every sequence $h_{n}\in\mathbb{R}_{>0}$
satisfying $h_{n}<R_{n}$ and $h_{n},R_{n}/h_{n}\rightarrow\infty$
as $n\rightarrow\infty$, the following $C^{0}-$convergence results
hold for the maps $\underline{\Gamma}_{n}$ together with their left
and right shift $\underline{\Gamma}_{n}^{\pm}$: 
\end{singlespace}
\begin{enumerate}
\begin{singlespace}
\item There exists a subsequence of $\underline{\Gamma}_{n}$ that converges
in $C^{0}([-1/2,1/2]\times S^{1})$ to $2\sigma s$. 
\item There exists a subsequence of $\underline{\Gamma}_{n}^{-}$ that converges
in $C^{0}([-1,-1/2]\times S^{1})$ to $\overline{\Gamma}^{-}-\sigma$,
where $\overline{\Gamma}^{-}(-1/2,t)=0$ for all $t\in S^{1}$. 
\item There exists a subsequence of $\underline{\Gamma}_{n}^{+}$ that converges
in $C^{0}([1/2,1]\times S^{1})$ to $\overline{\Gamma}^{+}+\sigma$,
where $\overline{\Gamma}^{+}(+1/2,t)=0$ for all $t\in S^{1}$. 
\end{singlespace}
\end{enumerate}
\end{thm}
\begin{proof}
\begin{singlespace}
\noindent We consider $\underline{\Gamma}_{n}(s,t)=S_{n}\check{\theta}_{n}^{-1}(s)+\overline{\Gamma}_{n}(s,t)$
for $(s,t)\in[-1/2,1/2]\times S^{1}$ with 
\begin{align*}
S_{n}\check{\theta}_{n}^{-1}(s)=2(S_{n}R_{n}-S_{n}h_{n})s.
\end{align*}
Corollary \ref{cor:Thus-we-may} implies that $S_{n}h_{n}s$ converges
in $C^{0}([-1/2,1/2]\times S^{1})$ to $0$, and similarly, that $S_{n}R_{n}s$
converges in $C^{0}([-1/2,1/2]\times S^{1})$ to $2\sigma s$. By
Theorem \ref{thm:The-maps--1}, $\overline{\Gamma}_{n}$ converges
in $C^{0}([-1/2,1/2]\times S^{1})$ to $0$, and so, the first assertion
is proved. Setting $\underline{\Gamma}_{n}^{-}(s,t)=S_{n}(\check{\theta}_{n}^{-})^{-1}(s)-S_{n}R_{n}+\overline{\Gamma}_{n}^{-}(s,t)$
for $(s,t)\in[-1,-1/2]\times S^{1}$, taking into account that $S_{n}(\check{\theta}_{n}^{-})^{-1}(s)$
converges in $C^{0}([-1,-1/2]\times S^{1})$ to $0$, and applying
Theorem \ref{thm:The-maps--1} proves the second assertion. The third
assertion follows in an analogous manner. 
\end{singlespace}
\end{proof}
\begin{singlespace}
\noindent Finally, we establish a convergence result for the differential
of $\Gamma_{n}$. Due to Lemma \ref{lem:For-the-harmonic}, we have
$d\Gamma_{n}^{-}=S_{n}ds+d\tilde{\Gamma}_{n}^{-}$ on $[0,h_{n}]\times S^{1}$
and $d\Gamma_{n}^{+}=S_{n}ds+d\tilde{\Gamma}_{n}^{+}$ on $[-h_{n},0]\times S^{1}$. For a sequence $h_{n}\in\mathbb{R}_{>0}$ satisfying $h_{n}<R_{n}$
and $h_{n},R_{n}/h_{n}\rightarrow\infty$ as $n\rightarrow\infty$,
consider the sequence of diffeomorphisms $\check{\theta}_{n}:[-R_{n},R_{n}]\rightarrow[-1,1]$
as in Remark \ref{rem:For-every-sequence}. In terms of $\check{\theta}_{n}$
we obtain the equations $d\underline{\Gamma}_{n}^{-}=S_{n}[(\check{\theta}_{n}^{-})^{-1}]'(s)ds+d\overline{\Gamma}_{n}^{-}$
on $[-1,-1/2]\times S^{1}$ and $d\underline{\Gamma}_{n}^{+}=S_{n}[(\check{\theta}_{n}^{+})^{-1}]'(s)ds+d\overline{\Gamma}_{n}^{+}$
on $[1/2,1]\times S^{1}$. Since by construction $[(\check{\theta}_{n}^{-})^{-1}]'(s)=2(R_n-h_n)$ for $s$ near $h_n$, we find
the following 
\end{singlespace}
\begin{cor}
\begin{singlespace}
\noindent \label{cor:After-going-over}After passing to a subsequence,
the following $C_{\text{loc}}^{\infty}-$convergence results for the
1-forms $d\Gamma_{n}^{-}$, $d\Gamma_{n}^{+}$, $d\underline{\Gamma}_{n}^{-}$
and $d\underline{\Gamma}_{n}^{+}$ hold: 
\end{singlespace}
\begin{enumerate}
\begin{singlespace}
\item The harmonic $1-$forms $d\Gamma_{n}^{-}$ converge in $C_{\text{loc}}^{\infty}([0,+\infty)\times S^{1})$
to a harmonic $1-$form $d\tilde{\Gamma}^{-}$ on $[0,+\infty)\times S^{1}$,
with $\lim_{s\rightarrow \infty}d\tilde{\Gamma}^-(s,\cdot)=0$. The $1-$forms $d\underline{\Gamma}_{n}^{-}$
converge in $C_{\text{loc}}^{\infty}([0,1/2)\times S^{1})$ to a $1-$form
$d\overline{\Gamma}^{-}$.
\item The harmonic $1-$forms $d\Gamma_{n}^{+}$ converge in $C_{\text{loc}}^{\infty}((-\infty,0]\times S^{1})$
to a harmonic $1-$form $d\tilde{\Gamma}^{+}$ on $(-\infty,0]\times S^{1}$,
with $\lim_{s\rightarrow -\infty}d\tilde{\Gamma}^+(s,\cdot)=0$. The $1-$forms $d\underline{\Gamma}_{n}^{+}$
converge in $C_{\text{loc}}^{\infty}((-1/2,0]\times S^{1})$ to a
$1-$form $d\overline{\Gamma}^{+}$.
\end{singlespace}
\end{enumerate}
\end{cor}
\newpage
\section{\label{subsec:A-version-of}A version of the Monotonicity Lemma}

\noindent In this appendix we describe how a monotonicity for
the transformed curves $\overline{u}$ as in Definition \ref{def:A-triple-} can be obtained. Such a $\overline{u}$ is
a holomorphic curves for a domain-dependent almost complex structure $\overline{J}_P$ on $\mathbb{R}\times M$. 
Recall from (\ref{eq:domain_dependent_almost_complex_structure}) the compact family $\mathcal{J}$ of $\mathbb{R}$-invariant almost complex structures
$J_\rho, \rho\in [-C,C]$ on $\mathbb{R}\times M$ which are compatible with $d\alpha$ on $\xi$ and exchange the Reeb and the $\mathbb{R}$-direction.\\
This implies that each $J_\rho$ is compatible to the symplectic form $\omega=d(e^r\alpha)$ on $\mathbb{R}\times S^1$.

\noindent If $u:(S,j)\rightarrow \mathbb{R}\times M$ is a smooth map defined on a Riemann surface $(S,j)$, we say it is (domain-dependent) \emph{$\mathcal{J}$-holomorphic},
if for each $x\in S $ there is $\rho(x)\in [-C,C]$ with $Du(x)\circ j=J_{\rho(x)} \circ Du(x)$.\\
Clearly, for any $\overline{J}_P$-holomorphic curve $(\overline{u},R,P)$, the map $\overline{u}$ is $\mathcal{J}$-holomorphic.

\begin{lem} \label{lem:mono_1}
There exist constants $\varepsilon_0>0$ and $c_1>0$ with the following properties:
Let $u:S\rightarrow \mathbb{R}\times M$ be any nonconstant $\mathcal{J}$-holomorphic curve defined on a compact Riemann surface $(S,\partial S)$.
If $u$ passes through $p\in \mathbb{R}\times M$ but $u(\partial S)$ lies in the complement of $B^{\overline{g}_0}_\varepsilon(p)$ for some $\varepsilon\leq \varepsilon_0$, 
then $\text{area}_{\overline{g}_0}(B^{\overline{g}_0}_\varepsilon(p)\cap u(S))\geq c_1\varepsilon^2$.
\end{lem}

\noindent On any bounded strip in $\mathbb{R}\times M$ the $\omega$-area of $\mathcal{J}$-holomorphic curves can be bounded by the energy $E(u;S)$:

\begin{lem}\label{lem:monolo}
Let $u$ be any $\mathcal{J}$-holomorphic curve $u:S\rightarrow (0,1)\times M$. Then $\int_S u^*\omega\leq 6 E(u;S)$.
\end{lem}
\begin{proof}
 Notice that $\int_S u^*(e^r d\alpha) 
\leq  6\int_S u^*d\alpha=E_{d\alpha}(u;S)$ and  $\int_Su^*(e^r dr\wedge \alpha)\leq  6\int_Su^*(1/2 dr\wedge \alpha)\leq 6 E_\alpha(u;S)$, the latter
since there is a $\varphi:\mathbb{R}\rightarrow [0,1]$ with $\varphi'(r)\geq\frac{1}{2}$ for all  $r\in [0,1]$, which is admissible in the definition of the $\alpha$-energy.
In both cases the estimates hold pointwise in the integrands, since the implicit
almost complex structures $J_\rho$ are compatible with $d\alpha $ on $\xi$ and since they exchange Reeb and $\mathbb{R}$-direction. 
Combining these two inequalities yields the claimed result.\end{proof}

\noindent Recall $\overline{C}_1$ from (\ref{eq:equivalence_metrics}) satisfying $\overline{C}_1\overline{g}_\rho\geq \overline{g}_0$. 
As a consequence, we have for the induced area-form $d area^P\overline{g}_0$ on any 2-plane $P$ the estimate
$|d area^P_{\overline{g}_0}|\leq \overline{C}_1|d area^P_{\overline{g}_\rho}|$.
If the 2-plane lies in some $T_p((0,1)\times M)$ and it is is $J_\rho$-complex for some $\rho$, then we have moreover $|d area_{\overline{g}_\rho}^P|\leq \omega\vert_P$ (by similar pointwise estimates as in the proof of Lemma \ref{lem:monolo}). \\
It follows that $area_{\overline{g}_0}(u(S))\leq \overline{C}_1\int_S u^*\omega$
for any $\mathcal{J}$-holomorphic curve $u:S\rightarrow (0,1)\times M$.
As a direct consequence of this discussion and of Lemma \ref{lem:mono_1} we have:
\begin{lem} \label{cor:There-exists-constants-2}
Let $\varepsilon_0,c_1 $ be constants as in Lemma \ref{lem:mono_1} and set $C_8:=c_1\overline{C}_1^{-1}/6$.
Let $u:S\rightarrow \mathbb{R}\times M$ be any nonconstant $\mathcal{J}$-holomorphic curve defined on a compact Riemann surface $(S,\partial S)$.
If $u$ passes through $p\in \mathbb{R}\times M$ but $u(\partial S)$ lies in the complement of $B^{\overline{g}_0}_\varepsilon(p)$ for some $\varepsilon\leq \varepsilon_0$, 
then $E(u;S\cap u^{-1}(B^{\overline{g}_0}_\varepsilon(p)))\geq C_8\varepsilon^2$.
\end{lem}

\emph{On the proof of Lemma \ref{lem:mono_1}}: Note that on $(0,1)\times M$ we have $\omega(v,J_{\rho}v)\geq \overline{g}_\rho(v,v)\geq \overline{C_1}^{-1} \overline{g}_0(v,v)$. 
Since the almost complex structures $\mathcal{J}_\rho$ are $\mathbb{R}$-invariant, we
can assume without loss of generality that the point $p$ of interest lies in $\{\frac{1}{2}\}\times M$; w.l.og. furthermore $\varepsilon_0\leq \frac{1}{4}$.
We claim also that in our situation, we can establish a quadratic isoperimetric inequality as in Lemma 3.1 in \cite{key-5}.
This Lemma is easily adapted to the present situation, by considering the following:
The complex structures we are given are domain-dependent and thus more general than those in \cite{key-5}, however the almost complex structure considered here are \emph{globally}
uniformly tamed by the global symplectic form $\omega$ on $(0,1)\times M$; thus the symplectic forms do not have to be constructed by hand as in Lemma 3.1 in \cite{key-5}.
Indeed, assume that $u:S\rightarrow B$ is a $\mathcal{J}$-holomorphic map into a coordinate ball $B$ and let $v:S\rightarrow B$ be a smooth map with $u\vert_{\partial S}=v\vert_{\partial S}$
and $4\pi area_{\overline{g}_0}(v(S))\leq c^4length_{\overline{g}_0}(v\vert_{\partial S})^2$ as on p.25 of \cite{key-5} (where the maps are called $f$ resp. $g$). 
This is possible, since the exponential maps $\exp_q^{\overline{g}_0}$ are c-bilipschitz onto balls $B=B_{\varepsilon_1}$ of radius $\varepsilon_1$.
Then a quadratic isoperimetric inequality for $u$ is obtained as follows:
$$area_{\overline{g}_0}(u)\leq \overline{C}_1\int_S u^*\omega=\overline{C}_1\int_S v^*\omega\leq 6\overline{C}_1 area_{\overline{g}_0}(v)\leq 
6\overline{C}_1 c^4length_{\overline{g}_0}(v\vert_{\partial S})^2=6\overline{C}_1 c^4length_{\overline{g}_0}(u\vert_{\partial S})^2.$$
Now Lemma \ref{lem:mono_1} follows from this quadratic isoperimetric inequality
exactly as Lemma 1.3 in \cite{key-5} follows from Lemma 3.1 on p.26-28 in \cite{key-5}.\quad \qed

\section{\label{subsec:A-version-of-example}An example showing the necessity of assumption A3}

Consider $S^{3}=\left\{ (z_{1},z_{2})\subset\mathbb{C}^{2}\mid\left|z_{1}\right|^{2}+\left|z_{2}\right|^{2}=1\right\} \subset\mathbb{C}^{2}$
equipped with the standard contact form. More precisely, on $\mathbb{C}^{2}$
we consider the Liouville $1-$form which is defined in complex coordinates
by $\lambda=-\frac{1}{2}\text{Im}\left(\overline{z}_{1}dz_{1}+\overline{z}_{2}dz_{2}\right)$.
Then $\alpha:=\lambda|_{S^{3}}$ defines a contact form on $S^{3}$,
which we will call the standard contact form. The Reeb flow of the
Reeb vector field of $\alpha$ is given by $\varphi_{t}^{\alpha}(z_{1},z_{2})=\left(e^{it}z_{1},e^{it}z_{2}\right)$
for all $(z_{1},z_{2})\in S^{3}$ and $t\in\mathbb{R}$. It becomes
evident that all Reeb orbits are periodic. Denote by $p=(z_{1},z_{2})\in S^{3}$.
In order to choose a $\mathbb{R}-$invariant almost complex structure
on the symplectization $\mathbb{R}\times S^{3}$ we consider the diffeomorphism
\begin{align*}
\Phi:\mathbb{R}\times S^{3} & \rightarrow\mathbb{C}^{2}\backslash\left\{ 0\right\} \\
(r,p) & \mapsto e^{r}p.
\end{align*}
Let $i$ be the standard complex structure on $\mathbb{C}^{2}\backslash\left\{ 0\right\} $.
Then one easily sees that 
\[
\tilde{J}(r,p):=d\Phi(r,p)^{-1}\circ i\circ d\Phi(r,p)
\]
defines an $\mathbb{R}-$invariant almost complex structure on $\mathbb{R}\times S^{3}$.
We will denote by $J$ the induced almost complex structure on the
contact structure $\xi=\ker(\alpha)$. Notice that $J$ is invariant
under the Reeb flow, i.e. $J(\varphi_{t}^{\alpha}(p))\circ d\varphi_{t}^{\alpha}(p)=d\varphi_{t}^{\alpha}(p)\circ J(p)$
for all $p\in S^{3}$ and all $t\in\mathbb{R}$.

Let $u=(a,f):D\rightarrow\mathbb{R}\times S^{3}$ be a pseudoholomorphic
disk, where $D$ is the unit disk in $\mathbb{C}$, with respect to
the standard complex structure $i$ on $D$ and $\tilde{J}$ on $\mathbb{R}\times S^{3}$
and with energies 
\begin{align*}
E_{d\alpha}(u;D) & =\int_{D}f^{*}d\alpha,\\
E_{\alpha}(u;D) & =\sup_{\varphi\in\mathcal{A}}\int_{D}\varphi'(a)da\circ i\wedge da.
\end{align*}

Let $R_{n}\in\mathbb{R}_{>0}$ be a sequence such that $R_{n}\nearrow\infty$
as $n\rightarrow\infty$ and consider the biholomorphism
\begin{align*}
\phi_{n}:\left(\left[-R_{n},R_{n}\right]\times S^{1},i\right) & \rightarrow\left(D(\epsilon_{n})\backslash\text{Int}\left(D(\epsilon_{n}^{3})\right),i\right)\\
(s,t) & \mapsto e^{-4\pi R_{n}}e^{-2\pi(s+it)}
\end{align*}
where $\epsilon_{n}:=e^{-2\pi R_{n}}$. Now  consider the sequence
of pseudoholomorphic cylinders
\[
u_{n}=(a_{n},f_{n}):=u\circ\phi_{n}:\left[-R_{n},R_{n}\right]\times S^{1}\rightarrow\mathbb{R}\times S^{1}
\]
with energies
\begin{align*}
E_{d\alpha}(u_{n};\left[-R_{n},R_{n}\right]\times S^{1}) & \rightarrow0\text{ as }n\rightarrow\infty\quad\text{ and }\quad E_{\alpha}(u_{n};\left[-R_{n},R_{n}\right]\times S^{1})\leq E_{\alpha}(u;D).
\end{align*}
Notice that $u_{n}$ have vanishing center action. Indeed,
\[
\int_{\left\{ 0\right\} \times S^{1}}f_{n}^{*}\alpha=\int_{\partial D(\epsilon_{n}^{2})}f^{*}\alpha=\int_{D(\epsilon_{n}^{2})}f^{*}d\alpha\rightarrow0
\]
as $n\rightarrow\infty$. In the following we will describe two examples
of $\mathcal{H}-$holomorphic cylinders (one example with vanishing
center action and the second with non-vanishing center action) which
satisfy conditions A0-A2 but violate condition A3.

\noindent\textbf{An example of $\mathcal{H}-$holomorphic cylinders satisfying
condition A0-A2 but not condition A3 and having vanishing center action.}

\noindent Let $P_{n}\in\mathbb{R}$ be a sequence such that $P_{n}\searrow0$
and $P_{n}R_{n}\rightarrow\infty$ as $n\rightarrow\infty$ and there
exists a constant $C>0$ such that $P_{n}^{2}R_{n}\leq C$. For example
one could choose $P_{n}=1/\sqrt{R_{n}}$. Consider the map
\[
\overline{f}_{n}:\left[-R_{n},R_{n}\right]\times S^{1}\rightarrow S^{3}
\]
defined by $\overline{f}_{n}(s,t)=\phi_{P_{n}s}^{\alpha}(f_{n}(s,t))$.
Then $\overline{u}_{n}=(a_{n},\overline{f}_{n})$ is a $\mathcal{H}-$holomorphic
cylinder with harmonic perturbation $\gamma_{n}:=-P_{n}dt$ satisfying
conditions A0-A2 but not A3 and having vanishing center action. Notice
that for any sequence of shifts $s_{n}\in[-R_{n},R_{n}]$ the map
$\overline{u}_{n}(s+s_{n},t)$ converges, up to passing to a subsequence
in $C_{\text{loc}}^{\infty}$ to the constant map $\varphi_{\kappa}^{\alpha}(u(0))$,
where $\kappa$ is defined as the limit of $P_{n}s_{n}$ in $\mathbb{R}/2\pi\mathbb{Z}$.
Furthermore, to show that $\overline{u}_{n}$ is $\mathcal{H}-$holomorphic
we have
\begin{align*}
J\circ\pi_{\alpha}d\overline{f}_{n} & =J\circ d\phi_{P_{n}s}^{\alpha}\circ\pi_{\alpha}df_{n}\\
 & =d\phi_{P_{n}s}^{\alpha}\circ J\circ\pi_{\alpha}df_{n}\\
 & =d\phi_{P_{n}s}^{\alpha}\circ\pi_{\alpha}df_{n}\circ i\\
 & =\pi_{\alpha}d\overline{f}_{n}\circ i,
\end{align*}
and
\begin{align*}
\overline{f}_{n}^{*}\alpha\circ i & =-P_{n}dt+f_{n}^{*}\alpha\circ i\\
 & =\gamma_{n}+da_{n}.
\end{align*}
The energy of $\overline{u}_{n}$ as well as its center action is
the same as the energy and center action of $u_{n}$. Furthermore,
since $P_{n}^{2}R_{n}\leq C$ we have that the $L^{2}-$norm of $\gamma_{n}$
is uniformly bounded. Direct computation shows that $d\phi_{n}$ is
uniformly bounded and therefore condition A1 holds. Thus conditions
A0-A2 are satisfied for $n$ big enough. Condition A3 does not hold
since $P_{n}R_{n}\rightarrow\infty$ as $n\rightarrow\infty$. However,
the conclusions of the $C^{0}-$convergence of Theorem 2 do not hold
for this sequence since the $v_{n}$ defined as in Theorem 2 do not
converge since they have unbounded gradient.\\

\noindent\textbf{An example of $\mathcal{H}-$holomorphic cylinders satisfying
condition A0-A2 but not condition A3 and having non-vanishing center
action.}

\noindent Consider the sequence $P_{n}$ from the first case and let
$S\in2\pi\mathbb{Z}\backslash\{0\}$. Consider the maps
\begin{align*}
\overline{f}_{n} & :[-R_{n},R_{n}]\times S^{1}\rightarrow S^{3}\\
\overline{a}_{n} & :[-R_{n},R_{n}]\times S^{1}\rightarrow\mathbb{R}
\end{align*}
defined by $\overline{f}_{n}(s,t):=\phi_{P_{n}s+St}^{\alpha}(f_{n}(s,t))$
and $\overline{a}_{n}(s,t)=a_{n}(s,t)+Ss$. In the same way as in
the first case it is apparent that $\overline{u}_{n}=(\overline{a}_{n},\overline{f}_{n})$
is a $\mathcal{H}-$holomorphic cylinder with harmonic perturbation
$\gamma_{n}:=-P_{n}dt$. Similar as in the first case we notice that
for any sequence of shifts $s_{n}\in[-R_{n},R_{n}]$ the map $\overline{f}_{n}(s+s_{n},t)$
converges, up to passing to a subsequence in $C_{\text{loc}}^{\infty}$
to $t\mapsto\varphi_{\kappa+St}^{\alpha}(f(0))$, where $\kappa$
is defined as the limit of $P_{n}s_{n}$ in $\mathbb{R}/2\pi\mathbb{Z}$.
The $d\alpha-$energy of $\overline{u}_{n}$ is the same as the $d\alpha-$energy
of $u_{n}$. In the same way as in the first case condition A1 is
satisfied. The $\alpha-$energy can be bounded as follows. For $\varphi\in\mathcal{A}$
we have
\begin{align*}
\int_{[-R_{n},R_{n}]\times S^{1}}\varphi'(\overline{a}_{n})d\overline{a}_{n}\circ i\wedge da_{n} & =-\int_{[-R_{n},R_{n}]\times S^{1}}\varphi'(\overline{a}_{n})da_{n}\wedge d\overline{a}_{n}\circ i\\
 & =-\int_{[-R_{n},R_{n}]\times S^{1}}d\left(\varphi(\overline{a}_{n})d\overline{a}_{n}\circ i\right)+\int_{[-R_{n},R_{n}]\times S^{1}}\varphi(\overline{a}_{n})d\left(d\overline{a}_{n}\circ i\right)\\
 & \leq\int_{\left\{ R_{n}\right\} \times S^{1}}\left|d\overline{a}_{n}\circ i\right|+\int_{\left\{ -R_{n}\right\} \times S^{1}}\left|d\overline{a}_{n}\circ i\right|\\
 & =\int_{\left\{ R_{n}\right\} \times S^{1}}\left|\overline{f}_{n}^{*}\alpha\right|+\int_{\left\{ -R_{n}\right\} \times S^{1}}\left|\overline{f}_{n}^{*}\alpha\right|.
\end{align*}
Let $C_{1}>0$ be the bound from condition A1 then we have that $E_{\alpha}(\overline{u}_{n};\left[-R_{n},R_{n}\right]\times S^{1})\leq2C_{1}$.
Thus conditions A0-A2 hold. Condition A3 does not hold by the choice
of $P_{n}$. However, the conclusions of the $C^{0}-$convergence
of Theorem 4 do not hold for this sequence since the $f_{n}\circ \theta_n^{-1}$ defined
as in Theorem 4 do not converge since they have unbounded gradient.
Finally we point out that the center action is non-vanishing. Indeed,
\[
\int_{\left\{ 0\right\} \times S^{1}}\overline{f}_{n}^{*}\alpha=S+\int_{\left\{ 0\right\} \times S^{1}}f_{n}^{*}\alpha\rightarrow S\text{ as }n\rightarrow\infty.
\]

\end{document}